\documentclass[10pt]{article}
\usepackage{amsmath,amssymb,amsthm}
\usepackage[dvipsnames]{xcolor}
\usepackage[nohead,margin=1.0in]{geometry}
\usepackage{booktabs}
\usepackage{graphicx,caption,subcaption}
\usepackage{url}
\usepackage{mathtools}
\usepackage{amsmath}
\usepackage{amssymb}
\usepackage{commath}
\usepackage[normalem]{ulem}
\usepackage{tikz}
\usetikzlibrary{decorations.pathreplacing}
\usepackage{pgfplots}
\usetikzlibrary{pgfplots.groupplots}
\usepackage{multicol}
\usepackage{multirow}
\usepackage{placeins}

\usepackage{nicefrac}

\usepackage{hyperref}
\hypersetup{
    colorlinks=true,
    linkcolor=blue,
    filecolor=magenta,      
    urlcolor=cyan,
    pdftitle={Bayesian view on iResNets for linear inverse problems},
    pdfpagemode=FullScreen,
    }

\theoremstyle{plain}
\newtheorem{theorem}{Theorem}[section]
\newtheorem{remark}{Remark}[section]

\newtheorem{lemma}{Lemma}[section]
\newtheorem{assumption}{Assumption}[section]

\newtheorem{corollary}{Corollary}[section]

\numberwithin{equation}{section}

\newcommand{\N}{\mathbb{N}}
\newcommand{\R}{\mathbb{R}}

\newcommand{\Id}{\mathrm{Id}}

\title{Bayesian view on the training of invertible residual networks for solving linear inverse problems}
\author{Clemens Arndt\thanks{ Center for Industrial Mathematics, University of Bremen, 28359 Bremen, Germany. Alphabetical author order. Emails: \texttt{\{carndt, heilenkoetter, iskem, judith.nickel\}@uni-bremen.de, \{sdittmer, tkluth\}@math.uni-bremen.de}.} \and S\"oren Dittmer\footnotemark[1] \footnotemark[2]\thanks{Cambridge Image Analysis Group, University of Cambridge, Cambridge CB3 0WA, UK} \and Nick Heilenk\"otter\footnotemark[1] \and Meira Iske\footnotemark[1] \and Tobias Kluth\footnotemark[1] \and Judith Nickel\footnotemark[1]}
\date{\today}

\begin{document}

\maketitle
\begin{abstract}
Learning-based methods for inverse problems, adapting to the data's inherent structure, have become ubiquitous in the last decade. Besides empirical investigations of their often remarkable performance, an increasing number of works addresses the issue of theoretical guarantees. Recently, \cite{iresnet_01_regtheory} exploited invertible residual networks (iResNets) to learn provably convergent regularizations given reasonable assumptions. They enforced these guarantees by approximating the linear forward operator with an iResNet. Supervised training on relevant samples introduces data dependency into the approach. An open question in this context is to which extent the data's inherent structure influences the training outcome, i.e., the learned reconstruction scheme. Here we address this delicate interplay of training design and data dependency from a Bayesian perspective and shed light on opportunities and limitations. We resolve these limitations by analyzing reconstruction-based training of the inverses of iResNets, where we show that this optimization strategy introduces a level of data-dependency that cannot be achieved by approximation training. We further provide and discuss a series of numerical experiments underpinning and extending the theoretical findings.

\end{abstract}

\section{Introduction}
\label{sec:introduction}


The mathematical field of inverse problems has many applications, e.g., imaging, image processing, and several more. 
Inverse problems come with characteristic difficulties summarized under the term ``ill-posedness.'' Typically, one wants to recover causes $x$, which discontinuously depend on some observed measurements $z$. However, a good reconstruction algorithm needs to be stable; otherwise, it cannot handle noisy measurement data. Still, one naturally wants the reconstructor to be as accurate as possible. This results in a compromise called regularization (see \cite{benning2018modern} for a recent survey on regularization theory). The more stable the reconstructor becomes, the more the set of causes for which it provides accurate results is restricted.

Hence, this set of accurate performance is critical, and one typically chooses it using \textit{prior knowledge} about the application-specific data. In imaging problems, this knowledge often amounts to solutions $x$ looking somehow ``natural.'' However, the mathematical characterization of natural images is challenging. Thus, learned methods often outperform in this area, learning stable and accurate reconstructions from given training data (see, e.g., the early survey \cite{arridge2019solving}).

While many experimental studies confirm the impressive performance of learned methods, the theoretical understanding remains limited. In particular, learned methods often lack stability guarantees. However, the topic is gaining in importance \cite{mukherjee2023learned}. In the present work, we address this issue by studying invertible residual networks (iResNets) \cite{behrmann2019invertible}. As proposed in \cite{iresnet_01_regtheory}, their invertibility makes them readily applicable to linear inverse problems. \cite{iresnet_01_regtheory} approximates the forward operation ($x \mapsto z$) using the iResNet, the iResNet's inverse, then solves the inverse problem ($z \mapsto x$). Here, one can control the regularization strength by choosing a hyperparameter of the iResNet that directly controls its stability.

The authors of \cite{iresnet_01_regtheory} also develop a regularization theory for these iResNets. For this purpose, they considered particular architectures and uncovered equivalences to filter functions from classical regularization theory. In the present article we now analyze what iResNets actually learn in practice from the given training data. For this purpose, the Bayesian view is suitable, as it encodes prior knowledge on $x$ and the measurement noise in $z$ as probability distributions. We consider two different ways of training, via the forward and via the inverse mapping, and investigate to which extent the iResNet uses the given information about the data to regularize inverse problems.

The manuscript is structured as follows:
Section \ref{sec:problem_setting} introduces the problem setting, basic assumptions, and the reconstruction approach using iResNets. The subsequent two sections contain the theoretical analysis of the iResNet's training in a Bayesian setting. First, Section \ref{sec:approx_train_Bayesian} considers the so-called approximation training, where the network is trained supervisedly to approximate the forward operator. In particular, we investigate what information the network learns from the training data distribution (i.e., the effect of prior distribution and noise on the network). Second, Section \ref{sec:reco_training} considers the so-called reconstruction training, where the iResNet's inverse is trained to map from noisy measurements to a reconstruction. Section \ref{sec:numeric_experiments} complements the theoretical analyses with extensive numerical experiments. We implement the two training types and underpin the theoretical findings of the previous sections.



\subsection*{Related literature}

The Bayesian theory for inverse problems differs from the functional analytic regularization theory. While the functional analytic theory focuses mainly on the stability and convergence of reconstruction algorithms, the Bayesian perspective considers the full posterior distribution and uncertainty estimation for reconstructions. Detailed introductions are given in \cite{kaipio2006statistical, dashti_2017, stuart_2010}. An overview of the basic theory and Bayesian methods for solving inverse problems is also contained in \cite{arridge2019solving}. Learning-based methods are particularly powerful for solving Bayesian inverse problems, e.g., \cite{adler2018deep} describes two different general concepts for an efficient application of neural networks in this framework.

\cite{Laumont2022} proposes a method that demonstrates the Bayesian theory's advantages for inverse problems using a trained denoiser in a plug-and-play Langevin algorithm. The denoiser is assumed to fulfill a Lipschitz condition (similar to the iResNet, see Section~\ref{sec:problem_setting}), implying guaranteed convergence of the algorithm to the posterior distribution. \cite{sherry2023} leverages convex analysis to create nonexpansive residual networks and uses them to solve inverse problems. This is particularly desirable for denoising and plug-and-play schemes. Furthermore, invertible neural networks are also of interest to generative modeling. In \cite{chen_behrmann_2019}, iResNets act as normalizing flows, i.e., learn to map from a base distribution to a target distribution and perform competitive or even superior to alternative architectures. \cite{iresnet_01_regtheory} provides a more detailed discussion of the literature concerning learned convergent regularization schemes.

\section{Problem setting and basic properties}
\label{sec:problem_setting}

We consider linear inverse problems based on the operator equation
\begin{equation} \label{eq:inverse_problem}
    Ax = z
\end{equation}
where $A \in L(X,X)$ is a self-adjoint and positive semidefinite operator and $X$ is a finite-dimensional inner product space, here $X = \R^n$. For simplification, we assume $\|A\|=1$, which a scaling of the operator can easily obtain. In practice, neural network domains tend to be finite-dimensional; this justifies the restriction to the finite-dimensional case.
Also, the Bayesian perspective becomes less standard if the underlying probability theory uses infinite-dimensional probability spaces. 
Due to the properties of $A$, there exist eigenvalues $\sigma_j^2 \in (0,1]$ and corresponding eigenvectors $v_j$, such that $\mathcal{N}(A)^\bot = \mathrm{span} \{v_j \, | \, j=1, ..., \tilde{n} \}$, $\tilde{n} \leq n$. We use this eigendecomposition in some of our theoretical analyses. 

The aim is to recover the unknown ground truth vector $x^\dagger$ as well as possible by only having access to a noisy observation $z^\delta = Ax^\dagger + \eta$. The noise $\eta$ is assumed to be distributed according to a probability density function (pdf) $p_H \colon X \to \R_{\geq 0}$. Since there may exist arbitrarily many solutions $x$ which could explain the data $z^\delta$, it is important to incorporate prior knowledge about the unknown solutions. The pdf $p_X \colon X \to \R_{\geq 0}$ encodes this knowledge. In practice, $p_X$ may describe the distribution of natural-looking images or the typical structure of a cross-section of the human body (e.g., in CT problems).

\begin{remark} \label{rem:general_inverse_problem}
    We can translate a more general linear inverse problem
    \begin{equation}
        \tilde{A} x = y
    \end{equation}
    with an arbitrary linear operator, $\tilde{A} \in L(X, Y)$ ($X$ and $Y$ being different spaces), into the above setting by considering $A = \tilde{A}^* \tilde{A}$ and $z = \tilde{A}^* y$.

    In this case the noise $\eta$ on $z$ may arise from noise $\tilde{\eta}$ on $y$ via $\eta = \tilde{A}^* \tilde{\eta}$. To illustrate this, let us consider the example of Gaussian noise $\tilde{\eta} \sim \mathcal{N}(0, \Sigma)$. Then, it holds $\eta = \tilde{A}^* \tilde{\eta} \sim \mathcal{N}(0, \tilde{A}^* \Sigma \tilde{A})$, which means that $\tilde{A}^*$ transforms the covariance matrix $\Sigma$. If $\tilde{A}$ has a nontrivial nullspace, the distribution of $\tilde{A}^*\tilde{\eta}$ is singular, and there exists no pdf. Nevertheless, it is possible to approximate the distribution, e.g., by adding $\varepsilon\, \Id$ to the covariance matrix or restricting the problem to $\mathcal{N}(\tilde{A})^\bot$.
\end{remark}

To solve the inverse problem \eqref{eq:inverse_problem}, we use the approach of \cite{iresnet_01_regtheory}, i.e., we approximate the forward operator $A$ with a (single-layer) invertible residual network (iResNet) 
\begin{equation}
    \varphi_\theta = \Id - f_\theta, \quad f_\theta \colon X \to X.
\end{equation} This is done by a supervised training of $\varphi_\theta$ for which a paired dataset $\{x^{(i)}, z^{\delta,(i)}\}_{i=1, ..., N}$  of i.i.d.\ (independent and identically distributed) samples $x^{(i)} \sim p_X$, $z^{\delta,(i)} - Ax^{(i)} \sim p_H$ is needed.
One can then use the trained network to compute a regularized solution of \eqref{eq:inverse_problem} by $\varphi_\theta^{-1} (z^\delta)$. Invertibility of $\varphi_\theta$ is guaranteed using the constraint 
\begin{equation}
    \mathrm{Lip}(f_\theta) \leq L
\end{equation}
for some $L < 1$
(see \cite[Lemma 2.1]{iresnet_01_regtheory}, \cite{behrmann2019invertible}).

While implicit knowledge about $p_X$ and $p_H$ via the given dataset is sufficient for training $\varphi_\theta$, we derive some theoretical results using these pdfs explicitly. For this purpose, we need to make the following assumptions.


\begin{assumption} \label{ass:prior_and_noise}
Let
    \begin{itemize}
        \item$p_X \colon X \to \R_{\geq 0}$ be a pdf $\left( \text{i.e., } \int_X p_X(x) \, \mathrm{d}x = 1 \right)$ with existing first and second moments (i.e., $p_X(x) \|x\|$ and $p_X(x) \|x\|^2$ are Lebesgue-integrable) and expectation value
        \begin{equation}
            \mu_X = \int_X p_X(x) x \, \mathrm{d}x, 
        \end{equation}
        \item $p_H \colon X \to \R_{\geq 0}$ be a pdf $\left( \text{i.e., } \int_X p_H(\eta) \, \mathrm{d}\eta = 1 \right)$ with existing first and second moments (i.e., $p_H(\eta) \|\eta\|$ and $p_H(\eta) \|\eta\|^2$ Lebesgue-integrable) and zero expectation $\left( \text{i.e., } \int_X p_H(\eta) \eta \, \mathrm{d}\eta = 0 \right)$, and
        \item the random variables $x \sim p_X$, $\eta \sim p_H$ be stochastically independent.
    \end{itemize}
\end{assumption}

The crucial condition to guarantee the invertibility of $\varphi_\theta$ is $\mathrm{Lip}(f_\theta) \leq  L < 1$. Consequently, the inverse $\psi_\theta = \varphi_\theta^{-1}$ fulfills a property describable as a combination of coercivity and Lipschitz-continuity, which, in turn, trivially implies strong monotonicity. We formulate this equivalence in the following lemma.

\begin{lemma}[Inverse of iResNet]
\label{lem:inverse_of_resnet}
For $\varphi \colon X \to X$ and $0 \leq  L < 1$, the following two conditions are equivalent:
\begin{enumerate}
    \item[(1)] $\exists\, f \colon X \to X$ with $\mathrm{Lip}(f) \leq  L$ such that $\varphi = \Id - f$
    \item[(2)] $\exists\, \psi \colon X \to X$ with
    \begin{equation} \label{eq:cond_on_inverse_net}
        \forall z_1, z_2 \in X \colon \qquad (1-L^2) \| \psi(z_1) - \psi(z_2) \|^2 + \|z_1 - z_2\|^2 \leq  2 \langle z_1 - z_2, \psi(z_1) - \psi(z_2) \rangle
    \end{equation}
    such that $\varphi = \psi^{-1}$.
\end{enumerate}
In particular, \eqref{eq:cond_on_inverse_net} guarantees the invertibility of $\psi$.
\end{lemma}

\begin{proof}
    We begin with $(1) \Rightarrow (2)$.
    For arbitrary $x_1, x_2 \in X$, the condition $\mathrm{Lip}(\Id - \varphi) \leq  L$ implies
    \begin{align}
        \label{eq:lipschitz_id_minus_phi}
        & & \|(x_1 - \varphi(x_1)) - (x_2 - \varphi(x_2))\|^2 &\leq  L^2 \|x_1 - x_2\|^2 \notag \\
        &\Leftrightarrow& \|x_1 - x_2\|^2 - 2 \langle x_1 - x_2, \varphi(x_1) - \varphi(x_2) \rangle + \|\varphi(x_1) - \varphi(x_2) \|^2 &\leq  L^2 \|x_1 - x_2\|^2 \notag \\
        &\Leftrightarrow & (1-L^2)\|x_1 - x_2\|^2 + \|\varphi(x_1) - \varphi(x_2) \|^2 &\leq  2 \langle x_1 - x_2, \varphi(x_1) - \varphi(x_2) \rangle.
    \end{align}
    Since $\mathrm{Lip}(f) \leq  L$ implies invertibility of $\varphi$ (see \cite[Lemma 2.1]{iresnet_01_regtheory}, \cite{behrmann2019invertible}), we can define $\psi = \varphi^{-1}$ and $z_i = \varphi(x_i)$. This yields 
    \begin{align}
        & & (1-L^2)\|\psi(z_1) - \psi(z_2)\|^2 + \|z_1 - z_2 \|^2 &\leq  2 \langle z_1 - z_2, \psi(z_1) - \psi(z_2) \rangle
    \end{align}
    for arbitrary $z_1, z_2 \in X$.

    For the converse implication we now prove that \eqref{eq:cond_on_inverse_net} guarantees invertibility of $\psi$. Injectivity and Lipschitz continuity follow directly by applying the Cauchy-Schwarz inequality. To prove surjectivity, we construct a convergent sequence $(z_k)$ such that $\psi(z_k)$ converges to an arbitrary $x \in X$. We recursively define 
    \begin{equation}
        z_{k+1} = z_k + (1-L^2) (x - \psi(z_k)), \qquad z_0 \in X.
    \end{equation}
    It can be observed that
    \begin{align}
            2 \langle x - \psi(z_k), \psi(z_{k+1}) - \psi(z_k) \rangle &= 2 \langle z_{k+1} - z_k, \psi(z_{k+1}) - \psi(z_k) \rangle \frac{1}{1-L^2}\notag\\
            &\geq  \|\psi(z_{k+1}) - \psi(z_k) \|^2 + \frac{1}{1-L^2} \|z_{k+1} - z_k\|^2\notag \\
            &= \|\psi(z_{k+1}) - \psi(z_k) \|^2 + \frac{1}{1-L^2} \|(1-L^2)(x - \psi(z_k))\|^2 \notag\\
            &= \|\psi(z_{k+1}) - \psi(z_k) \|^2 + (1-L^2) \|x - \psi(z_k)\|^2
    \end{align}
    holds. Using this, it follows
    \begin{align}
            &\quad \,\, \|x - \psi(z_{k+1})\|^2 = \|(x - \psi(z_k)) - (\psi(z_{k+1}) - \psi(z_k))\|^2 \notag\\
            &= \|x - \psi(z_k)\|^2 - 2 \langle x - \psi(z_k), \psi(z_{k+1}) - \psi(z_k) \rangle + \|\psi(z_{k+1}) - \psi(z_k)\|^2 \notag\\
            &\leq  \|x - \psi(z_k)\|^2 - \|\psi(z_{k+1}) - \psi(z_k) \|^2 - (1-L^2) \|x - \psi(z_k)\|^2 + \|\psi(z_{k+1}) - \psi(z_k)\|^2 \notag\\
            &= L^2 \|x - \psi(z_k)\|^2.
    \end{align}
    Thus, we have $\|x - \psi(z_{k+1})\| \leq  L \|x - \psi(z_k)\|$, which implies $\|x - \psi(z_k)\| \leq  L^k \|x - \psi(z_0)\|$. Hence, it holds $\psi(z_k) \to x$ and \eqref{eq:cond_on_inverse_net} guarantees convergence of $(z_k)$. Since $x$ was arbitrary, $\psi$ is surjective and therefore invertible.    
    With the argumentation from the beginning in reversed order, we obtain the implication $(2) \Rightarrow (1)$.
\end{proof}

The following remark simplifies condition~\eqref{eq:cond_on_inverse_net} for $X=\R$.
\begin{remark}\label{rem:invertability_1_d}
    In case of $X = \R$ (one-dimensional space), condition \eqref{eq:cond_on_inverse_net} becomes
\begin{equation}
    \forall z_1, z_2 \in \R \colon \qquad \frac{1}{1+L} \leq  \frac{\psi(z_1) - \psi(z_2)}{z_1 - z_2} \leq  \frac{1}{1-L},
\end{equation}
which is a constraint on the slope of $\psi$ from above and from below.
\end{remark}

This motivates us to think of the condition on $\psi$ as a Lipschitz constraint similar to the one that applies to an iResNet. The following remark shows a direct connection between the iResNet and its inverse.

\begin{remark}[Inverses of iResNets are iResNets]\label{rem:inverse_of_ires_is_ires}
    From Lemma~\ref{lem:inverse_of_resnet}, we can deduce that one can write the inverse of an iResNet as a scaled iResNet. The constraint~\eqref{eq:cond_on_inverse_net} is equivalent to
    \begin{align}
   & & (1-L^2)^2 \| \psi(z_1) - \psi(z_2) \|^2 - 2 (1-L^2) \langle z_1 - z_2, \psi(z_1) - \psi(z_2) \rangle+ \|z_1 - z_2\|^2 &\leq  L^2\|z_1 - z_2\|^2 \notag \\
  &\Leftrightarrow    & \| (\Id-(1-L^2)\psi)(z_1)-(\Id-(1-L^2)\psi)(z_2)\| &\leq  L \|z_1 - z_2\| \notag \\
  &\Leftrightarrow  & \mathrm{Lip}(\Id-(1-L^2)\psi) & \leq L.
\end{align}
By defining $g := \Id-(1-L^2)\,\psi$
we obtain
\begin{align}
    \psi = \frac{1}{1-L^2}(\Id-g)\quad \text{where }\mathrm{Lip}(g)\leq L,
\end{align}
which is a scaled iResNet $\Id-g$ where $g$ satisfies the same Lipschitz constraint as $f$ in the forward mapping.
\end{remark}

\section{Approximation training}\label{sec:approx_train_Bayesian}


%

In \cite{iresnet_01_regtheory}, the \textit{approximation training} is introduced, in which the iResNet $\varphi_\theta$ is trained to approximate $A$, i.e., to solve
\begin{equation}\label{eq:approx_training}
    \min_{\theta \in \Theta_L} \frac{1}{N} \sum_{i=1}^N \| \varphi_\theta(x^{(i)}) - z^{\delta,(i)} \|^2
\end{equation}
for a given dataset of $N$ pairs $(x^{(i)}, z^{\delta,(i)}) \in X \times X$, $z^{\delta,(i)}=A x^{(i)} + \eta^{(i)}$.
The parameter space $\Theta_L$ encodes the architecture choice, and the Lipschitz constraint $\mathrm{Lip}(f_\theta) \leq L$. This setting was partly motivated by the so-called local approximation property (\cite[Theorem 3.1]{iresnet_01_regtheory}) characterizing convergence guarantees for the regularized solution $\varphi_\theta^{-1}(z^\delta)$ as $\delta\to 0$. In \cite{iresnet_01_regtheory}, specific network architectures were trained according to the approximation training and analyzed under which conditions they satisfy the properties of a convergent regularization scheme. This revealed a connection to the classical linear filter-based regularization theory. 

In contrast, we now aim to derive more general results without making restrictions on the architecture of the iResNet apart from the constraint on the Lipschitz constant of $f$. This enables us to analyze the influence of the noise and prior distribution on the trained network and, especially, the regularized solution. To this end, we consider the case of an infinite amount of training data allowing us to interpret Equation~\eqref{eq:approx_training} from a Bayesian point of view. To be more precise, taking the limit $N\to \infty$ in Equation~\eqref{eq:approx_training} and exploiting the independence of $x$ and $\eta$ (Assumption \ref{ass:prior_and_noise}) results in
\begin{equation} \label{eq:approx_training_bayesian}
    \min_{\theta \in \Theta_L} 
    \mathbb{E}_{x \sim p_X} \mathbb{E}_{\eta \sim p_H} \left(  \| \varphi_\theta(x) - A x - \eta \|^2 \right).
\end{equation}
The Euclidian norm can be decomposed into $\| \varphi_\theta(x) - A x - \eta \|^2 = \| \varphi_\theta(x) - A x\|^2 - 2 \langle \varphi_\theta(x) - A x, \eta \rangle + \|\eta\|^2$. Again because of the independence of $x$ and $\eta$ and due to $\mathbb{E}_{p_H}(\eta) = 0$, the mixed term vanishes in expectation. Therefore, we obtain
\begin{align}
    &\min_{\theta \in \Theta_L} \mathbb{E}_{x \sim p_X} \left( \| \varphi_\theta(x) - A x \|^2 + \mathbb{E}_{\eta\sim p_H}(\| \eta \|^2) \right) \notag \\
    \Leftrightarrow \quad & \min_{\theta \in \Theta_L} \mathbb{E}_{x \sim p_X} \left( \| \varphi_\theta(x) - A x \|^2 \right) .
\end{align}
Consequently, the noise does not influence the training. We could interpret this positively since the noise cannot lead to approximation errors of $\varphi_\theta$. However, a big drawback is that $\varphi_\theta^{-1}$, which shall regularize the inverse problem, neither depends on the noise level. Accordingly, the amount of regularization has to be set manually by choice of $L$ for the noise level $\delta$ (see \cite{iresnet_01_regtheory}) and is not data-dependent.


What remains is the influence of the prior distribution $p_X$ on the training of $\varphi_\theta$. We are especially interested in how $\varphi_\theta$ acts on the different eigenspaces of $A$ to analyze the dependence on the size of the eigenvalues. Therefore, we
make the rather strong assumption of  stochastic independence of the components $x_j = \langle x, v_j \rangle$:
\begin{assumption} \label{ass:independence}
    Let $x_j \sim p_{X,j}$ with $p_X(x)=\prod_j p_{X,j}(x_j)$.
\end{assumption}

Observe that this assumption is implicitly made when using Tikhonov regularization with $\|\cdot\|^2$-penalty term.
Furthermore, Assumption~\ref{ass:prior_and_noise} implies that $p_{X,j}$ has existing first and second moments with 
\begin{equation}
    \mu_{X,j} = \int_\R p_{X,j}(x_j) x_j \, \mathrm{d}x,
\end{equation}
which follows from Fubini's Theorem and the independence of the components. In this setting, a diagonal structure of the network
\begin{equation}\label{eq:diagonal}
    f_\theta(x) = \sum_j f_{j, \theta} (\langle x, v_j \rangle) v_j \quad \text{with} \quad f_{j, \theta} \colon \R \to \R, 
\end{equation}
as in \cite{iresnet_01_regtheory}, is sufficient. 
Hence, the above minimization problem can be analyzed for each component separately due to properties of the eigendecomposition, and we get
\begin{equation}\label{eq:training_Bayesian_view}
    \min_{\theta \in \Theta_L} \mathbb{E}_{x_j \sim p_{X,j}} \left( | (1 - \sigma_j^2) x_j - f_{j, \theta}(x_j) |^2 \right).
\end{equation}
This is equivalent to a 1d-setting with $A \colon \R \to \R$, $x_j \mapsto \sigma_j^2\, x_j$.

In the following, instead of minimizing over a parameter space $\Theta_L$, we directly consider a function space $\mathcal{F}$ encoding the Lipschitz constraint ($\mathrm{Lip}(f) \leq L$) and the architecture choice.
For simplicity, in what follows, we omit the index $j$ and consider
\begin{equation} \label{eq:bayes_train_probl_comp_wise}
    \min_{f \in \mathcal{F}} \int_{\R} p_X(x) | (1 - \sigma^2) x - f(x) |^2 \, \mathrm{d}x.
\end{equation}

If $\mathcal{F}$ allows for (affine) linear functions and in case of $1-\sigma^2 \leq  L$, we can indicate the trivial solution $f = (1-\sigma^2)\, \Id$. Obviously, this solution is unique on $\mathrm{supp}(p_X)$. Thus, for eigenvalues $\sigma^2$, which are not too small, the training leads to a perfect approximation of the forward operator and no regularization of the inverse problem. For $1-\sigma^2 > L$, the minimization problem gets more interesting due to the Lipschitz constraint.
First, we derive the following result, which builds the basis for a subsequent generalization.

\begin{lemma} \label{lem:approx_training_linear}
Let $\mathcal{F} = \{f \in C(\R) \, | \, \exists\, m \in [-L, L], b \in \R \colon \,f(x) = mx + b\}$
and $L < 1-\sigma^2$. Then,
\begin{equation}
    f^*(x) = Lx + (1-\sigma^2-L)\, \mu_X
\end{equation}
is the unique solution of the minimization problem \eqref{eq:bayes_train_probl_comp_wise}. 
\end{lemma}

\begin{proof}
For a function $f$ of the form $f(x) = mx + b$ with the constraint $m^2 \leq  L^2$, we can solve \eqref{eq:bayes_train_probl_comp_wise} by using the Lagrangian
\begin{equation}
    K(m,b, \lambda) = \int_\R p_X(x) | (1 - \sigma^2 - m) x - b |^2 \, \mathrm{d}x + \lambda (m^2 - L^2),
\end{equation}
where the integral is well-defined due to the existence of the first and second moments of $p_X$.
The convexity, coercivity, and continuity of the integral term w.r.t. $(m,b)$ imply that there exists a minimizer. Therefore, we can calculate the minimizer using the necessary conditions (KKT conditions) 
\begin{align} \label{eq:dLdm}
    \frac{\partial K}{\partial m} (m,b,\lambda) = -\int_\R 2 p_X(x) \left((1 - \sigma^2 - m) x - b \right) x \, \mathrm{d}x + 2\lambda m \stackrel{!}{=} 0, \\
    \label{eq:dLdb}
    \frac{\partial K}{\partial b} (m,b,\lambda) = -\int_\R 2 p_X(x) \left((1 - \sigma^2 - m) x - b \right) \, \mathrm{d}x \stackrel{!}{=} 0,\\
    \label{eq:lagrange_constraint}
    \lambda (m^2 - L^2) \stackrel{!}{=} 0, \\
    \lambda \geq  0.
\end{align}
Rearranging \eqref{eq:dLdm} for $m$ and \eqref{eq:dLdb} for $b$ leads to
\begin{align}
m = \frac{\int_\R 2 p_X(x) \left((1 - \sigma^2) x - b \right) x \, \mathrm{d}x}{\int_\R 2 p_X(x) x^2 \, \mathrm{d}x + 2\lambda} &= \frac{(1-\sigma^2)\mathbb{E}_{ p_X}(x^2) - b  \mu_X}{\mathbb{E}_{ p_X}(x^2) + \lambda}, \\
    b = \frac{\int_\R 2 p_X(x) (1 - \sigma^2 - m) x \, \mathrm{d}x}{\int_\R 2 p_X(x) \, \mathrm{d}x} &= (1 - \sigma^2 - m) \mu_X,
\end{align}
where we use the abbreviated notation $\mathbb{E}_{ p_X}$ for $\mathbb{E}_{x\sim p_X}$. 
Now, plugging $b = (1 - \sigma^2 - m)  \mu_X$ into the equation for $m$ implies
\begin{align}
    & & m &= \frac{(1-\sigma^2)\mathbb{E}_{ p_X}(x^2) - (1-\sigma^2 - m)  \mu_X^2}{\mathbb{E}_{ p_X}(x^2) + \lambda} \notag\\
    &\Leftrightarrow & \left(1 - \frac{ \mu_X^2}{\mathbb{E}_{ p_X}(x^2) + \lambda} \right) m &= \frac{(1-\sigma^2) \left(\mathbb{E}_{ p_X}(x^2) -   \mu_X^2 \right)}{\mathbb{E}_{ p_X}(x^2) + \lambda}\notag \\
    &\Leftrightarrow & \frac{\mathrm{Var}_{p_X}(x) + \lambda}{\mathbb{E}_{ p_X}(x^2) + \lambda}  m &= \frac{(1-\sigma^2)\mathrm{Var}_{p_X}(x)}{\mathbb{E}_{ p_X}(x^2) + \lambda} \notag \\
    &\Leftrightarrow & m &= (1-\sigma^2)\frac{\mathrm{Var}_{p_X}(x) }{\mathrm{Var}_{p_X}(x)+ \lambda}
\end{align}
Since $1-\sigma^2 > L$ holds by assumption, we need $\lambda >0$ to ensure $m\leq  L$. Then, \eqref{eq:lagrange_constraint} directly implies $m = L$. As $m$ is uniquely determined we also know that $b$ is unique with $b = (1-\sigma^2 - L)\, \mu_X$.
\end{proof}

The previous lemma provides the prerequisite for the following theorem, where $\mathcal{F}$ contains arbitrary Lipschitz continuous functions with constrained Lipschitz constant.

\begin{theorem}\label{thm:approx_training_lipschitz}
Let $\mathcal{F} = \{f \in C^{0,1}(\R) \, | \, \mathrm{Lip}(f) \leq  L\}$.
Then,
\begin{equation}
    f^*(x) = 
    \begin{cases}
    (1-\sigma^2) x & \text{if } 1-\sigma^2 \leq L, \\
    Lx + (1-\sigma^2-L)  \mu_X & \text{if } 1-\sigma^2 > L
    \end{cases}
\end{equation}
is the solution of the minimization problem \eqref{eq:bayes_train_probl_comp_wise}. 
This solution is unique on $\mathrm{supp}(p_X)$ and for $1-\sigma^2 > L$ even on the convex hull of $\mathrm{supp}(p_X)$.
\end{theorem}

\begin{proof}
We define $F \colon \mathcal{F} \to \R$,
\begin{equation}
F(f) = \int_\R p_X(x) |(1-\sigma^2)x - f(x)|^2 \, \mathrm{d}x
\end{equation}
and start with the case $1-\sigma^2 \leq L$. Obviously, it holds $F(f^*) = 0$, so $f^*$ is a minimizer. Using the fundamental lemma of the calculus of variations, one can deduce the uniqueness on $\mathrm{supp}(p_X)$.

Now, consider $1-\sigma^2 > L$ and let $g \in C^{0,1}(\R)$, $\mathrm{Lip}(g) \leq  L$ be an arbitrary function. We will show that $F(g) > F(f^*)$ holds, if $g \neq f^*$ on the convex hull of $\mathrm{supp}(p_X)$.

First, we verify that $F$ is well-defined, i.e, for $f\in \mathcal{F}$
\begin{align}
 F(f) & \leq 2 \int_\R p_X(x) ( (1-\sigma^2)^2 x^2 + |f(x)|^2 )  \, \mathrm{d}x \notag \\
 & = 2 \int_\R p_X(x) ( (1-\sigma^2)^2 x^2 + |f(x)-f(0) + f(0) |^2 )  \, \mathrm{d}x \notag \\
& \leq  2 \int_\R p_X(x)  (1-\sigma^2)^2 x^2  \, \mathrm{d}x  + 4 \int_\R p_X(x) |f(x)-f(0)|^2  \, \mathrm{d}x + 4 \int_\R p_X(x) |f(0)|^2   \, \mathrm{d}x \notag \\
& \leq 2 (1-\sigma^2)^2 \int_\R p_X(x) x^2  \, \mathrm{d}x  + 4 L^2 \int_\R p_X(x) x^2  \, \mathrm{d}x + 4 f(0)^2 < \infty
\end{align}
holds, as the second moment of $p_X$ exists. 

Due to $L < 1-\sigma^2$, the function $g$ has always a smaller slope than $(1-\sigma^2)\,\Id$, which implies that there exists an intersection point $x_0$ such that $g(x_0) = (1-\sigma^2)\,x_0$. The affine linear function $\tilde{f}(x) = L (x - x_0) + (1-\sigma^2)\, x_0$ possesses the same intersection point.

In case of $g=\tilde{f}$ on the convex hull of $\mathrm{supp}(p_X)$, we simply apply Lemma \ref{lem:approx_training_linear}. This shows that $g$ can be the minimizer only if $\tilde{f} = f^*$.

In the case of $g\neq \tilde{f}$, let us examine the integrand of $F(g)$. For any $x \in \R$, it holds
\begin{align}
        & \quad \, \, |(1-\sigma^2)x - g(x)|^2 = |(1-\sigma^2)x - (g(x) - \tilde{f}(x)) - \tilde{f}(x)|^2\notag \\
        &= |(1-\sigma^2)x- \tilde{f}(x)|^2 - 2((1-\sigma^2)x - \tilde{f}(x))(g(x) - \tilde{f}(x)) + |g(x) - \tilde{f}(x)|^2.
\end{align}
For $x \leq  x_0$, we have $(1-\sigma^2)x - \tilde{f}(x) \leq  0$ and $\mathrm{Lip}(g) \leq  L$ implies $g(x) - \tilde{f}(x) \geq  0$. Thus, we obtain
\begin{equation}
-2\left((1-\sigma^2)x - \tilde{f}(x)\right)\left(g(x) - \tilde{f}(x)\right) \geq  0,
\end{equation}
which implies $|(1-\sigma^2)x - g(x)|^2 \geq  |(1-\sigma^2)x- \tilde{f}(x)|^2$. Analogously, for $x \geq  x_0$, we observe that $(1-\sigma^2)x - \tilde{f}(x) \geq  0$ and  $g(x) - \tilde{f}(x) \leq  0$, which also implies $|(1-\sigma^2)x - g(x)|^2 \geq  |(1-\sigma^2)x- \tilde{f}(x)|^2$. 
Therefore, it holds $F(g) \geq  F(\tilde{f})$.

Finally, we show that $F(g) = F(\tilde{f})$ implies $\tilde{f} =g$ on the convex hull of $\mathrm{supp}(p_X)$. If $F(g) = F(\tilde{f})$, it holds
\begin{equation}
\int_\Omega p_X(x) |(1-\sigma^2)x - g(x)|^2 - p_X(x) |(1-\sigma^2)x - \tilde{f}(x)|^2 \, \mathrm{d}x = 0.
\end{equation}
for any measurable $\Omega \subset \R$, since the term under the integral is always greater than or equal to zero. The fundamental lemma of the calculus of variations then implies
\begin{equation}
p_X(x) |(1-\sigma^2)x - g(x)|^2 = p_X(x) |(1-\sigma^2)x - \tilde{f}(x)|^2
\end{equation}
for almost all $x \in \R$. Thus, for any $x \in \mathrm{supp}(p_X)$, it holds $g(x) = \tilde{f}(x)$ and for $x_1, x_2 \in \mathrm{supp}(p_X)$, we obtain $g(x_1) - g(x_2) = L(x_1 - x_2)$. Consequently, for any $x$ in between of $x_1$ and $x_2$, $g(x) = \tilde{f}(x)$ must also hold, otherwise $\mathrm{Lip}(g) \leq  L$ would be violated. Hence, $g$ and $\tilde{f}$ coincide on the convex hull of $\mathrm{supp}(p_X)$.
\end{proof}

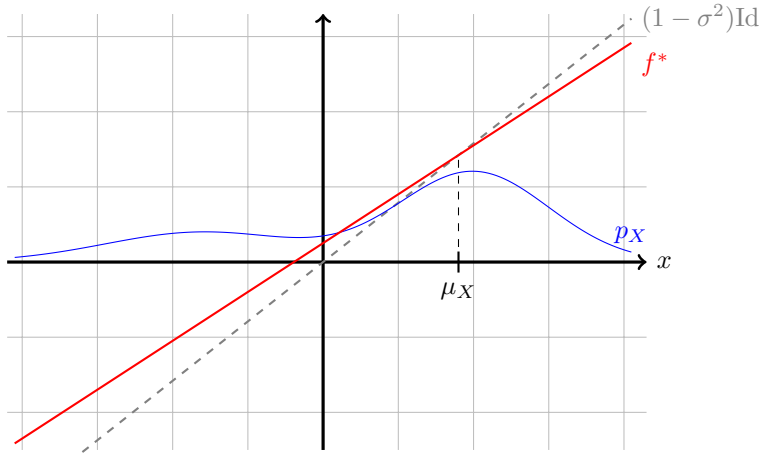
\begin{figure}[h]
\centering
\begin{tikzpicture}
  \draw[step=1.0, lightgray, thin] (-4.2,-2.5) grid (4.3,3.3);
  \draw[->, thick, very thick] (-4.2, 0) -- (4.3, 0) node[right] {$x$};
  \draw[->, thick, very thick] (0, -2.5) -- (0, 3.3);
  \draw[scale=1, domain=-3.2:4.1, variable=\t, gray, thick, dashed] plot[samples=161] ({\t}, {(1-0.21)*\t}) node[right] {$(1-\sigma^2)\Id$};
   \draw[scale=1, domain=-4.1:4.1, variable=\t, blue] plot[samples=161] ({\t}, {1.2*exp(-0.5*(\t-2)*(\t-2)) + 0.4*exp(-0.3*(\t+1.6)*(\t+1.6))}) node[above] {$p_{X}$};
   \draw[scale=1, domain=-4.1:4.1, variable=\t, red, thick] plot[samples=161] ({\t}, {(0.65)*(\t-1.8)+(1-0.21)*1.8}) node[below right] {$f^*$};
\draw[dashed] (1.8,0.2) -- (1.8,1.4);
\draw[thick] (1.8,0.14) -- (1.8,-0.14) node[below] {$\mu_{X}$};
\end{tikzpicture}
\caption{The residual function $f^*$ which results from the approximation training (Theorem \ref{thm:approx_training_lipschitz}) is affine linear and only depends on $\sigma^2$, $L$ and $\mu_{X}$. In case of $\sigma^2 < 1- L$, $f^*$ exhibits the maximum possible slope of $L$ and intersects $(1-\sigma)^2\, \Id$ at the mean $\mu_{X}$ of the prior distribution.}\label{fig:Sol approx train}
\end{figure}




Figure~\ref{fig:Sol approx train} exemplifies the solution $f^*$ for a Gaussian mixture prior $p_X$.
The inverse $\varphi_\theta^{-1}$ corresponding to the minimizer of \eqref{eq:bayes_train_probl_comp_wise} derived in the previous theorem provides a convergent regularization scheme, which we discuss in the following remark.

\begin{remark}\label{remark:sq_soft_tsvd}
    One can express $\varphi_\theta^{-1}$ as a filter-based regularization scheme with the squared soft TSVD filter function with bias. The affine linear diagonal architecture was already analyzed in \cite[Lemma 4.2]{iresnet_01_regtheory}. For
    \begin{equation}
        f_j(x_j) = \min\{1-\sigma_j^2, L\} x_j + \max\{0, 1-\sigma_j^2-L\} \mu_{X,j}
    \end{equation}
    (which coincides with the solution in Theorem \ref{thm:approx_training_lipschitz}), it holds
    \begin{align}
        \varphi_\theta^{-1}(z) &= \hat{b}_L +  \sum_{j} \hat{r}_L (\sigma_j^2) \langle z, v_j \rangle,\\
        \hat{r}_L(\sigma_j^2) = \frac{1}{\max\{\sigma_j^2, 1-L\}}, &\qquad \hat{b}_L = \sum_{\sigma_j^2 < 1-L} \frac{1-\sigma_j^2-L}{1-L} \mu_{X,j} v_j.
    \end{align}
    By \cite[Lemma 3.3]{iresnet_01_regtheory}, this filter scheme with bias defines a convergent regularization method for $L \to 1$ in case of vanishing noise and a suitable parameter choice $L(\delta)$.
\end{remark}

The previous results show that approximation training of a diagonal architecture always leads to an affine linear $\varphi_\theta$, independent of prior and noise distribution ($p_X$, $p_H$).
Hence, an affine linear residual layer is the best architecture choice for this task. This implies that $\varphi_\theta^{-1}$ is a reconstruction scheme with minimal data dependency since only the mean $\mu_X$ of the prior distribution has an influence. 
Furthermore, $\varphi_\theta^{-1}$ is equivalent to a classical regularization scheme, where one predefines the amount of regularization by choosing the parameter $L$ depending on the noise level.

For the general approximation training problem 
\begin{equation}\label{eq:approx_general_network}
    \min_{f\in C(\R^n, \R^n), \ \mathrm{Lip}(f)\leq L} \mathbb{E}_{x\sim p_X} \left( \| (\Id - f) (x) - Ax\|^2 \right) 
\end{equation}
the previous investigations suggest that the solution depends on the second moments of the prior distribution $p_X$ at most. A detailed consideration of the general setting for the approximation training is beyond the scope of the present work.

\section{Reconstruction training}
\label{sec:reco_training}


%
%
%
%
The results in the last section show that the approximation training is insufficient for learning noise and more data-dependent regularization. To address this, we instead consider the training
\begin{equation}\label{eq:reco_training}
    \min_{\theta \in \Theta} \frac{1}{N} \sum_{i=1}^N \|\varphi_\theta^{-1}(A x^{(i)} + \eta^{(i)}) - x^{(i)} \|^2 \quad \text{s.t. } \mathrm{Lip}(f_\theta) \leq L
\end{equation}
for given training data $\{x^{(i)}\}_i \subset X$, noise realizations $\{\eta^{(i)}\}_i \in X$ and $\varphi_\theta = \Id - f_\theta$. This is also motivated by sufficient conditions for the convergence analysis in \cite[Remark 4.1]{iresnet_01_regtheory}.
We refer to this training scheme as the \textit{reconstruction training}.
One can also interpret this reconstruction training as a supervised training on data pairs $(x^{(i)}, z^{\delta,(i)})$ for $\varphi_{\theta}^{-1}(z^{\delta,(i)}) \approx x^{(i)}$ with $z^{\delta,(i)} = Ax^{(i)} + \eta^{(i)}$.

Using Lemma \ref{lem:inverse_of_resnet}, we know that
\begin{equation} \label{eq:constraint_reco_training}
    \begin{split}
        &\min_{\theta \in \Theta} \frac{1}{N} \sum_{i=1}^N \|\psi_\theta(A x^{(i)} + \eta^{(i)}) - x^{(i)} \|^2\\
        \quad \text{s.t. } &(1-L^2) \|\psi_\theta(z_1) - \psi_\theta(z_2)\|^2 + \|z_1 - z_2\|^2 \leq 2 \langle \psi_\theta(z_1) - \psi_\theta(z_2), z_1 - z_2 \rangle \quad \forall z_1, z_2 \in X
    \end{split}
\end{equation}
is an equivalent problem,
assuming that
the architectures of $\varphi_\theta$ and $\psi_\theta$ can approximate any continuous function. 

Similar to the approximation training, we analyze the case of an unlimited amount of training data with $x \sim p_X$ and $\eta \sim p_H$ fulfilling Assumption \ref{ass:prior_and_noise}.
Thus, we obtain the minimization problem
\begin{equation} \label{eq:reco_training_integral}
 \min_{\psi \in \Psi} \int_X \int_X p_X(x)\, p_H(\eta)\, \|\psi(Ax + \eta) - x\|^2 \, \mathrm{d}\eta \, \mathrm{d}x,
\end{equation}
where the set of functions $\Psi$ represents the choice of the architecture and the constraints on the parameters. In this context, we also make use of the density function of $z^\delta = Ax + \eta$, which is given by
\begin{equation} \label{eq:pdf_of_z}
    p_Z(z^\delta) = \int_X p_X(x) p_H(z^\delta - Ax) \, \mathrm{d}x.
\end{equation}
With this, we can define the space of $p_Z$-weighted $L^2$-functions as
\begin{align}
    L_{p_Z}^2(X,X) &= \{ \psi \colon X \to X \, | \, \|\psi\|_{p_Z, 2} < \infty\}, \\
    \|\psi\|_{p_Z,2}^2 &= \int_X p_Z(z) \|\psi(z)\|_X^2 \, \mathrm{d}z,
\end{align}
which is a Hilbert space. Note that functions from $L_{p_Z}^2(X,X)$ are (only) well-defined on $\mathrm{supp}(p_Z) \subset X$, which is sufficient for our purposes.

At first, we consider the unconstrained case of $\Psi = L_{p_Z}^2(X,X)$. In this setting, the conditional mean\footnote{Expectation value corresponding to $p(x|z^\delta) = \frac{p_H(z^\delta - Ax) p_X(x)}{p_Z(z^\delta)}$.} $\hat{\psi}(z^\delta) = \mathbb{E}(x | z^\delta)$ is the solution of \eqref{eq:reco_training_integral} which is in line with the established theory in statistical inverse problems (see, e.g., conditional mean estimator in the discussion of Bayes cost estimators in \cite{kaipio2006statistical} or \cite[Prop. 2]{adler2018deep}).

\begin{lemma} \label{lem:unconstrained_reco_training}
    Let Assumption \ref{ass:prior_and_noise} hold and $\Psi = L_{p_Z}^2(X,X)$. Then, 
    \begin{equation}
        \hat{\psi} = \left(z^\delta \mapsto \mathbb{E}(x | z^\delta) \right) = \left(z^\delta \mapsto \int_{X} p(x | z^\delta) x \, \mathrm{d}x\right)
    \end{equation}
    is the solution of \eqref{eq:reco_training_integral}, which is unique w.r.t.\ the $L_{p_Z}^2$-norm.
\end{lemma}

\begin{proof}
    We denote the objective function by $F \colon L_{p_Z}^2(X,X) \to [0, \infty)$,
    \begin{equation}
        F(\psi) = \int_X \int_X p_X(x)\, p_H(z^\delta-Ax)\, \|\psi(z^\delta) - x\|^2 \, \mathrm{d}z^\delta \, \mathrm{d}x,
    \end{equation}
    which coincides with \eqref{eq:reco_training_integral} with the substitution $z^\delta = Ax + \eta$.
    Note that for all $\psi \in L_{p_Z}^2(X,X)$, it holds $F(\psi) < \infty$ since
    \begin{equation}
        \|\psi(z^\delta) - x\|^2 = \|\psi(z^\delta)\|^2 - 2 \langle \psi(z^\delta), x \rangle + \|x\|^2 \leq 2\|\psi(z^\delta)\|^2 + 2\|x\|^2
    \end{equation}
    and the integrals
    \begin{align}
        &\int_X \int_X p_X(x)\, p_H(z^\delta-Ax)\, \|\psi(z)\|^2 \, \mathrm{d}z^\delta \, \mathrm{d}x = \int_X p_Z(z^\delta)\, \|\psi(z^\delta)\|^2 = \|\psi\|_{p_Z, 2}^2, \notag \\
        &\int_X \int_X p_X(x)\, p_H(z^\delta-Ax) \,\|x\|^2 \, \mathrm{d}z^\delta = \int_X p_X(x)\, \|x\|^2 \int_X p_H(z^\delta-Ax) \, \mathrm{d}z^\delta \, \mathrm{d}x = \int_X p_X(x)\, \|x\|^2 \, \mathrm{d}x
    \end{align}
    are both finite.
    
    Besides, note that $F$ is convex w.r.t.\ $\psi$ since $\psi \mapsto \|\psi(z^\delta) - x\|^2$ is convex for any $x, z^\delta \in X$. Thus, we can find the minimizer of $F$ by setting its derivative to zero.

    To compute the derivative of $F$, we consider
    \begin{align}
            F(\psi + h) &= \int_X \int_X p_X(x)\, p_H(z^\delta-Ax)\, \|\psi(z^\delta) + h(z^\delta) - x\|^2 \, \mathrm{d}z^\delta \, \mathrm{d}x \notag\\
            &= \int_X \int_X p_X(x)\, p_H(z^\delta-Ax)\, \|\psi(z^\delta) - x\|^2 \, \mathrm{d}z^\delta \, \mathrm{d}x \notag\\
            &+ 2\int_X \int_X p_X(x)\, p_H(z^\delta-Ax)\, \langle \psi(z^\delta) - x,  h(z^\delta) \rangle \, \mathrm{d}z^\delta \, \mathrm{d}x\notag \\
            &+ \int_X \int_X p_X(x)\, p_H(z^\delta-Ax)\, \| h(z^\delta)\|^2 \, \mathrm{d}z^\delta \, \mathrm{d}x
    \end{align}
    The first of the three summands is $F(\psi)$, the last one equals $\|h\|_{p_Z, 2}^2$ and the second summand equals $\partial F(\psi) h$. Note that the second summand is finite since $F(\psi)$ and $\|h\|_{p_Z, 2}^2$ are both finite.
    Using Fubini's theorem, we obtain
    \begin{align}
        \partial F(\psi) h &= 2\int_X \int_X p_X(x)\, p_H(z^\delta- Ax)\, \langle \psi(z^\delta) - x,  h(z^\delta) \rangle \, \mathrm{d}z^\delta \, \mathrm{d}x \notag \\
        &= 2\int_X \left\langle \int_X p_X(x)\, p_H(z^\delta - Ax)\, ( \psi(z^\delta) - x )\, \mathrm{d}x, \, h(z^\delta) \right\rangle \, \mathrm{d}z^\delta.
    \end{align}
    We are looking for $\hat{\psi}$ such that $\partial F(\hat{\psi}) h = 0$ for any $h \in L_{p_Z}^2(X,X)$. Hence, the fundamental lemma of the calculus of variations implies 
    \begin{align}
        & & \int_X p_X(x)\, p_H(z^\delta- Ax)\, ( \hat{\psi}(z^\delta) - x )\, \mathrm{d}x &= 0 \notag\\
        &\Leftrightarrow & p_Z(z^\delta) \cdot \hat{\psi}(z^\delta) &= \int_X p_X(x)\, p_H(z^\delta- Ax) x \, \mathrm{d}x \notag\\
        &\Leftrightarrow & \hat{\psi}(z^\delta) &= \frac{\int_X p_X(x)\, p_H(z^\delta - Ax)\, x \, \mathrm{d}x}{p_Z(z^\delta)}
    \end{align}
    for almost all $z^\delta \in \mathrm{supp}(p_Z)$.
    According to Bayes' formula
    \begin{equation}
        p(x | z^\delta) = \frac{p(z^\delta | x)\, p_X(x)}{p_Z(z^\delta)} = \frac{p_H(z^\delta - Ax)\, p_X(x)}{p_Z(z^\delta)},
    \end{equation}
    $\hat{\psi}(z^\delta)$ is the expectation value of the posterior density function $p(x | z^\delta)$.

    Finally, we need to make sure that $\hat{\psi}(z^\delta) = \mathbb{E}(x|z^\delta)$ is an $L_{p_Z}^2$-function. For this, we use Jensen's inequality for conditional expectation values \cite[Theorem 8.20, Corollary 8.21]{klenke2020probability} and obtain
    \begin{equation}
        \|\mathbb{E}(x|z^\delta)\|^2 = \sum_{j=1}^n |\mathbb{E}(x_j|z^\delta)|^2 \leq \sum_{j=1}^n \mathbb{E} \left(|x_j|^2 \, \big| \, z^\delta \right) = \mathbb{E} \left( \|x\|^2 \, \big| \, z^\delta \right).
    \end{equation}
    Then, by the definition of conditional expectations, we get
    \begin{equation}
        \int_X p_Z(z^\delta)\, \mathbb{E} \left( \|x\|^2 \, \big| \, z^\delta \right) \, \mathrm{d}z^\delta = \mathbb{E}_{p_Z} \left(\mathbb{E} \left( \|x\|^2 \, \big| \, z^\delta \right) \right) = \mathbb{E}(\|x\|^2),
    \end{equation}
    which is finite, and the proof is complete.
\end{proof}

Next, we consider the constrained reconstruction training, where we encode an arbitrary constraint, e.g., \eqref{eq:constraint_reco_training}, by choosing $\Psi$ to be a suitable subset of $L_{p_Z}^2(X,X)$.

\begin{lemma}\label{lem:reco training equivalent formulation}
   Let Assumption \ref{ass:prior_and_noise} hold, $\Psi$ be an arbitrary subset of $L_{p_Z}^2(X,X)$ and let $\hat{\psi}: X \to X$, $z^\delta \mapsto \mathbb{E}(x|z^\delta)$ be the conditional mean estimator. Then, the minimization problem \eqref{eq:reco_training_integral} is equivalent to
    \begin{equation}\label{eq:min_distance_to_posterior}
        \min_{\psi \in \Psi} \int_X p_Z(z^\delta)\, \| \psi(z^\delta) - \hat{\psi}(z^\delta) \|^2 \, \mathrm{d}z^\delta.
    \end{equation}
    Note that the existence of an actual minimizer is only guaranteed for closed $\Psi$.
\end{lemma}

\begin{proof}
    The minimization problem \eqref{eq:reco_training_integral} is equivalent to
    \begin{equation} \label{eq:min_psi_hatpsi}
        \min_{\psi \in \Psi} \int_X \int_X p_X(x)\, p_H(z^\delta - Ax) \left(\|\psi(z^\delta) - x\|^2 - \|\hat{\psi}(z^\delta) - x\|^2 \right) \, \mathrm{d}z^\delta \, \mathrm{d}x.
    \end{equation}
    In the proof of Lemma \ref{lem:unconstrained_reco_training}, we have already established that the integrals are finite. To split the integral term into two parts, we use
    \begin{equation}
        \|\psi(z^\delta) - x\|^2 - \|\hat{\psi}(z^\delta) - x\|^2 = \|\psi(z^\delta)\|^2 - \|\hat{\psi}(z^\delta)\|^2 - \langle 2x, \psi(z^\delta) - \hat{\psi}(z^\delta) \rangle.
    \end{equation}
    Fubini's theorem and the definition of $p_Z$ \eqref{eq:pdf_of_z} implies
    \begin{equation}
         \int_X \int_X p_X(s)\, p_H(z^\delta - Ax) \left(\|\psi(z^\delta)\|^2 - \|\hat{\psi}(z^\delta)\|^2 \right) \, \mathrm{d}z^\delta \, \mathrm{d}x = \int_X  p_Z(z^\delta) \left(\|\psi(z^\delta)\|^2 - \|\hat{\psi}(z^\delta)\|^2 \right) \, \mathrm{d}z^\delta.
    \end{equation}
    Again using Fubini's theorem and the definition of $\hat{\psi}(z^\delta) = \mathbb{E}(x|z^\delta)$ (see proof of Lemma \ref{lem:unconstrained_reco_training}), we obtain
    \begin{align}
        &\quad \, \, \int_X \int_X p_X(x)\, p_H(z^\delta - Ax)\, \langle 2x, \psi(z^\delta) - \hat{\psi}(z^\delta) \rangle \, \mathrm{d}z^\delta \, \mathrm{d}x  \notag \\
        &= \int_X \left\langle \int_X p_X(x)\, p_H(z^\delta - Ax)\, 2x  \, \mathrm{d}x, \, \psi(z^\delta) - \hat{\psi}(z^\delta) \right\rangle\, \mathrm{d}z^\delta \notag \\
        &= \int_X p_Z(z^\delta) \left\langle 2\int_X \frac{p_X(x)\, p_H(z^\delta - Ax)}{p_Z(z^\delta)} x  \, \mathrm{d}x, \, \psi(z^\delta) - \hat{\psi}(z^\delta) \right\rangle\, \mathrm{d}z^\delta \notag \\
        &= \int_X p_Z(z^\delta) \left\langle 2\hat{\psi}(z^\delta), \, \psi(z^\delta) - \hat{\psi}(z^\delta) \right\rangle\, \mathrm{d}z^\delta.
    \end{align}
    Thus, \eqref{eq:min_psi_hatpsi} is equivalent to
    \begin{equation}
        \min_{\psi \in \Psi} \int_X p_Z(z^\delta) \left(\|\psi(z^\delta)\|^2 - \|\hat{\psi}(z^\delta)\|^2 - \left\langle 2\hat{\psi}(z^\delta), \, \psi(z^\delta) - \hat{\psi}(z^\delta) \right\rangle \right) \, \mathrm{d}z^\delta.
    \end{equation}
    Now, the assertion follows from $\|\psi(z^\delta)\|^2 - \|\hat{\psi}(z^\delta)\|^2 - \langle 2\hat{\psi}(z^\delta), \, \psi(z^\delta) - \hat{\psi}(z^\delta) \rangle = \|\psi(z^\delta) - \hat{\psi}(z^\delta)\|^2$.
\end{proof}

Thus, in the constraint case, the function $\psi^*$, obtained by reconstruction training, aims to approximate the conditional mean estimator for the $p_Z$-weighted $L^2$-norm. In other words, reconstruction training with a constraint corresponds to a projection of the conditional mean estimator onto the constraint set with respect to the $p_Z$-weighted $L^2$-norm.

\begin{remark}
Since we know from Remark~\ref{rem:inverse_of_ires_is_ires} that the inverse network $\psi$ can be interpreted as a scaled iResNet, we can further compare the minimization problem to the case of approximation training. In the notation of an iResNet, the problem formulated in~\eqref{eq:min_distance_to_posterior} is equivalent to
    \begin{equation}
        \min_{g \in C(\R^n,\R^n), \ \mathrm{Lip}(g)\leq L} \int_X p_Z(z^\delta) \|(\Id-g)(z^\delta) - (1-L^2)\,\mathbb{E}(x\vert z^\delta) \|^2 \, \mathrm{d}z^\delta.
    \end{equation}
Thus, the reconstruction training is equivalent to training an iResNet with residual function $g$ ($\mathrm{Lip}(g)\leq L$) to fit a scaled version of the posterior expectation estimator. In contrast, approximation training aims to fit the same architecture type to the linear operator $A$. Overall, 
this indicates that theoretical and numerical properties (such as data-dependence) for the two strategies are the sole consequences of the training approach, and there is no additional bias due to the architecture choice of an iResNet as the forward mapping when assuming sufficient approximation capability.
\end{remark}

So far, the distribution of the noise $p_H$ was fixed. Now, we want to consider a variable noise level $\delta > 0$ by introducing the pdf $p_{H, \delta} \colon X \to \R_{\geq 0}$. We do not specify the exact relation of $p_{H, \delta}$ on $\delta$ but make the rather informal assumption that $\eta^\delta \sim p_{H, \delta}$ implies $\|\eta^\delta\| \sim \delta$ with high probability. So, $\delta \to 0$ corresponds to the case of vanishing noise. Analogously, let $p_{Z, \delta}$ be defined according to \eqref{eq:pdf_of_z} s.t.\ $z^\delta \sim p_{Z, \delta}$ holds for $z^\delta = A x + \eta^\delta$. The posterior mean now also depends on $\delta$ and may therefore be defined as
\begin{equation}
    \hat{\psi}_\delta (z) = \int_X \frac{p_{H,\delta}(z
- Ax)\, p_X(x)}{p_{Z,\delta}(z)} x \, \mathrm{d}x.
\end{equation}

Further, we want to specify the set $\Psi \subset L_{p_Z}^2 (X,X)$, which represents the inverses of possible iResNet architectures depending on $\delta$ and $L$. To encode the side constraint of \eqref{eq:constraint_reco_training}, we define
\begin{align}
    \Psi_L^\delta = \{ \psi \in L_{p_{Z, \delta}}^2 (X,X) \cap C( \mathrm{supp}( p_{Z, \delta}), X) \, &| \, \eqref{eq:inverse_lipschitz_constraint} \text{ holds } \forall z_1, z_2 \in \mathrm{supp}( p_{Z, \delta}) \}, \\
    (1-L^2) \|\psi(z_1) - \psi(z_2) \|^2 + \|z_1 - z_2\|^2 &\leq 2 \langle z_1 - z_2, \psi(z_1) - \psi(z_2)\rangle.
    \label{eq:inverse_lipschitz_constraint}
\end{align}
The set $\Psi_L^\delta$ is closed and convex in $ L_{p_{Z, \delta}}^2 (X,X)$ for any $L<1$ and $\delta>0$. Consequently, the minimization problem in Lemma \ref{lem:reco training equivalent formulation} is well-defined and admits a unique solution. This solution, w.r.t.\ $\Psi_L^\delta$, $p_{Z,\delta}$ and $\hat{\psi}_\delta$, is denoted by $\psi_{L, \delta}^*$.

The parameter $L$ controls the stability of the elements from $\Psi_L^\delta$ and can therefore be interpreted as a regularization parameter. The question remains whether we can also obtain a convergence result, i.e., $\psi^*_{L, \delta}(A x^\dagger + \eta^\delta) \to x^\dagger$ for $\|\eta^\delta\| \leq \delta$, $\delta \to 0$ (vanishing noise) and $L \to 1$, analogous to the convergence theorems of classical methods like Tikhonov's. In the following remark, we discuss difficulties and supporting facts regarding a convergence result.

\begin{remark}\label{rem:convergence reco}
There are different results in the literature for the posterior distribution to converge to one single point $x^\dagger$ (or its respective delta distribution) \cite[Theorem 1 and 2]{bochkina_2013}, \cite{vollmer_2013}. Thus, it might be realistic to expect that the conditional mean estimator $\hat{\psi}_\delta(Ax^\dagger)$ also converges to $x^\dagger$.
In classical convergence theorems \cite{engl1996regularization}, the ground truth $x^\dagger$ is mostly assumed to be a minimum-norm-solution of \eqref{eq:inverse_problem}. 
However, in a Bayesian setting with a learned reconstruction scheme, a criterion based on $p_X$ for characterizing $x^\dagger$ is more appropriate, e.g., 
\begin{equation}
    x^\dagger = \mathrm{arg} \max_{Ax = z} p_X(x) \qquad \text{for } z \in \mathcal{R}(A) 
\end{equation}
as in \cite{bochkina_2013} or the one provided in the subsequent Lemma~\ref{lem:solution_reco_train_example}.

Since $\psi_{L, \delta}^*$ is the projection of $\hat{\psi}_\delta$ onto the set $\Psi_L^\delta$ (Lemma \ref{lem:reco training equivalent formulation}), it is not unlikely that $\psi_{L, \delta}^*(Ax^\dagger)$ also converges to $x^\dagger$. However, the projection is w.r.t.\ the $L_{p_{z,\delta}}^2$-norm, and we are actually interested in pointwise convergence.
Besides, all functions $\psi$ from the sets $\Psi_L^\delta$ are Lipschitz continuous and they fulfill the monotonicity condition
\begin{equation} \label{eq:psi_monotonicity}
    \|z_1 - z_2\|^2 \leq 2 \langle z_1 - z_2, \psi(z_1) - \psi(z_2) \rangle.
\end{equation}
Thus, one cannot approximate arbitrary $L_{p_{z,\delta}}^2$-functions.

There are two facts that partly address this difficulty. First, the posterior mean $\hat{\psi}_\delta$ is also a Lipschitz continuous function under certain assumptions \cite[Theorem 4.5, Remark 4.6]{dashti_2017}. Second, 
a generalized inverse on $\mathcal{R}(A)$ of the form $A^\dagger \colon z \mapsto x^\dagger$ would indeed fulfill \eqref{eq:psi_monotonicity}.
This follows from $A$ being self-adjoint and positive semidefinite (we can write $A = \tilde{A}^* \tilde{A}$) and $\|A\|=1$ by
\begin{align}
    \|z_1 - z_2\|^2 = \|Ax_1^\dagger - A x_2^\dagger\|^2 &\leq \|\tilde{A}^*\|^2 \| \tilde{A} (x_1 - x_2) \|^2 = \langle \tilde{A} (x_1^\dagger - x_2^\dagger), \tilde{A} (x_1^\dagger - x_2^\dagger) \rangle \notag \\ 
    &\leq \langle A (x_1^\dagger - x_2^\dagger), x_1^\dagger - x_2^\dagger, \rangle = \langle z_1 - z_2, A^\dagger z_1 - A^\dagger z_2, \rangle.
\end{align}

Assuming that $\psi_{L, \delta}^*$ converges pointwise to $A^\dagger$ on $\mathcal{R}(A)$ would be sufficient to obtain a convergence result for noisy data as well. If $\|z^\delta - z\| \leq \delta$ holds and one chooses $L$ in a way that $L \to 1$ and $\frac{\delta}{1-L} \to 0$ for $\delta \to 0$, the desired convergence
\begin{equation}
    \begin{split}
    \| \psi_{L, \delta}^*(z^\delta) - A^\dagger(z) \| &\leq \| \psi_{L, \delta}^*(z^\delta) - \psi_{L, \delta}^*(z) \| + \| \psi_{L, \delta}^*(z) - A^\dagger(z)\| \\
    &\leq \frac{1}{1-L} \|z^\delta - z \| + \| \psi_{L, \delta}^*(z) - A^\dagger(z) \| \to 0
    \end{split}
\end{equation}
would follow, since $\mathrm{Lip}(\psi_{L, \delta}^*) \leq \frac{1}{1-L}$ \cite[Lemma 2.1]{iresnet_01_regtheory}.



\end{remark}

In the following, we provide a result for a potential candidate for $x^\dagger$ for the convergence analysis. 

\begin{lemma}\label{lem:solution_reco_train_example}
 Let Assumption \ref{ass:prior_and_noise} hold with $p_{H,\delta}=p_H$ indicating the dependence on $\delta$. In addition, let $\hat{\psi}_\delta: X \to X$, $z^\delta \mapsto \mathbb{E}(x|z^\delta)$ be the conditional mean estimator with 
 \begin{equation}
 p(x|z^\delta)=        \frac{p_{H,\delta}(z^\delta - Ax)\, p_X(x)}{p_{Z,\delta}(z^\delta)}
 \end{equation}
 and $    p_{Z,\delta}(z^\delta) = \int_X p_X(x)\, p_{H,\delta}(z^\delta - Ax) \, \mathrm{d}x$. We further make the following assumptions:
 \begin{itemize}
     \item[(i)] Noise on $\mathcal{R}(A)^\bot=\mathcal{N}(A)$ and $\mathcal{R}(A)=\mathcal{N}(A)^\bot$ (as $A=A^\ast$ and $X$ is finite-dimensional) is stochastically independent, i.e., there exist pdfs $p^{\dagger}_{H,\delta}: \mathcal{N}(A)^\bot\to \R_{\geq 0}$ and $p^{0}_{H,\delta}: \mathcal{N}(A) \to \R_{\geq 0}$ such that $p_{H,\delta}(\eta)=p^{0}_{H,\delta}(P_{\mathcal{N}(A)}\eta) \cdot p^{\dagger}_{H,\delta}(P_{\mathcal{N}(A)^\bot}\eta)$,
     \item[(ii)] $p^{0}_{H,\delta}$ and $p^{\dagger}_{H,\delta}$ define Dirac sequences with respect to $\delta \to 0$,
     \item[(iii)] $p_X$ is compactly supported and continuous. We define $\mathcal{R}_{p_X}(A):=\{ Ax\, |\, x \in \mathrm{supp}(p_X)\} \subset \mathcal{R}(A)  $,  
     \item[(iv)] For any $z\in \mathcal{R}_{p_X}$ there exists a $\bar{\delta}$ such that for all $\delta \in (0,\bar{\delta}]$ it holds $z\in \mathrm{supp}(p_{Z,\delta})$.
 \end{itemize}

We then have pointwise convergence of $\hat{\psi}_\delta$ for $z\in \mathcal{R}_{p_X}(A)$ such that it holds

\begin{equation}
   \hat{\psi}_\delta(z)  \underset{\delta \to 0}{ \longrightarrow} A^\dagger z + \int_{\mathcal{N}(A)} p(x_0|A^\dagger z)\,  x_0 \ \mathrm{d} x_0, \quad \text{with } \quad p(x_0|A^\dagger z)=\frac{p_X(x_0 + A^\dagger z)}{\int_{\mathcal{N}(A)} p_X(x'_0 + A^\dagger z)\ \mathrm{d} x'_0 }
\end{equation}
i.e., it converges to the minimum-norm solution $A^\dagger z$ plus the conditional expectation $\mathbb{E}(x_0|A^\dagger z) \in \mathcal{N}(A)$ in the nullspace. 
\end{lemma}
\begin{proof}
    The proof can be found in Appendix~\ref{appendix:proof_solution_type}.
\end{proof}

\subsection*{Reconstruction training for diagonal architecture}
In order to derive more specific results for the minimizer of \eqref{eq:reco_training_integral}, we make the assumption of stochastic independence of the components $x_j = \langle x, v_j \rangle \sim p_{X,j}$ and $\eta_j = \langle \eta, v_j \rangle \sim p_{H,j}$ similar to the setting in the last section on the approximation training:
\begin{assumption}
    Let $p_X(x)=\prod_j p_{X,j}(x_j)$ and $p_H(\eta)=\prod_j p_{H,j}(\eta_j)$.
\end{assumption}
Observe that the first and second moments of $p_{X,j}$ and $p_{H,j}$ exist due to Assumption~\ref{ass:prior_and_noise} with
\begin{equation}
    \mu_{X,j} = \int_\R p_{X,j}(x_j)\, x_j \, \mathrm{d}x_j \qquad \text{and} \qquad \mu_{H,j}=\int_\R p_{H,j}(\eta_j)\, \eta_j \, \mathrm{d}\eta_j = 0
\end{equation}
for all $j\in\N$. In addition, the density of $z_j = \sigma_j^2 x_j + \eta_j$ is given by 
\begin{equation}
    p_{Z,j}(z_j) = \int_\R p_{X,j}(x_j)\, p_{H,j}(z_j - \sigma_j^2 x_j) \, \mathrm{d}x_j
\end{equation}
with 
\begin{equation}
    \mu_{Z,j} = \int_\R p_{Z,j}(z_j)\, z_j \, \mathrm{d}z_j = \int_\R \int_\R z_j\,  p_{X,j}(x_j)\, p_{H,j}(z_j - \sigma_j^2 x_j) \, \mathrm{d}x_j \, \mathrm{d}z_j = \sigma_j^2\, \mu_{X,j}.
\end{equation}
Analogously to the approximation training, in this setting, it is sufficient to consider a diagonal structure of $\varphi$, which then implies the same structure for $\psi$, i.e.,\
\begin{equation}
    \psi(z) = \sum_j \psi_j(\langle z,v_j\rangle)v_j
\end{equation}
and the resulting optimization problem to train each $\psi_j$ now reads
\begin{equation} \label{eq:reco_training_integral_1_D}
     \min_{\psi_j \in \Psi_j} \int_\R \int_\R p_{X,j}(x_j)\, p_{H,j}(\eta_j)\, \|\psi_j(\sigma_j^2 x_j + \eta_j) - x_j\|^2 \, \mathrm{d}\eta_j \, \mathrm{d}x_j
\end{equation}
with a suitable set of functions $\Psi_j$. Recall that $\psi_j:\R\to\R$ represents the inverse of an iResNet 
if $\psi_j$ satisfies
\begin{equation}\label{eq:invertability condition 1d}
    \frac{1}{1+L} \leq  \frac{\psi_j(z_1) - \psi_j(z_2)}{z_1 - z_2} \leq  \frac{1}{1-L} \qquad \forall z_1, z_2 \in \R
\end{equation}
for some $0\leq L < 1$, c.f.\ Remark~\ref{rem:invertability_1_d}.
Therefore, we consider the set
\begin{equation} \label{Eq:Psi 1 D}
        \Psi_j = \left\{ \psi_j:\R \to \R \, \Big|\, \psi_j(z_j) = m z_j + b\text{ for } z_j\in\R \text{ with } m\in\R, \frac{1}{1+L}\leq m\leq \frac{1}{1-L}, b\in\R  \right\}
\end{equation}
for some $0\leq L < 1$. The following lemma provides a formula for the minimizer of problem~\eqref{eq:reco_training_integral_1_D}. For better readability, we leave out the index $j$ in the subsequent derivations.

\begin{lemma}\label{lem:reco training solution linear functions} The unique solution to the minimization problem~\eqref{eq:reco_training_integral_1_D} with $\Psi_j$ as in \eqref{Eq:Psi 1 D} is given by
    \begin{align}
        \psi^\ast(z) = m\,z + (1 - \sigma^2 m)\,\mu_X \quad \text{for } z\in\R
    \end{align}
    with
    \begin{align}
        m =   \begin{cases}
                        \frac{1}{1+L} &\text{if } \frac{\sigma^2\text{Var}_{p_X}(x)}{\sigma^4\text{Var}_{p_X}(x) + \text{Var}_{p_H}(\eta)} < \frac{1}{1+L},  \\
                        \frac{\sigma^2\text{Var}_{p_X}(x)}{\sigma^4\text{Var}_{p_X}(x) + \text{Var}_{p_H}(\eta)} &\text{if } \frac{1}{1+L}\leq \frac{\sigma^2\text{Var}_{p_X}(x)}{\sigma^4\text{Var}_{p_X}(x) + \text{Var}_{p_H}(\eta)} \leq \frac{1}{1-L},  \\
                        \frac{1}{1-L} &\text{if } \frac{\sigma^2\text{Var}_{p_X}(x)}{\sigma^4\text{Var}_{p_X}(x) + \text{Var}_{p_H}(\eta)} > \frac{1}{1-L}.
                    \end{cases}
    \end{align}
\end{lemma}

\begin{proof}
We solve the minimization problem~\eqref{eq:reco_training_integral} by applying the KKT conditions. To this end, consider the Lagrange functional 
\begin{equation}
    K(m,b,\lambda_1, \lambda_2) = \int_\R \int_\R p_X(x)\, p_H(\eta)\, ( m(\sigma^2 x + \eta) + b - x)^2 \, \mathrm{d}\eta \, \mathrm{d}x +  \lambda_1\left(m-\frac{1}{1-L}\right) + \lambda_2\left(\frac{1}{1+L}-m\right)
\end{equation}
for $m,b,\lambda_1, \lambda_2\in\R$. Observe that the integral is well-defined as the first and second moments of $p_X$ and $p_H$ exist by assumption. Moreover, $K$ is convex and coercive w.r.t.\ $(m,b)$ and the integral is continuous w.r.t.\ $(m,b)$. Consequently, there exists a minimizer of problem~\eqref{eq:reco_training_integral} which satisfies the necessary KKT conditions
\begin{align} \label{eq:dLdmReco}
    \frac{\partial K}{\partial m} (m,b,\lambda_1,\lambda_2) = \int_\R \int_\R 2\,p_X(x)\, p_H(\eta)\, (\sigma^2 x + \eta)\,( m(\sigma^2 x + \eta) + b - x) \, \mathrm{d}\eta \, \mathrm{d}x + \lambda_1 m - \lambda_2 m \stackrel{!}{=} 0, \\
    \label{eq:dLdbReco}
    \frac{\partial K}{\partial b} (m,b,\lambda_1, \lambda_2) = \int_\R \int_\R 2\,p_X(x)\, p_H(\eta) \,( m(\sigma^2 x + \eta) + b - x) \, \mathrm{d}\eta \, \mathrm{d}x  \stackrel{!}{=} 0,\\
    \label{eq:lagrange_constraint_1}
    \lambda_1\left(m-\frac{1}{1-L}\right) \stackrel{!}{=} 0, \\
    \label{eq:lagrange_constraint_2}
    \lambda_2\left(\frac{1}{1+L}-m\right) \stackrel{!}{=} 0, \\
    \lambda_1, \lambda_2 \geq  0.
\end{align}
\textit{Case $\lambda_1 =\lambda_2 = 0$}: Equation~\eqref{eq:dLdmReco} implies
\begin{align}\label{eq:mForLambdaZero}
    \int_\R \int_\R &p_X(x)\, p_H(\eta)\, (\sigma^2 x + \eta)\,( m(\sigma^2 x + \eta) + b - x) \, \mathrm{d}\eta \, \mathrm{d}x \nonumber \\ 
    &=\int_\R \int_\R p_X(x)\, p_H(\eta)\, ( m\sigma^4 x^2 + 2 m \sigma^2 x \eta + m \eta^2 + b \sigma^2 x + b\eta-\sigma^2x^2-\eta x) \, \mathrm{d}\eta \, \mathrm{d}x \nonumber \\ 
    &= m\sigma^4 \mathbb{E}_{p_X}(x^2) + 2m\sigma^2 \mu_X \mu_H + m \mathbb{E}_{p_H}(\eta^2) + b \sigma^2 \mu_X + b \mu_H - \sigma^2\mathbb{E}_{p_X}(x^2) - \mu_H\mu_X \nonumber  \\
    &= m\sigma^4 \mathbb{E}_{p_X}(x^2) + m \mathbb{E}_{p_H}(\eta^2) + b \sigma^2 \mu_X - \sigma^2\mathbb{E}_{p_X}(x^2) \nonumber \\
    &\stackrel{!}{=} 0
\end{align}
and Equation~\eqref{eq:dLdbReco} gives
\begin{align}
    \int_\R \int_\R 2\,p_X(x) p_H(\eta) \,( m(\sigma^2 x + \eta) + b - x) \, \mathrm{d}\eta \, \mathrm{d}x = m\sigma^2 \mu_X + m \mu_H + b - \mu_X \stackrel{!}{=} 0
\end{align}
resulting in 
\begin{align} \label{eq:bForLambdaZero}
    b = \mu_X - m\sigma^2\mu_X.
\end{align}
Inserting the last equation into Equation~\eqref{eq:mForLambdaZero} yields
\begin{align}
    &m\sigma^4 \mathbb{E}_{p_X}(x^2) + m \mathbb{E}_{p_H}(\eta^2) +  \sigma^2 \mu_X^2 - m\sigma^4\mu_X^2 - \sigma^2\mathbb{E}_{p_X}(x^2) \stackrel{!}{=} 0 \notag \\
    \Leftrightarrow \quad &m\sigma^4\text{Var}_{p_X}(x) + m \text{Var}_{p_H}(\eta)-\sigma^2\text{Var}_{p_X}(x) = 0 \notag \\
    \Leftrightarrow \quad & m = \frac{\sigma^2\text{Var}_{p_X}(x)}{\sigma^4\text{Var}_{p_X}(x) + \text{Var}_{p_H}(\eta)}
\end{align}
and 
\begin{align}
    b = \left( 1 - \frac{\sigma^4\text{Var}_{p_X}(x)}{\sigma^4\text{Var}_{p_X}(x) + \text{Var}_{p_H}(\eta)} \right) \mu_X.
\end{align}
The formulas for $m$ and $b$ correspond to the unconstrained solution of problem~\eqref{eq:reco_training_integral}. In order to satisfy the constraint on $m$, we need to guarantee that
\begin{equation}
    \frac{1}{1+L}\leq \frac{\sigma^2\text{Var}_{p_X}(x)}{\sigma^4\text{Var}_{p_X}(x) + \text{Var}_{p_H}(\eta)} \leq \frac{1}{1-L}.
\end{equation}
If this is not satisfied, we must require $\lambda_1 > 0$ or $\lambda_2 > 0$, which we will deal with in the following paragraphs.\\
\textit{Case $\lambda_1 > 0,\, \lambda_2 = 0$}:
If $\lambda_1>0$, Equation~\eqref{eq:lagrange_constraint_1} directly yields
\begin{equation}
    m = \frac{1}{1-L}.
\end{equation}
For $b$ we again obtain
\begin{equation}
    b = \mu_X - m\sigma^2\mu_X = \left(1 - \frac{\sigma^2}{1-L}\right)\mu_X
\end{equation}
as Equation~\eqref{eq:dLdbReco} is independent of $\lambda_1$ and $\lambda_2$. Furthermore, Equation~\eqref{eq:dLdmReco} implies
\begin{align}
    &\int_\R \int_\R p_X(x)\, p_H(\eta)\, (\sigma^2 x + \eta)\,( m(\sigma^2 x + \eta) + b - x) \, \mathrm{d}\eta \, \mathrm{d}x + \lambda_1 m \stackrel{!}{=} 0 \notag \\
    \Leftrightarrow \quad & \lambda_1 = 2\,\frac{\sigma^2\text{Var}_{p_X}(x)-m\sigma^4\text{Var}_{p_X}(x) - m \text{Var}_{p_H}(\eta)}{m}.
\end{align}
In combination with the condition $\lambda_1>0$ we now obtain
\begin{align}
    \frac{\sigma^2\text{Var}_{p_X}(x)}{\sigma^4\text{Var}_{p_X}(x) + \text{Var}_{p_H}(\eta)} > \frac{1}{1-L}.
\end{align}
\textit{Case $\lambda_1 = 0,\, \lambda_2 > 0$}: The same line of reasoning as in the preceding case yields
\begin{align}
    &m = \frac{1}{1+L}\\
    &b = \mu_X - m\sigma^2\mu_X = \left(1 - \frac{\sigma^2}{1+L}\right)\mu_X \\
    &\frac{\sigma^2\text{Var}_{p_X}(x)}{\sigma^4\text{Var}_{p_X}(x) + \text{Var}_{p_H}(\eta)} < \frac{1}{1+L}.
\end{align}
Observe that the case $\lambda_1>0$ and $\lambda_2>0$ is not possible as there is no $m$ satisfying Equations~\eqref{eq:lagrange_constraint_1} and \eqref{eq:lagrange_constraint_2} simultaneously for $0<L<1$. The uniqueness of the solution directly follows from the uniqueness of $m$ and $b$, which concludes the proof.
\end{proof}

Note that the inverse of the function $\psi^\ast:\R\to\R$ is given by $\varphi^\ast(x) = x-f^\ast(x)$ with
\begin{align}
    f^\ast(x) =   \begin{cases}
                    -L x + \left(1 +L - \sigma^2\right)\mu_X &\text{if } \frac{\sigma^2\text{Var}_{p_X}(x)}{\sigma^4\text{Var}_{p_X}(x) + \text{Var}_{p_H}(\eta)} < \frac{1}{1+L}, \\
                    \left(1-\frac{\sigma^4\text{Var}_{p_X}(x) + \text{Var}_{p_H}(\eta)}{\sigma^2\text{Var}_{p_X}(x)}\right) x +  \frac{\text{Var}_{p_H}(\eta)}{\sigma^2\text{Var}_{p_X}(x)}  \, \mu_X &\text{if } \frac{1}{1+L}\leq \frac{\sigma^2\text{Var}_{p_X}(x)}{\sigma^4\text{Var}_{p_X}(x) + \text{Var}_{p_H}(\eta)} \leq \frac{1}{1-L}, \\
                    L x + \left(1 - L - \sigma^2\right)\mu_X &\text{if } \frac{\sigma^2\text{Var}_{p_X}(x)}{\sigma^4\text{Var}_{p_X}(x) + \text{Var}_{p_H}(\eta)} > \frac{1}{1-L}
                \end{cases}
\end{align}
for $x\in\R$. In the case of noiseless data, i.e.\ $\text{Var}_{p_H}(\eta)=0$, $f^\ast$ corresponds to the function $f^*$ derived in Lemma~\ref{lem:approx_training_linear} and Theorem~\ref{thm:approx_training_lipschitz}. 

The function $\psi^\ast$ plays an important role in the case of Gaussian prior and noise distributions, which we will deal with in the following corollary.

\begin{corollary}\label{cor:1D Reco Gaussian Distributions}
    Assume that $p_X: \R\to\R$ and $p_H:\R\to\R$ are Gaussian probability density functions. Then, the function $\psi^\ast$ of Lemma~\ref{lem:reco training solution linear functions} is a solution to the minimization problem~\eqref{eq:reco_training_integral_1_D} with 
    \begin{equation}
        \Psi = \left\{\psi \in C^1(\R)\cap L_{p_Z}^2(\R) \, \bigg| \,  \frac{1}{1+L}\leq \psi'(z)\leq \frac{1}{1-L} \text{ for all } z\in\R\right\}.
    \end{equation}
\end{corollary}
\begin{proof}
    In Lemma~\ref{lem:unconstrained_reco_training} we have seen that the unconstrained solution of problem~\eqref{eq:reco_training_integral_1_D} is given by $\hat{\psi}(z) = \mathbb{E}(x|z)$ for all $z\in\R$. In the case of Gaussian noise and prior distributions, $\mathbb{E}(x|z)$ can be expressed as
    \begin{equation}
        \mathbb{E}(x|z) = \frac{\sigma^2\text{Var}_{p_X}(x)}{\sigma^4\text{Var}_{p_X}(x) + \text{Var}_{p_H}(\eta)} z + \left( 1 - \frac{\sigma^4\text{Var}_{p_X}(x)}{\sigma^4\text{Var}_{p_X}(x) + \text{Var}_{p_H}(\eta)} \right) \mu_X
    \end{equation}
    for all $z\in\R$ \cite[Theorem 6.20 and Eq.(2.16a)]{stuart_2010}, which is an element of $C^1(\R)\cap L_{p_Z}^2(\R)$. In combination with Lemma~\eqref{lem:reco training equivalent formulation}, minimization problem~\eqref{eq:reco_training_integral} can be rewritten as
    \begin{equation}
        \min_{\psi \in \Psi} \int_\R p_Z(z) \left( \psi(z) - \frac{\sigma^2\text{Var}_{p_X}(x)}{\sigma^4\text{Var}_{p_X}(x) + \text{Var}_{p_H}(\eta)} z - \left( 1 - \frac{\sigma^4\text{Var}_{p_X}(x)}{\sigma^4\text{Var}_{p_X}(x) + \text{Var}_{p_H}(\eta)} \right)\mu_X \right)^2 \, \mathrm{d}z.
    \end{equation}
    The same reasoning as in the proof of Theorem~\ref{thm:approx_training_lipschitz} now shows that $\psi^\ast$ of Lemma~\ref{lem:reco training solution linear functions} is a solution to the minimization problem.
\end{proof}

Figure~\ref{fig:1d Reco Gaussian Distribution} illustrates the behavior of the unconstrained solution $\hat{\psi}$ and the constrained solution $\psi^\ast$ in case of Gaussian probability density functions for varying noise and small singular values ($1-\sigma^2 > L$). Note that both solutions can be rewritten to depend on $\mu_Z$ instead of on $\mu_X$ using $\mu_Z = \sigma^2 \mu_X$. It can be observed that the noise level affects the slope of the unconstrained solution, with decreasing values at higher noise levels. Thus, $\hat{\psi}$ violates the invertibility condition \eqref{eq:invertability condition 1d} for very small and very large values of $\text{Var}_{p_H}(\eta)$ leading to $\psi^\ast \neq \hat{\psi}$ in these cases (left and right image of Figure~\ref{fig:1d Reco Gaussian Distribution}). 

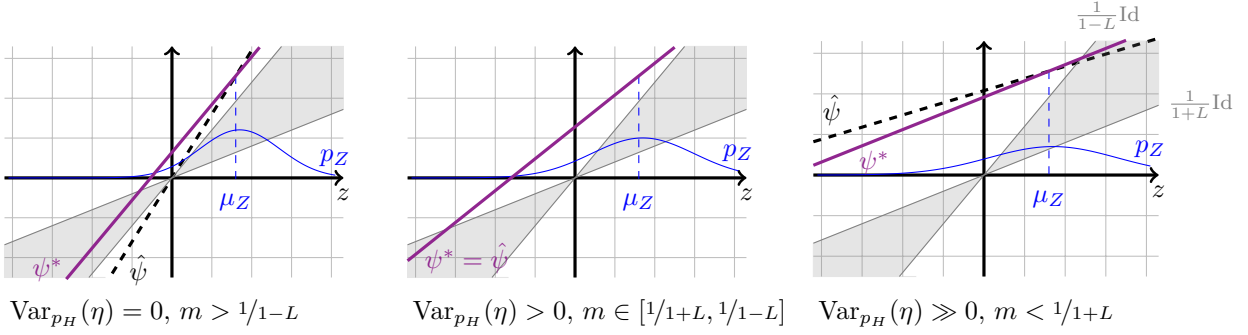
\begin{figure}
 \begin{minipage}[c]{0.33\textwidth}
\begin{center}
\begin{subfigure}[b]{\textwidth}
\begin{tikzpicture}[scale=0.53]
  \draw[step=1.0, lightgray, thin] (-4.2,-2.5) grid (4.3,3.3);
  \draw[->, thick, very thick] (-4.2, 0) -- (4.3, 0) node[below] {$z$};
  \draw[->, thick, very thick] (0, -2.5) -- (0, 3.3);
  
  \draw[scale=1, domain=-1.5:2.05, variable=\t, black, very thick, dashed] plot[samples=161] ({\t}, {(1.6)*\t});
  \node[black, right] at (-1.3,-2.3) {$\hat{\psi}$};
    
    \filldraw[gray, opacity=0.2] (-2.083,-2.5) -- (2.75,3.3) -- (4.3,3.3) -- (4.3,1.72) -- (-4.2,-1.68) -- (-4.2,-2.5) -- (-1.667,-2.5); 
    \draw[scale=1, domain=-2.083:2.75, variable=\t, gray] plot[samples=161] ({\t}, {(1.2)*\t});
   \phantom{\node[gray, right] at (2.5,3.8) {$\frac{1}{1-L}\Id$};} 
   \draw[scale=1, domain=-4.2:4.3, variable=\t, gray] plot[samples=161] ({\t}, {(0.4)*\t});
   
   \draw[scale=1, domain=-4.1:4.1, variable=\t, blue] plot[samples=161] ({\t}, {1.2*exp(-0.5*(\t-1.6)*(\t-1.8))}) node[above] {$p_{Z}$};
   \draw[, blue, dashed] (1.6,2.56) -- (1.6,-0.16) node[below] {$\mu_Z$};
   
   \draw[scale=1, domain=-2.65:2.2, variable=\t, Plum, very thick] plot[samples=161] ({\t}, {(1.2)*\t+(1.6-1.2)*1.6}); 
   \node[Plum, left] at (-2.5,-2.3) {$\psi^\ast$};

\node[black,right] at (-4.3,-3.4) {$\text{Var}_{p_H}(\eta) = 0$, $m>\nicefrac{1}{1-L}$};
\end{tikzpicture}
\end{subfigure}
\end{center}
\end{minipage}
\hspace{-7pt}
\begin{minipage}[c]{0.33\textwidth}
\begin{center}
\begin{subfigure}[b]{\textwidth}
\begin{tikzpicture}[scale=0.53]
  \draw[step=1.0, lightgray, thin] (-4.2,-2.5) grid (4.3,3.3);
  \draw[->, thick, very thick] (-4.2, 0) -- (4.3, 0) node[below] {$z$};
  \draw[->, thick, very thick] (0, -2.5) -- (0, 3.3);
  
    
    \filldraw[gray, opacity=0.2] (-2.083,-2.5) -- (2.75,3.3) -- (4.3,3.3) -- (4.3,1.72) -- (-4.2,-1.68) -- (-4.2,-2.5) -- (-1.667,-2.5); 
    \draw[scale=1, domain=-2.083:2.75, variable=\t, gray] plot[samples=161] ({\t}, {(1.2)*\t});
   \phantom{\node[gray, right] at (2.5,3.8) {$\frac{1}{1-L}\Id$}; }
   \draw[scale=1, domain=-4.2:4.3, variable=\t, gray] plot[samples=161] ({\t}, {(0.4)*\t});
   
   \draw[scale=1, domain=-4.1:4.1, variable=\t, blue] plot[samples=161] ({\t}, {1*exp(-0.3*(\t-1.6)*(\t-1.8))}) node[above] {$p_{Z}$};
   \draw[, blue, dashed] (1.6,2.56) -- (1.6,-0.16) node[below] {$\mu_Z$};
   
   \draw[scale=1, domain=-4.2:2.5, variable=\t, Plum, very thick] plot[samples=161] ({\t}, {(0.8)*\t+(1.6-0.8)*1.6}); 
   \node[Plum, right] at (-4,-2) {$\psi^\ast = \hat{\psi}$};

\node[black,right] at (-4.3,-3.4) {$ \text{Var}_{p_H}(\eta) > 0$, $m\in [\nicefrac{1}{1+L},\nicefrac{1}{1-L}]$};
\end{tikzpicture}
\end{subfigure}
\end{center}
\end{minipage}
\hspace{-6pt}
\begin{minipage}[c]{0.4\textwidth}
\vspace{-3pt}
\begin{center}
\begin{subfigure}[b]{\textwidth}
\begin{tikzpicture}[scale=0.54]
  \draw[step=1.0, lightgray, thin] (-4.2,-2.5) grid (4.3,3.3);
  \draw[->, thick, very thick] (-4.2, 0) -- (4.3, 0) node[below] {$z$};
  \draw[->, thick, very thick] (0, -2.5) -- (0, 3.3);
  
  \draw[scale=1, domain=-4.2:4.3, variable=\t, black, very thick, dashed] plot[samples=161] ({\t}, {(0.3)*\t+(1.6-0.3)*1.6});
  \node[black, right] at (-4.2,1.6) {$\hat{\psi}$};
    
    \filldraw[gray, opacity=0.2] (-2.083,-2.5) -- (2.75,3.3) -- (4.3,3.3) -- (4.3,1.72) -- (-4.2,-1.68) -- (-4.2,-2.5) -- (-1.667,-2.5); 
    \draw[scale=1, domain=-2.083:2.75, variable=\t, gray] plot[samples=161] ({\t}, {(1.2)*\t});
   \node[gray, right] at (2,3.9) {\footnotesize$\frac{1}{1-L}\Id$}; 
   \draw[scale=1, domain=-4.2:4.3, variable=\t, gray] plot[samples=161] ({\t}, {(0.4)*\t});
   \node[gray, right] at (4.3,1.72) {\footnotesize$\frac{1}{1+L}\Id$};
   
   \draw[scale=1, domain=-4.1:4.1, variable=\t, blue] plot[samples=161] ({\t}, {0.7*exp(-0.2*(\t-1.6)*(\t-1.8))}) node[above] {$p_{Z}$};
   \draw[, blue, dashed] (1.6,2.56) -- (1.6,-0.16) node[below] {$\mu_Z$};
   
   \draw[scale=1, domain=-4.2:3.5, variable=\t, Plum, very thick] plot[samples=161] ({\t}, {(0.4)*\t+(1.6-0.4)*1.6}); 
   \node[Plum, right] at (-3.3,0.3) {$\psi^\ast$};

\node[black,right] at (-4.3,-3.4) { $ \text{Var}_{p_H}(\eta) \gg 0$, $m<\nicefrac{1}{1+L}$};
\end{tikzpicture}
\end{subfigure}
\end{center}
\end{minipage}
\caption{Illustration of the constrained solution $\psi^\ast$ and the unconstrained solution $\hat{\psi}$ in the case of Gaussian probability density functions $p_X$ and $p_H$, cf.\ Corollary~\ref{cor:1D Reco Gaussian Distributions}. The slope of the unconstrained solution $\hat{\psi}$ is denoted by $m$, i.e.\ $m = \frac{\sigma^2\text{Var}_{p_X}(x)}{\sigma^4\text{Var}_{p_X}(x) + \text{Var}_{p_H}(\eta)}$. The plots exemplify the behavior of $\psi^\ast$ and $\hat{\psi}$ for small singular values ($1-\sigma^2 > L$) assuming a fixed prior distribution $p_X$ but increasing variance of $p_H$. In the case that $\text{Var}_{p_H}(\eta) = 0$ (left), the slope of the unconstrained solution exceeds $\frac{1}{1-L}$. If the noise increases, the slope of the unconstrained solution decreases, resulting in $m\in [\frac{1}{1+L},\frac{1}{1-L}]$ (middle). For very noisy data, the slope of the unconstrained solution is smaller than $\frac{1}{1+L}$ (right), again resulting in $\psi^\ast\neq \hat{\psi}$. Observe that $\psi^\ast$ and $\hat{\psi}$ are equal to $\frac{1}{\sigma^2}\mu_Z = \mu_X$ for $z = \mu_Z = \sigma^2\mu_X$ in all cases.} \label{fig:1d Reco Gaussian Distribution}
\end{figure}

\paragraph{General behavior of $\psi^*$:} The previous results deal with special cases where either the architecture or the probability density functions are known. In order to derive more general results, we make use of the theory of optimal control. For this, we need to restrict ourselves to piecewise continuously differentiable functions $\psi$ with bounded domain, i.e., we consider the set
\begin{align}
 \Psi = \left\lbrace \psi \in C^0([z_0, z_1]) \,\bigg|\, \psi \text{ piecewise continuously differentiable with } \frac{1}{1+L}\leq \psi'(z)\leq \frac{1}{1-L} \right\rbrace 
\end{align}
with fixed $z_0, z_1\in\R$ and $\text{Pr}(z\leq z_0)\footnote{$\text{Pr}$ denotes the probability w.r.t.\ $z \sim p_Z$.} \leq \varepsilon$, $\text{Pr}(z\geq z_1)\leq \varepsilon$ for some small $\varepsilon>0$ to stay close to the previous setting. 
Furthermore, to apply the optimal control theory we need to split the optimization problem into two successive minimization problems. First, we minimize over all functions $\psi\in\Psi$ with fixed starting point $\psi(z_0) = \psi^0\in\R$. Then, the starting point minimizing the objective function is determined. In combination with Lemma~\ref{lem:reco training equivalent formulation}, the minimization problem thus reads
\begin{equation}\label{eq:overall minimization problem setting optimal control}
   \min_{\psi^0\in\R} \left( \min_{\psi \in \Psi\cap\{ \psi \,|\, \psi(z_0) = \psi^0 \}} \frac{1}{2} \int_{z_0}^{z_1}  p_Z(z) \vert \psi(z) - \hat{\psi}(z) \vert^2 \, \mathrm{d}z\right).
\end{equation}
We would like to stress that the minimization problem defined in Lemma~\ref{lem:reco training equivalent formulation} is not equivalent to the initial one of Equation~\eqref{eq:reco_training_integral} due to the bounded domain of $\psi$. However, this error is negligible for small $\varepsilon$ and the two minimization problems coincide if $\text{supp}(p_Z)\subset [z_0,z_1]$. 
\begin{remark}
    The restriction to a bounded domain of $\psi$ might seem artificial at first. Nevertheless, in applications, the dataset rarely contains samples belonging to low-density regions of $p_Z$, and thus, these cases are covered in our setting. 
\end{remark}

The inner minimization problem can be solved with the help of Pontryagin's maximum principle resulting in the following necessary and sufficient conditions for the derivative of $\psi$.
\begin{lemma}\label{lem:cond derivative pontryagin}
    Let $\psi_0\in \Psi\cap\{ \psi \,|\, \psi(z_0) = \psi^0 \}$ be a solution of the minimization problem 
    \begin{equation}\label{eq:min problem optimal control}
        \min_{\psi \in \Psi\cap\{ \psi \,|\, \psi(z_0) = \psi^0 \}} \frac{1}{2} \int_{z_0}^{z_1}  p_Z(z) \vert \psi(z) - \hat{\psi}(z) \vert^2 \, \mathrm{d}z.
    \end{equation}
    Then, in all points of differentiability, the derivative $\psi_0'$ must satisfy the necessary and sufficient conditions
    \begin{align}
        \psi_0'(z) =   \begin{cases}
                        \frac{1}{1+L} \quad &\text{if } \lambda(z)>0\\
                        f_0(z)\quad &\text{if } \lambda(z)=0\\
                        \frac{1}{1-L}\quad &\text{if } \lambda(z)<0
                    \end{cases} \quad \text{with } z\in [z_0, z_1]
    \end{align}
    for some $f_0:[z_0, z_1]\to\R$ satisfying
    \begin{equation}\label{eq:cond f pontryagin}
        \frac{1}{1+L}\leq f_0(z)\leq \frac{1}{1-L} \quad \forall\, z\in [z_0,z_1]
    \end{equation}
    and $\lambda:[z_0,z_1]\to\R$ with
    \begin{equation}
        \lambda'(z) = - p_Z(z) ( \psi_0(z) - \hat{\psi}(z))  \text{ and } \lambda(z_1)=0.
    \end{equation}
\end{lemma}
\begin{proof}
    Let us denote the set of all points $z\in [z_0,z_1]$ where $\psi_0$ is differentiable by $D$. 
    For Problem~\eqref{eq:min problem optimal control}, Pontryagin's maximum principle, c.f.\ \cite[*9.6 Theorem 1]{Luenberger1969optimization} and \cite[Theorem 1]{Seierstad1977optimalControl}, provides the necessary conditions
    \begin{align}
        &\psi_0'(z)=u_0(z) \quad \text{for all } z\in D \text{ and some piecewise continuous function } u_0:[z_0,z_1]\to \R \\
        &\frac{1}{1+L}\leq u_0(z)\leq \frac{1}{1-L} \quad \forall\, z\in [z_0,z_1] \label{eq:bi-lipschitz constraint derivative 2} \\
        &\lambda'(z) = - p_Z(z) ( \psi_0(z) - \hat{\psi}(z))  \text{ with } \lambda(z_1)=0 \label{eq:derivative lambda} \\
        &H(\psi_0,u_0,\lambda,z)\leq H(\psi_0,u,\lambda,z) \quad \forall\, u \text{ satisfying } \eqref{eq:bi-lipschitz constraint derivative 2} \label{eq:Hamiltonian}
    \end{align}
    with the Hamiltonian function
    \begin{equation}
         H(\psi,u,\lambda,z) = \lambda(z)\, u(z) + \frac{1}{2}\, p_Z(z)\, \vert \psi(z) - \hat{\psi}(z) \vert^2.
    \end{equation}
    Condition~\eqref{eq:Hamiltonian} is equivalent to setting
    \begin{equation}
        u_0(z) =   \begin{cases}
                        \frac{1}{1+L} \quad &\text{if } \lambda(z)>0\\
                        f_0(z)\quad &\text{if } \lambda(z)=0\\
                        \frac{1}{1-L}\quad &\text{if } \lambda(z)<0
                    \end{cases}
        \qquad \text{for some function } f_0 \text{ satisfying } \eqref{eq:cond f pontryagin}.
    \end{equation}
     Furthermore, For a function $\psi_0$ satisfying the conditions of Pontryagin's maximum principle to be a solution of Problem~\eqref{eq:min problem optimal control}, the Hamiltonian needs to be jointly convex in $\psi$ and $u$ and the constrained set defined by Equation~\eqref{eq:bi-lipschitz constraint derivative 2} needs to be convex, cf.\ \cite[Theorem 2]{Seierstad1977optimalControl}. Both of these conditions are satisfied in our setting and thus, the proof is complete.  
\end{proof}

We would like to remark that $\lambda$ can be expressed as 
\begin{equation}
    \lambda(z) = \int_{z}^{z_1} p_Z(\tilde{z}) ( \psi_0(\tilde{z}) - \hat{\psi}(\tilde{z}))  \, \mathrm{d}\tilde{z}.
\end{equation}
whenever $p_Z$ and $\hat{\psi}$ are continuous. To illustrate the previous lemma let us look at a very simple example. Assume that $\hat{\psi}'(z) > \nicefrac{1}{1-L}$ for all $z\in [z_0,z_1]$. Then, Lemma~\ref{lem:cond derivative pontryagin} in combination with a minimization over the starting points $\psi^0$ shows that a solution to problem~\eqref{eq:overall minimization problem setting optimal control} is given by 
\begin{equation}
    \psi^\ast(z) = \frac{1}{1-L} z + \frac{\int_{z_0}^{z_1} p_Z(\tilde{z}) \hat{\psi}(\tilde{z})\,\mathrm{d}\tilde{z} - \nicefrac{1}{1-L} \int_{z_0}^{z_1} p_Z(\tilde{z}) \tilde{z} \,\mathrm{d}\tilde{z}}{\int_{z_0}^{z_1} p_Z(\tilde{z}) \,\mathrm{d}\tilde{z}}\quad \text{for } z\in [z_0,z_1].
\end{equation}
\begin{figure}[t]
    \begin{center}
    \begin{tikzpicture}[scale=0.55]
      \draw[step=1.0, lightgray, thin] (-6.2,-4.5) grid (6.3,5.3);
      \draw[->, thick, very thick] (-6.2, 0) -- (6.3, 0) node[below] {$z$};
      \draw[->, thick, very thick] (0, -4.5) -- (0, 5.3);

      \draw[-, thick] (0.8, -0.1) -- (0.8, 0.1)node[below] {$k$};
      \draw[-, thick] (-0.8, -0.1) -- (-0.8, 0.1)node[below] {-$k$};
      
      \draw[scale=1, domain=-4.2:5.1, variable=\t, black, very thick, dashed] plot[samples=161] ({\t}, {(2/(0.5^2+(0.9*1)^2)^(3/2)*(0.9*1^2 *\t + 4*0.5^2)*1/(sqrt(2*pi))*exp(-0.5*(\t-4*0.9)^2/(0.5^2+(0.9*1)^2)) + 2/(0.5^2+(0.9*1)^2)^(3/2)*(0.9*1^2 *\t - 4*0.5^2)*1/(sqrt(2*pi))*exp(-0.5*(\t+4*0.9)^2/(0.5^2+(0.9*1)^2)))/(2*1/(sqrt(2*pi)*sqrt(0.5^2+(0.9*1)^2))*exp(-0.5*(\t-4*0.9)^2/(0.5^2+(0.9*1)^2)) + 2*1/(sqrt(2*pi)*sqrt(0.5^2+(0.9*1)^2))*exp(-0.5*(\t+4*0.9)^2/(0.5^2+(0.9*1)^2)))});
      \node[black, right] at (-1,-2) {$\hat{\psi}$};
        
        \filldraw[gray, opacity=0.2] (-3.46,-4.5) -- (4.07,5.3) -- (6.3,5.3) -- (6.3,1.26) -- (-6.2,-1.24) -- (-6.2,-4.5); 
        \draw[scale=1, domain=-3.46:4.07, variable=\t, gray] plot[samples=161] ({\t}, {1.3*\t});
       \node[gray, right] at (3.3,5.85) {$\frac{1}{1-L}\Id$}; 
       \draw[scale=1, domain=-6.2:6.3, variable=\t, gray] plot[samples=161] ({\t}, {0.2*\t});
       \node[gray, right] at (6.3,1.3) {$\frac{1}{1+L}\Id$};

       \draw[scale=1, domain=-6.1:6.1, variable=\t, blue] plot[samples=161] ({\t}, {6*1/(sqrt(2*pi)*sqrt(0.5^2+(0.9*1)^2))*exp(-0.5*(\t-4*0.9)^2/(0.5^2+(0.9*1)^2)) + 6*1/(sqrt(2*pi)*sqrt(0.5^2+(0.9*1)^2))*exp(-0.5*(\t+4*0.9)^2/(0.5^2+(0.9*1)^2))}) node[above] {$p_Z$}; 
       
       \draw[scale=1, domain=-4.2:-2.1, variable=\t, Plum, very thick] plot[samples=161] ({\t}, {(2/(0.5^2+(0.9*1)^2)^(3/2)*(0.9*1^2 *\t + 4*0.5^2)*1/(sqrt(2*pi))*exp(-0.5*(\t-4*0.9)^2/(0.5^2+(0.9*1)^2)) + 2/(0.5^2+(0.9*1)^2)^(3/2)*(0.9*1^2 *\t - 4*0.5^2)*1/(sqrt(2*pi))*exp(-0.5*(\t+4*0.9)^2/(0.5^2+(0.9*1)^2)))/(2*1/(sqrt(2*pi)*sqrt(0.5^2+(0.9*1)^2))*exp(-0.5*(\t-4*0.9)^2/(0.5^2+(0.9*1)^2)) + 2*1/(sqrt(2*pi)*sqrt(0.5^2+(0.9*1)^2))*exp(-0.5*(\t+4*0.9)^2/(0.5^2+(0.9*1)^2)))});
       \draw[scale=1, domain=-2.1:2.1, variable=\t, Plum, very thick] plot[samples=161] ({\t}, {1.3*\t});
       \draw[scale=1, domain=2.1:5.1, variable=\t, Plum, very thick] plot[samples=161] ({\t},{(2/(0.5^2+(0.9*1)^2)^(3/2)*(0.9*1^2 *\t + 4*0.5^2)*1/(sqrt(2*pi))*exp(-0.5*(\t-4*0.9)^2/(0.5^2+(0.9*1)^2)) + 2/(0.5^2+(0.9*1)^2)^(3/2)*(0.9*1^2 *\t - 4*0.5^2)*1/(sqrt(2*pi))*exp(-0.5*(\t+4*0.9)^2/(0.5^2+(0.9*1)^2)))/(2*1/(sqrt(2*pi)*sqrt(0.5^2+(0.9*1)^2))*exp(-0.5*(\t-4*0.9)^2/(0.5^2+(0.9*1)^2)) + 2*1/(sqrt(2*pi)*sqrt(0.5^2+(0.9*1)^2))*exp(-0.5*(\t+4*0.9)^2/(0.5^2+(0.9*1)^2)))});
       \node[Plum, right] at (-2.1,-1.3) {$\psi^\ast$};
    
    \end{tikzpicture}
    \caption{Behavior of the solution $\psi^\ast$ of \eqref{eq:overall minimization problem setting optimal control} in the case that $p_z$ can be represented as a Gaussian mixture, c.f.\ Remark~\ref{rem:posterior expectation gmm}. Observe that the slope of the unconstrained solution $\hat{\psi}$ exceeds $\nicefrac{1}{1-L}$ in the interval $[-k,k]$ resulting in $\lambda(z)<0$ for $z\in[-k-\varepsilon,k+\varepsilon]$ with $\varepsilon>0$. As a result, $\psi^\ast$ is equal to $z\mapsto\nicefrac{1}{1-L}\,z$ for $z\in[-k-\varepsilon,k+\varepsilon]$ and to $\hat{\psi}$ outside this interval.
    }
    \label{fig:fig:1d reco general case}
    \end{center}
\end{figure}
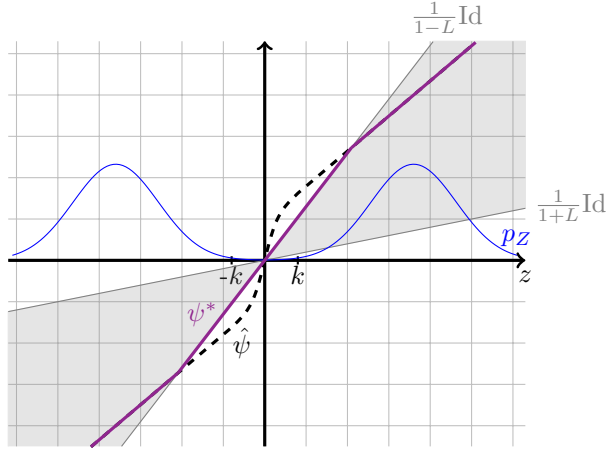
In addition, the general behavior of the solution to problem~\eqref{eq:overall minimization problem setting optimal control} is illustrated in Figure~\ref{fig:fig:1d reco general case}. This exemplifies that the best possible architecture choice, when considering the reconstruction training loss, is not necessarily an affine linear one unlike in the approximation training studied in the last section. This is partly because of the influence of the noise distribution $p_H$, which cancels out when using the approximation training loss. Moreover, the variance of the prior and noise distribution influences the best architecture and parameter choice, which is not the case in the approximation training setting where only the expectation of the prior distribution influences the solution.

\section{Numerical experiments}
\label{sec:numeric_experiments}

To study the implications of the previously developed theory for the practical application of iResNets for solving inverse problems, we perform experiments on two forward operators, where we compare approximation training~\eqref{eq:approx_training} to reconstruction training~\eqref{eq:reco_training}. In all experiments, we train single-layer iResNets with diagonal~\eqref{eq:diagonal} structure where the residual functions $f_\theta$ comprise multiple layers.

In the setting presented in the following sections, we consider a discrete convolution with a smoothing kernel $a\in\mathbb{R}^{9\times 9}$ that is depicted in Figure~\ref{fig:kernel} and zero padding to preserve dimensionality. Since the resulting Toeplitz matrix $M_a$ that performs the convolution with $a$ is symmetric and positive definite, this serves immediately as a self-adjoint operator $A=M_a$.
The second inverse problem we aim to solve is given by a discrete Radon operator $\tilde{A}:\mathbb{R}^{28\times28}\to\mathbb{R}^{30\times41}$, such that $A=\tilde{A}^\ast \tilde{A}$, which is in line with the setting used in the prior work~\cite{iresnet_01_regtheory}. We restrict the discussion of the numerical results to the convolution operator, and for the sake of completeness, the results for the Radon operator are provided in Appendix~\ref{app:numerics_radon}.

\begin{figure}[t!]
    \centering
    \includegraphics[width=0.3\textwidth]{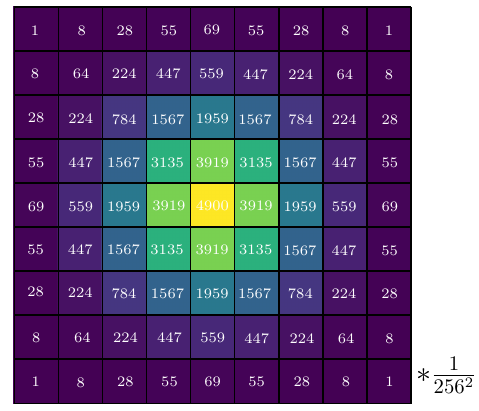}
    \caption{The filter kernel $a$ used in the convolution operator $M_a$.}
    \label{fig:kernel}
\end{figure}

In both cases, we train our models on the MNIST handwritten digits dataset \cite{lecun1998mnist}, where we treat images as flattened vectors in $\mathbb{R}^{28\cdot28}$. In addition, we study an artificially generated bimodal Gaussian dataset for which we sample the prior distribution in every singular vector independently from a (bimodal) Gaussian mixture distribution. The pdf in every $j$ therefore reads
\begin{align}
    p_{X,j}(x_j)=\rho_1\ g_{v, t_1}(x_j) + \rho_2\ g_{v, t_2}(x_j)
\end{align}
where $\rho_1=0.35$, $\rho_2=0.65$, $v=0.15$, $t_1=-0.6$, $t_2=0.6$ and $g_{s,t}$ is the pdf of a Gaussian with standard deviation $s$ and expectation value $t$. This bimodal structure enables us to further explore the data dependency of the optimized models.

The architecture of the subnetworks included in the diagonal architecture and their training is realized identical to~\cite{iresnet_01_regtheory}: Each subnetwork consists of a three-layer fully connected network, equipped with 35 hidden neurons in each of the first two layers and one output neuron. In addition, we apply the ReLU activation function to the first and second layers of each subnetwork. Figure~7 of~\cite{iresnet_01_regtheory} depicts the architecture. The network weights are optimized using Adam~\cite{KingmaB14}. Extending the approach in~\cite{miyato2018spectral}, we parameterize the network weights in each layer of all subnetworks to fulfill the Lipschitz constraint.
Reconstruction training is accomplished by computing the inverse of the iResNets through the usual fixed point iteration (for 30 iterations) and backpropagating through the unrolled iteration to optimize the network weights. As a result, one has to run the iterative inversion in every training step, resulting in a much greater computational effort for the reconstruction training than approximation training. The Lipschitz bound, realized by a proper parameterization of the $f_\theta$, has to be computed only once per iteration.

Of course, a numerically more efficient training approach would be to extend on Remark~\ref{rem:inverse_of_ires_is_ires} and construct $\psi$, the inverse of an iResNet, as a scaled iResNet that is trained on reversed data points but similar to the approximation approach. However, we currently do not have guarantees on the approximation capability of the involved (fully-connected) iResNets architectures and their inverses. To provide a fair comparison, we aim to enforce the same inductive bias in both training methods by choosing the same forward mapping architecture.

The code for our experiments is available at \url{https://gitlab.informatik.uni-bremen.de/inn4ip/iresnet-regularization}.

In the described settings, we perform experiments for varying noise levels $\delta_\ell$ and Lipschitz constants $L_m$, where we choose
\begin{align}
    \hat{\delta}_\ell &= \begin{cases}
        \left(\frac{1}{3}\right)^\ell, & \text{for } \ell > 0\\
        0, & \text{for } \ell=0
    \end{cases} & \ell&= 0,\cdots, 6\\
   L_m &= 1-\left(\frac{1}{3}\right)^m, & m &= 1,\cdots,5
\end{align}
and the resulting noise $\eta$ is Gaussian noise with standard deviation $\delta_\ell=\hat{\delta}_\ell\cdot\mathrm{std}_{\text{dataset}}$, where $\mathrm{std}_{\text{dataset}}$ denotes the averaged standard deviation of the coefficients $\mathrm{std}_j := \mathrm{std}( \langle x^{(i)},v_j \rangle_{i=1,\hdots,N})$ of the current dataset (i.e., standard deviation with
respect to $i$, mean with respect to $j$). In the following, we discuss the results in terms of the learned solutions, the resulting data-dependent filter functions, and the regularization and approximation properties of the models.

\subsection{Learned inverse mappings}\label{sec:numerics_solutions}
To compare the characteristics of the approximation and reconstruction training to our theoretical findings, we plot the learned one-dimensional inverse mappings $\psi_j$ in the different components $j$ (corresponding to the singular values $\sigma_j$) for the bimodal dataset. We visualize the results in Figure~\ref{fig:solution_plots} for a large and a small eigenvalue of $A$.

\begin{figure}
    \centering
    \begin{subfigure}[b]{0.48\textwidth}
    \centering
        \includegraphics[width=7.8cm]{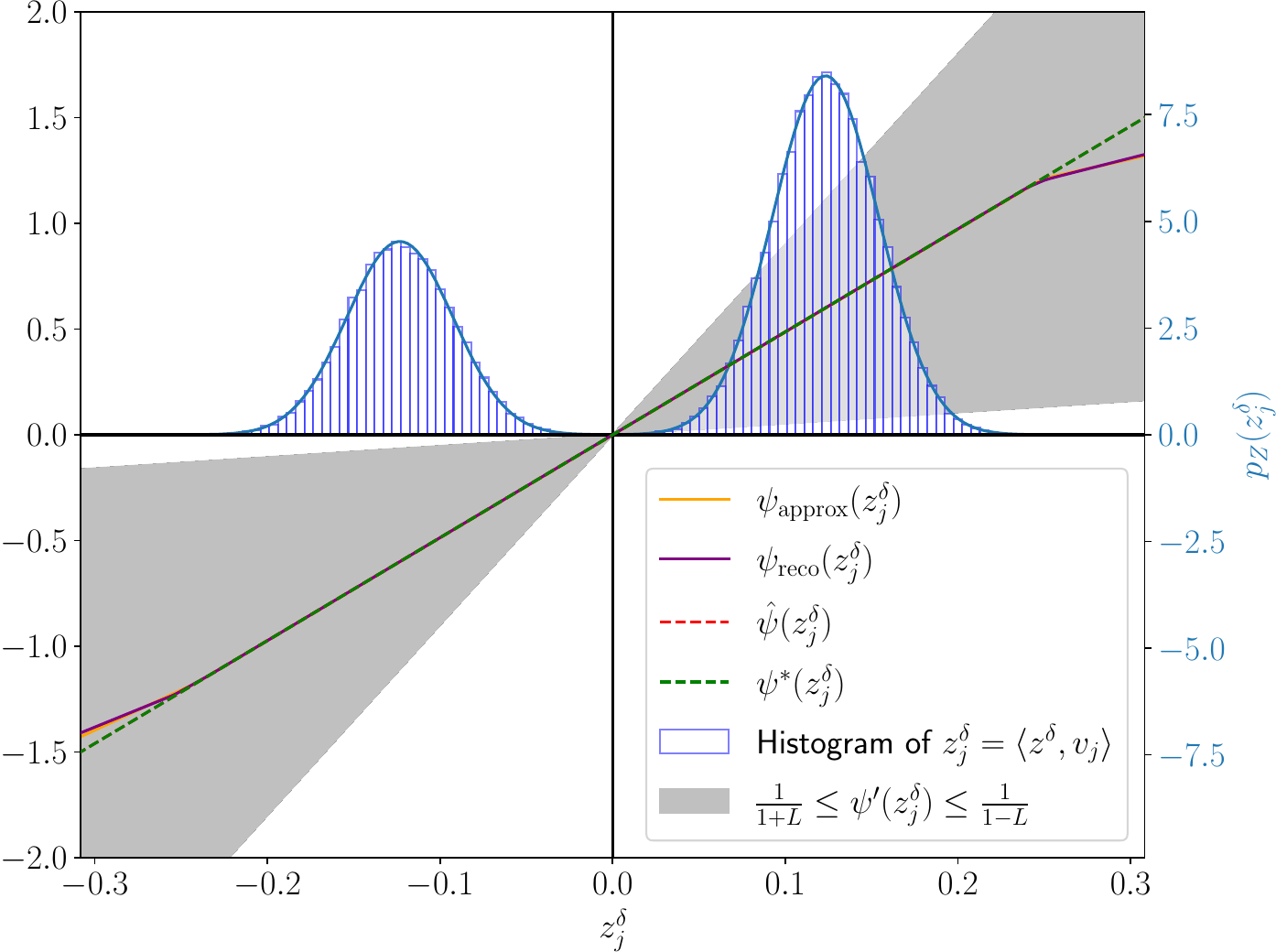}
    \end{subfigure}
    \hfill
    \begin{subfigure}[b]{0.48\textwidth}
    \centering
        \includegraphics[width=7.8cm]{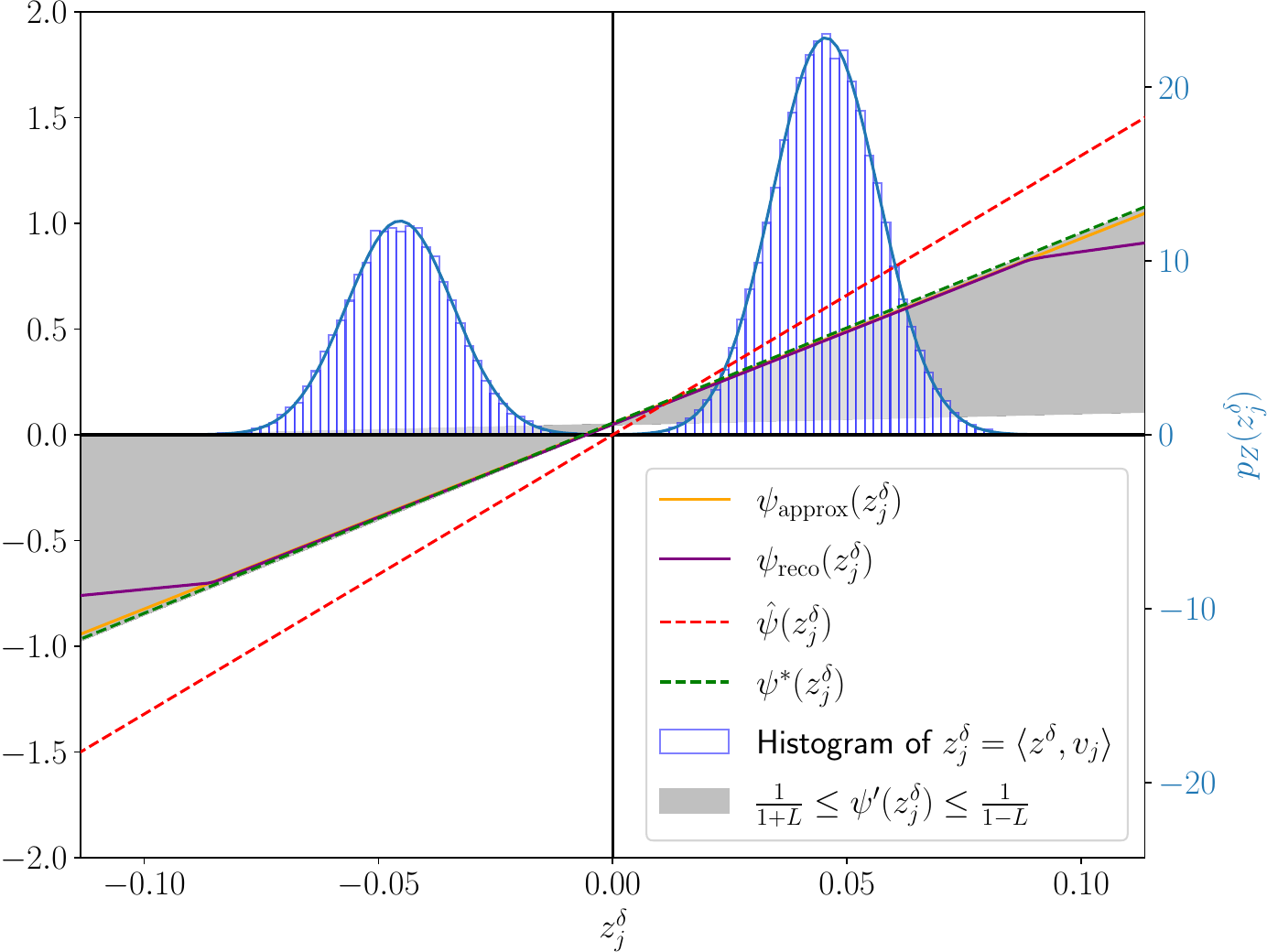}
    \end{subfigure}

    \vspace{2.0ex}
    \begin{subfigure}[b]{0.48\textwidth}
    \centering
        \includegraphics[width=7.8cm]{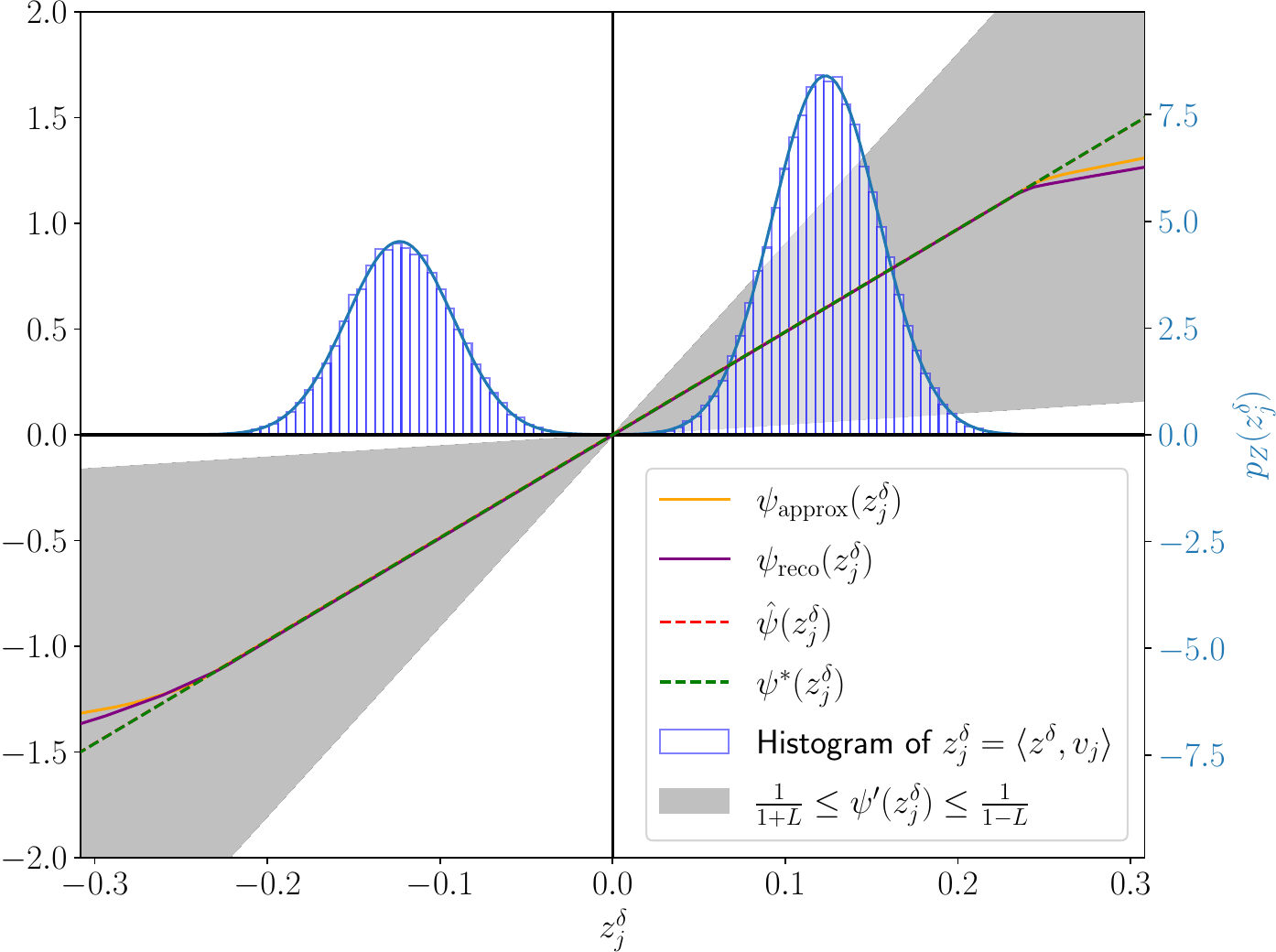}
    \end{subfigure}
    \hfill
    \begin{subfigure}[b]{0.48\textwidth}
    \centering
        \includegraphics[width=7.8cm]{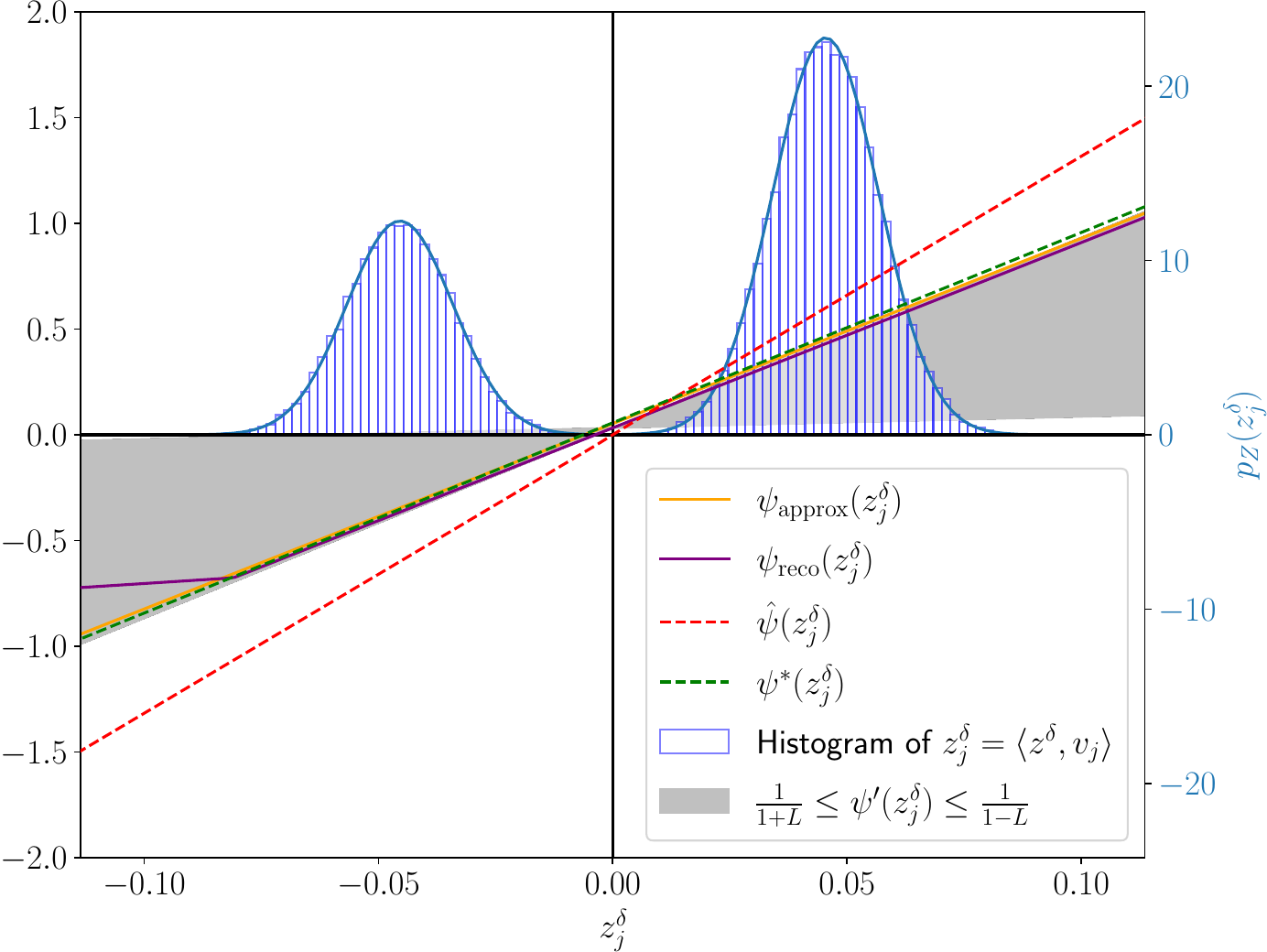}
    \end{subfigure}

    \vspace{2.0ex}
    \begin{subfigure}[b]{0.48\textwidth}
    \centering
        \includegraphics[width=7.8cm]{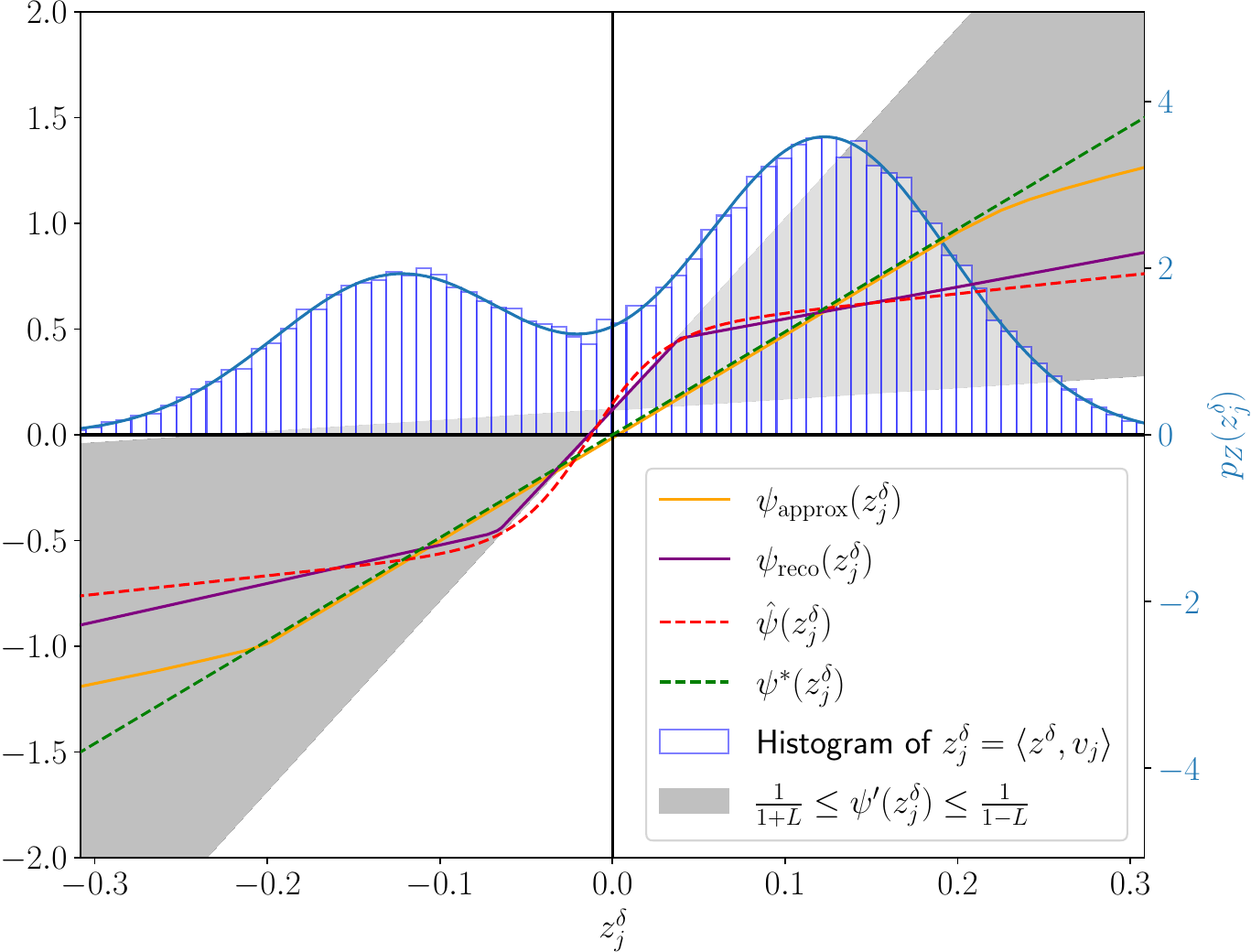}
        \subcaption{$\sigma_j^2=0.205, \ j = 95$}
    \end{subfigure}
    \hfill
    \begin{subfigure}[b]{.48\textwidth}
    \centering
        \includegraphics[width=7.8cm]{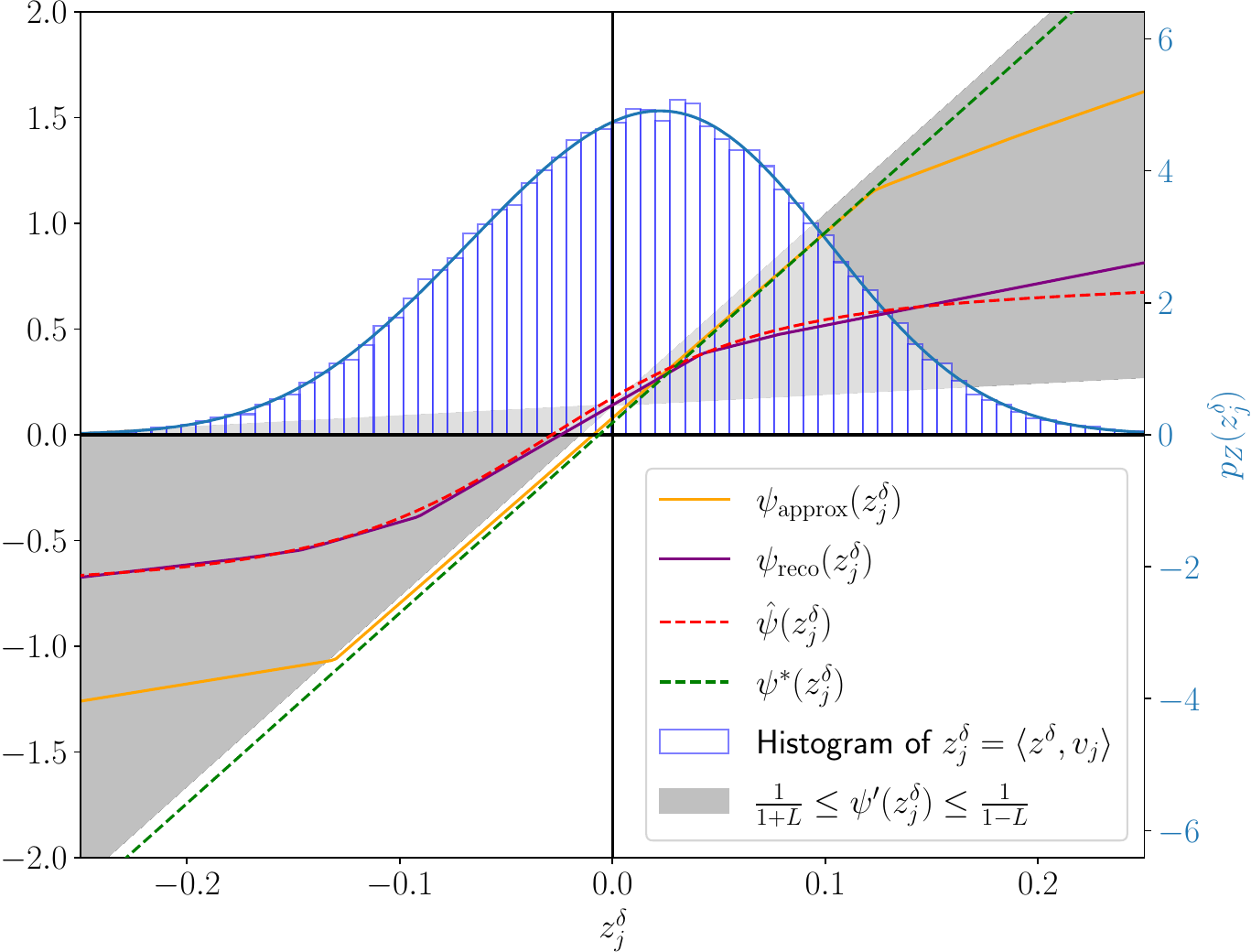}
        \subcaption{$\sigma_j^2=0.075, \ j = 150$}
    \end{subfigure}
    
    \caption{Reconstructions $\psi_\mathrm{approx}^\ast(z_j^\delta)$ trained via approximation training and  $\psi_\mathrm{reco}^\ast(z_j^\delta)$ trained via reconstruction training at Lipschitz bound $L_2$ for different singular values and 
    for noise levels "zero"~\textit{($\delta_0$, top row)}, "small"~\textit{($\delta_6$, middle row)} and "large"~\textit{($\delta_2$, bottom row)} for $A=M_a$.}
    \label{fig:solution_plots}
\end{figure}

In the case of approximation training, we observe the predicted affine linear behavior in the support of the data distribution, limited by the Lipschitz constraints and aligned to the expectation value of the training data. This result is independent of the noise and optimal only for small noise levels. At the boundaries of the data distribution seen during training, the proper behavior of the optimal solution from Theorem~\ref{thm:approx_training_lipschitz} is not learned properly. 

For the case of reconstruction training, we can again corroborate our theoretical findings numerically. For this purpose, we compute the posterior expectation value $\mathbb{E}(x_j|z_j^\delta)$ for our setting.
\begin{remark}\label{rem:posterior expectation gmm}
    For the multimodal Gaussian distribution with pdf $p_X$ and Gaussian noise $p_H$,
\begin{align}
    p_H(z) &= g_{w,0}(z) \\
    p_X(x) &= \sum_{k=1}^K \rho_k\, g_{v_k,t_k}(x),
\end{align}
we have 
\begin{align}
    p_z(z) = \sum_{k=1}^K \rho_k \,g_{u_k,t_k\sigma^2}(z)
\end{align}
and the posterior expectation value $\hat{\psi}$ reads 
\begin{align}
\mathbb{E}(x \vert z) &= \frac{\sum_{k=1}^K  \frac{\rho_k}{u_k^2} (\sigma^2 v_k^2z+t_kw^2)\,g_{u_k,t_k\sigma^2}(z)}{\sum_{k=1}^K \rho_k \,g_{u_k,t_k\sigma^2}(z)} \notag \\
&= \sigma^2 z \frac{ \sum_{k=1}^K  \frac{v_k^2 \rho_k}{u_k^2} \, g_{u_k,t_k\sigma^2}(z)}{\sum_{k=1}^K \rho_k\, g_{u_k,t_k\sigma^2}(z)} + w^2\, \frac{ \sum_{k=1}^K  \frac{t_k\rho_k}{u_k^2}\,g_{u_k,t_k\sigma^2}(z)}{\sum_{k=1}^K \rho_k \,g_{u_k,t_k\sigma^2}(z)},
\end{align}
where $u_k = \sqrt{w^2+(\sigma^2v_k)^2}$. We note that this recovers the linear behavior $\mathbb{E}(x \vert z)=z/\sigma^2$ in the noise-free case while it adds a correction term in the noisy case that pulls and pushes data points towards more likely results. 
\end{remark}
For regions within the support of the data distribution, where the constraint~\eqref{eq:cond_on_inverse_net} permits the model to approximate $\mathbb{E}(x_j\vert z_j^\delta)$, we obverse in Figure~\ref{fig:solution_plots} that the learned solutions match well with the posterior expectation. If the model reaches the limiting constraint, it exhausts the possible slope to be as close to the posterior expectation value as possible. This results in a much more data-dependent inversion scheme, where reconstructions that were more likely to appear during training are favored. Consequently, the model can compensate for larger noise levels based on additional learned knowledge about the data. The behavior of the learned solution thus coincides with the theoretically founded one in Figure~\ref{fig:fig:1d reco general case}.

In the case of large noise, the reconstruction-based model regularizes and does not necessarily exhaust the Lipschitz constraint, while the approximation model always tries to fit the operator as well as possible. If noise is absent, the learned mappings coincide for both training strategies.

\subsection{Learned filter functions}\label{sec:numerics_filters}
As a link to classical regularization theory, we visualize the data-dependent filter functions that correspond to the learned models. For this purpose, we evaluate the filter $r_L$, where
\begin{align}
    (\Id-f_{\theta,j})^{-1}(s(q))-\hat{b}_{L,j}=r_L(\sigma_j^2,s(q))\,s(q)\ \text{for }s\in\mathbb{R},
\end{align}
for each singular value $\sigma_j$ at data points $s(q) := \sigma_j^2 (\mu_{X,j} + q \cdot \mathrm{std}_j)$, where we subtract the axis intercept $\hat{b}_{L,j}=(\Id-f_{\theta,j})^{-1}(0)$. The variable $q \in \mathbb{R}$ determines the number of standard deviations $\mathrm{std}_j$ away from the mean value $\mu_{X,j} = \frac{1}{N} \sum_{i=1}^N \langle x^{(i)},v_j  \rangle$ in the image of the dataset with respect to the fixed singular value of the operator. For simplicity, we define 
\begin{equation}\label{eq:filter_function_numeric}
    R_L(\sigma_j,q) := r_L(\sigma_j^2,s(q))\,s(q).
\end{equation}
The results for approximation and reconstruction training are visualized as surface plots for both datasets in Figures~\ref{fig:surfacefilter_conv_bimodal} and~\ref{fig:surfacefilter_conv_MNIST}. In all cases, the $r_L$ show a sensible behavior, damping small singular values and roughly satisfying $r_L\to 1$ as $\sigma^2\to 1$.

On the bimodal dataset in Figure~\ref{fig:surfacefilter_conv_bimodal}, data dependency occurs in all filter functions. In the case of approximation training, this is visible only in the sense that proper regularizing behavior is learned exclusively within the support of the data distribution. Reconstruction training shows a more complex dependency since the posterior mean aims to push points that lie close to $0$ towards the two peaks of the bimodal data distribution if noise is present in the data, as can also be seen in Figure~\ref{fig:solution_plots}. As a result, the medium singular values, in which an amplified slope in the solution is, on the one hand, permitted by the Lipschitz constraint and, on the other hand, also necessary due to the present influence of noise, are elevated near the center of the distribution. This results in filter functions that may not lead to convergent regularization schemes but include data-driven corrections for the observed noise.

\begin{figure}
    \centering
    \begin{subfigure}{\textwidth}
        \centering
        \includegraphics[width=.9\textwidth]{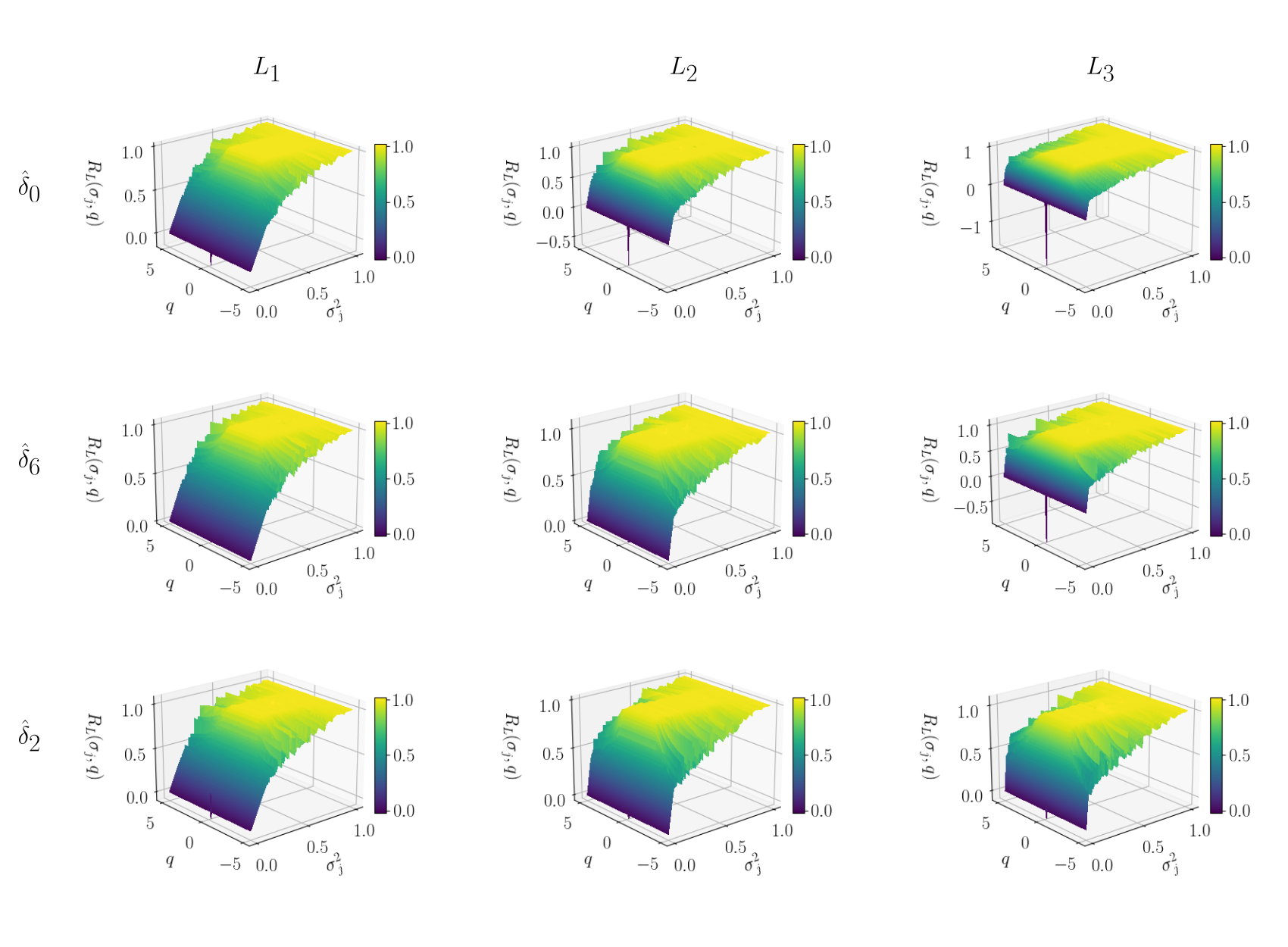}
    \end{subfigure}

    \vspace{-5.0ex}
    \begin{subfigure}{\textwidth}
        \centering
        \includegraphics[width=.9\textwidth]{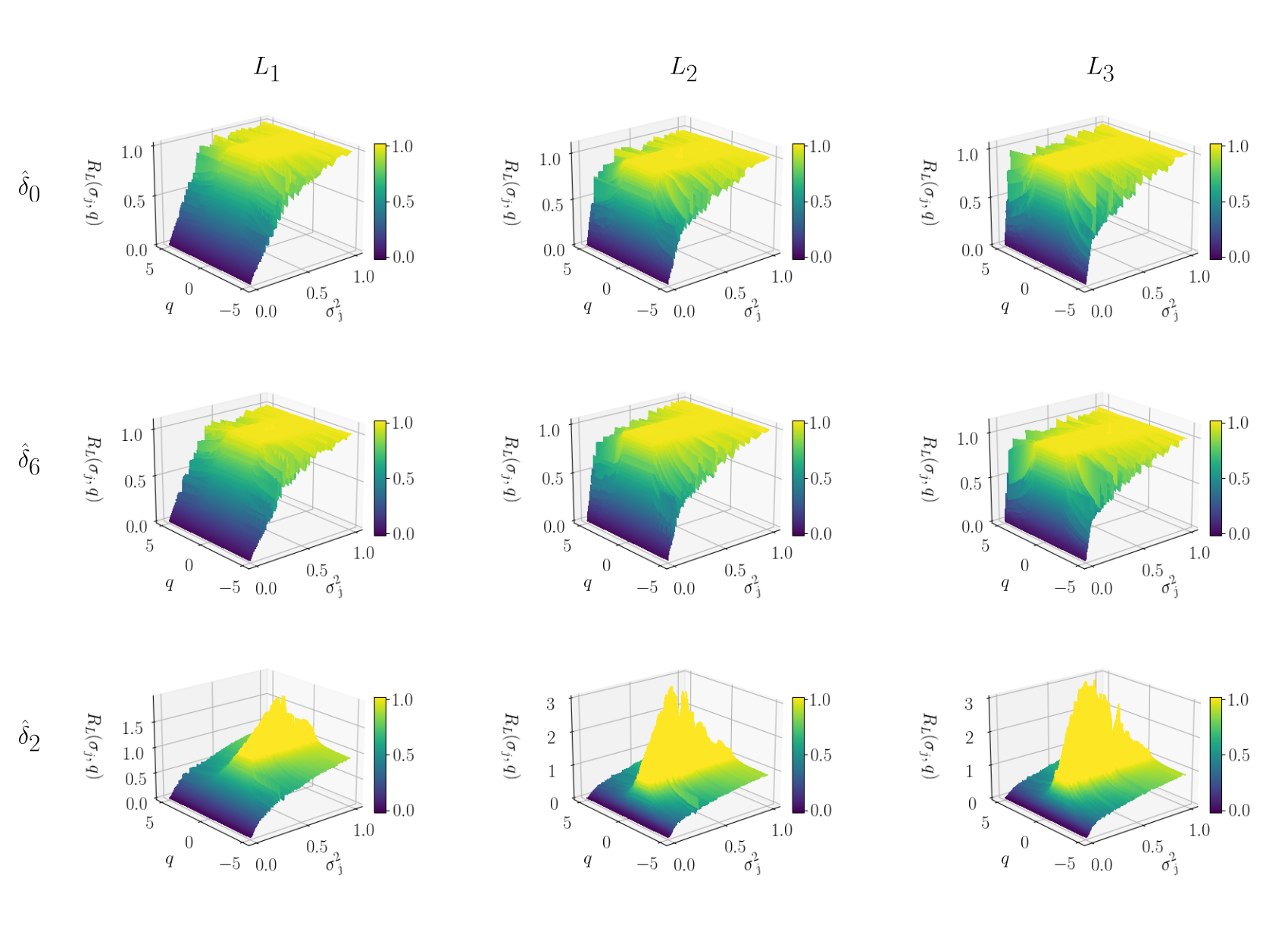}
    \end{subfigure}
    \caption{Filter functions $R_L(\sigma_j,q)$ as defined in~\eqref{eq:filter_function_numeric} corresponding to trained networks $\varphi_{\theta(L_m),\delta_\ell}$ for $m=1,2,3$~\textit{(columns)} and $\ell = 0,6,2$~\textit{(rows)}, trained via approximation training~\textit{(top)} and via reconstruction training~\textit{(bottom)} on the bimodal dataset for $A = M_a$.}
    \label{fig:surfacefilter_conv_bimodal}
\end{figure}

On MNIST in Figure~\ref{fig:surfacefilter_conv_MNIST}, the stronger data dependence of the reconstruction training is not directly visible; all filters appear to be approximately constant in the range of 5 standard deviations around the mean. This is likely to be due to the fact that the distributions in the singular values are all approximately unimodal. Therefore, the correction could be similar to the simple regularizing behavior of the neural networks in the approximation training approach. In return, the action of $L$ as a regularization parameter becomes visible. Especially for small $L$, the filter functions show similarities to the case of squared soft TSVD; see Remark~\ref{remark:sq_soft_tsvd}. 
The learned filter functions in approximation training are very similar for every noise level, which is in line with the developed theory. In contrast, the filters learned in reconstruction training show stronger regularization (i.e., more dampening of small singular values) for larger noise levels, adapting to the data seen during training. This again fits well within the theoretical results developed in Section~\ref{sec:reco_training} and is especially visible for large $L=L_3$, where the model is, in principle, able to fit the operator well. However, the filter functions for the largest noise $\hat{\delta}_3$ and $L_2, L_3$ look very similar, indicating that the model learns strong regularization from data at the cost of a worse operator approximation.

\begin{figure}
    \centering
    \begin{subfigure}{\textwidth}
        \centering
        \includegraphics[width=.9\textwidth]{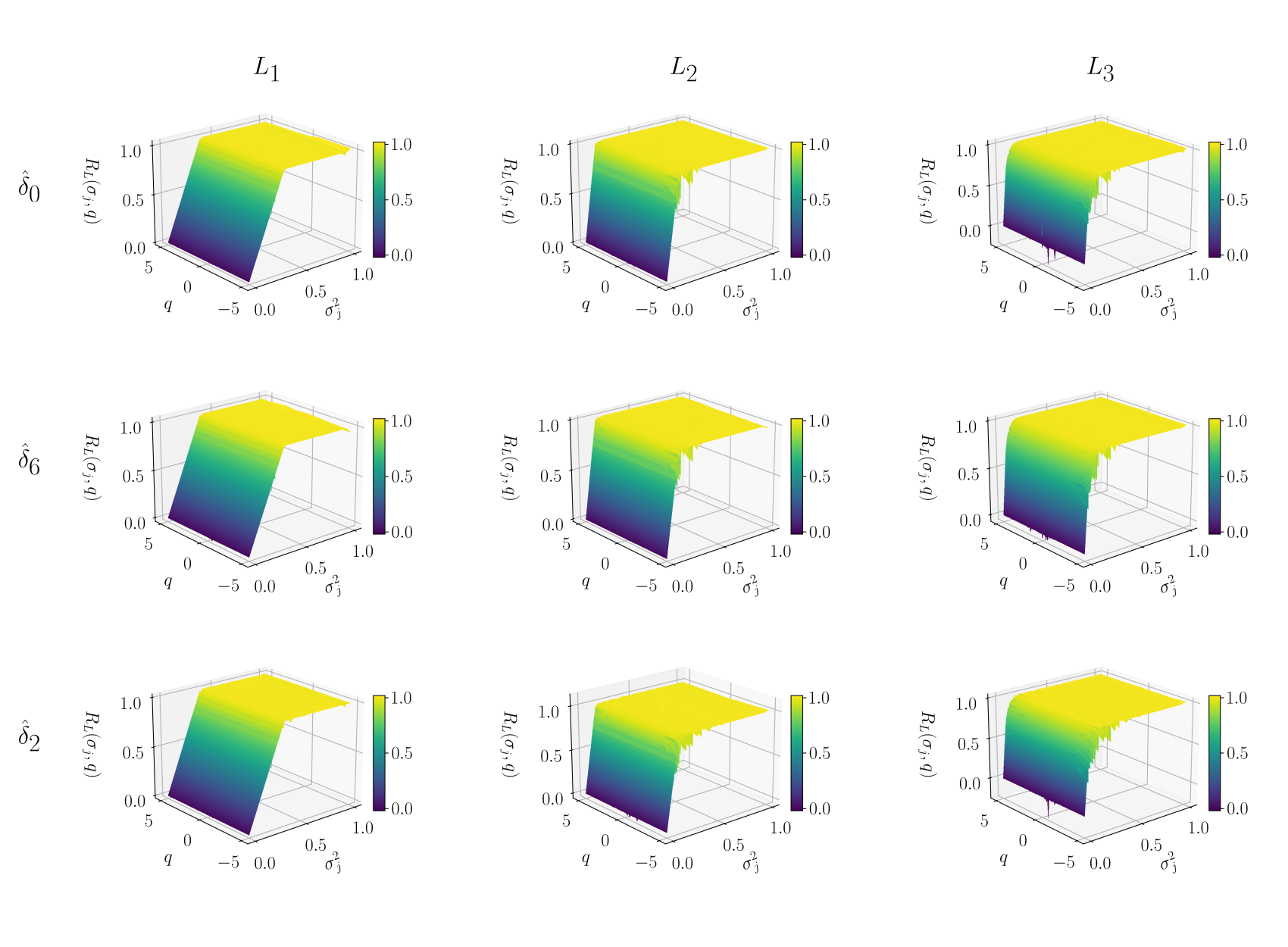}
    \end{subfigure}

    \vspace{-5.0ex}
    \begin{subfigure}{\textwidth}
        \centering
        \includegraphics[width=.9\textwidth]{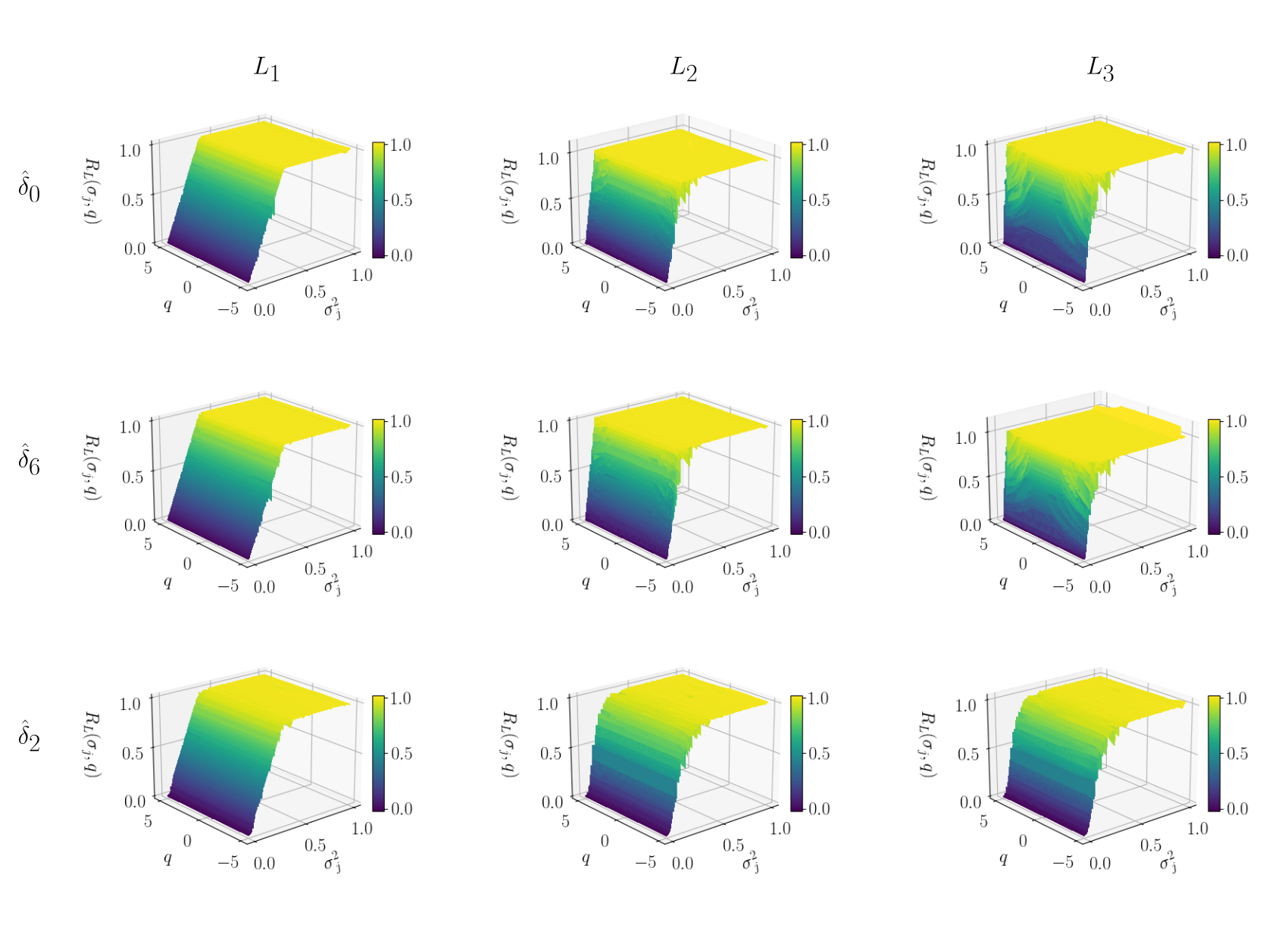}
    \end{subfigure}
    \caption{Filter functions $R_L(\sigma_j,q)$ as defined in~\eqref{eq:filter_function_numeric} corresponding to trained networks $\varphi_{\theta(L_m,\delta_\ell)}$ for $m=1,2,3$~\textit{(columns)} and $\ell = 0,6,2$~\textit{(rows)}, trained via approximation training~\textit{(top)} and via reconstruction training~\textit{(bottom)} on the MNIST dataset for $A = M_a$.}
    \label{fig:surfacefilter_conv_MNIST}
\end{figure}

\subsection{Reconstruction quality and convergence}\label{sec:numerics_convergence}
To compare the performance of the different models in terms of reconstruction quality, we show images and filter functions for a single MNIST digit in Figure~\ref{fig:reconstructions_conv}.

Especially for large noise, the visual quality clearly benefits from the reconstruction training compared to the approximation approach. This supports the argument that additional data dependence may be desirable for large-noise applications. In the same case, the approximation training shows its regularization properties and indicates proper parameter choice rules: The image quality improves for smaller Lipschitz constraints which imply a strong regularization. This behavior is not visible in the reconstruction approach, where a large $L$ generally seems to improve the quality. The filter plots we provide for the given sample in Figure~\ref{fig:reconstructions_conv} (right) also underlined this. In this case, they show a similar graph for $L_2$ and $L_3$. The filter plots also reveal that the models optimized via approximation training are independent of the noise seen during training and therefore learn identical filter functions for all noise levels. Table~\ref{tab:reconstructions_conv} also reveals an advantage of this method for small noise levels.

\begin{figure}
    \centering
    \begin{subfigure}{\textwidth}
        \centering 
         \includegraphics[width=.9\textwidth]{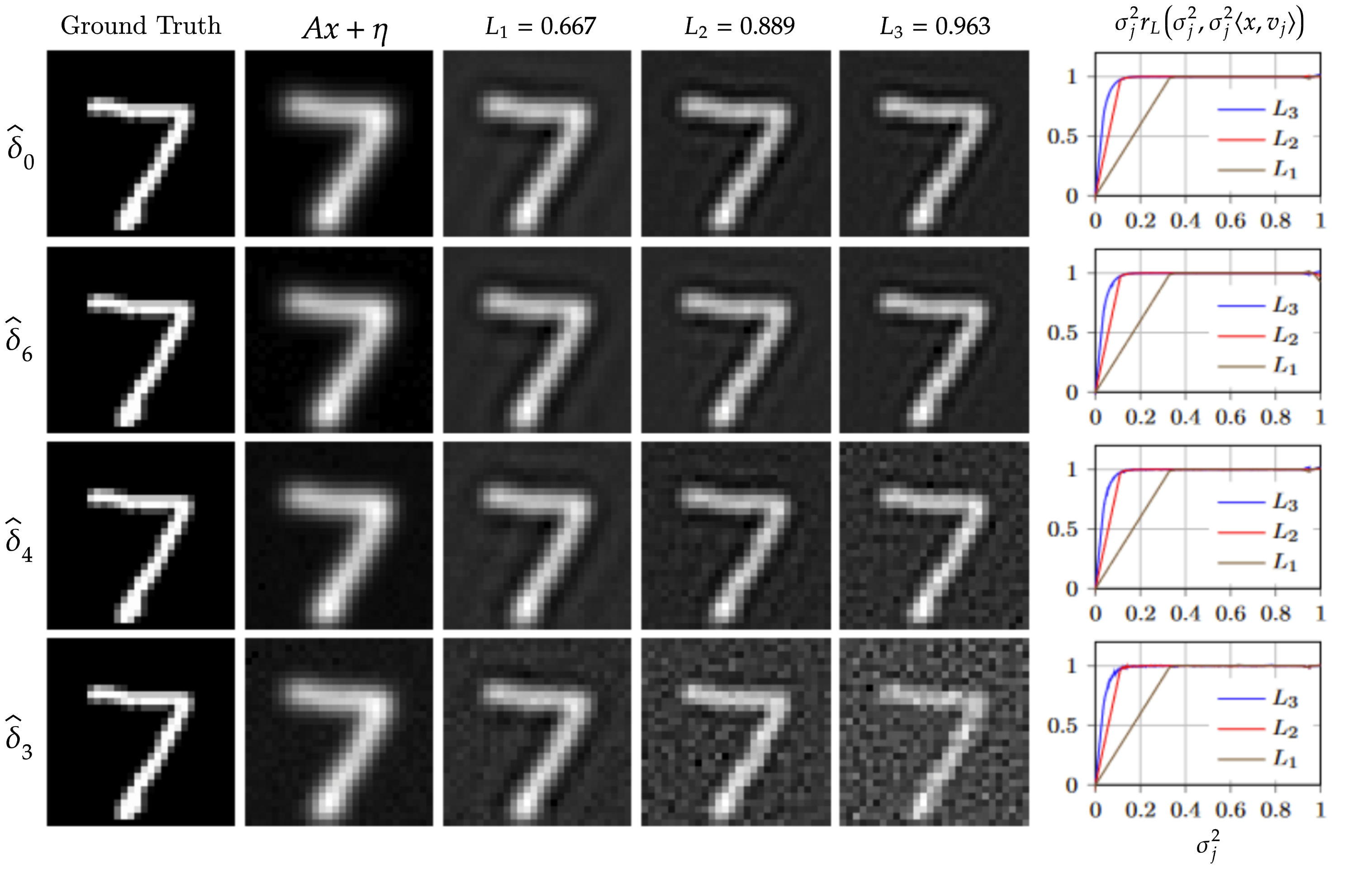}
    \end{subfigure}
    
    \vspace{2.0ex}
    \begin{subfigure}{\textwidth}
        \centering 
         \includegraphics[width=.9\textwidth]{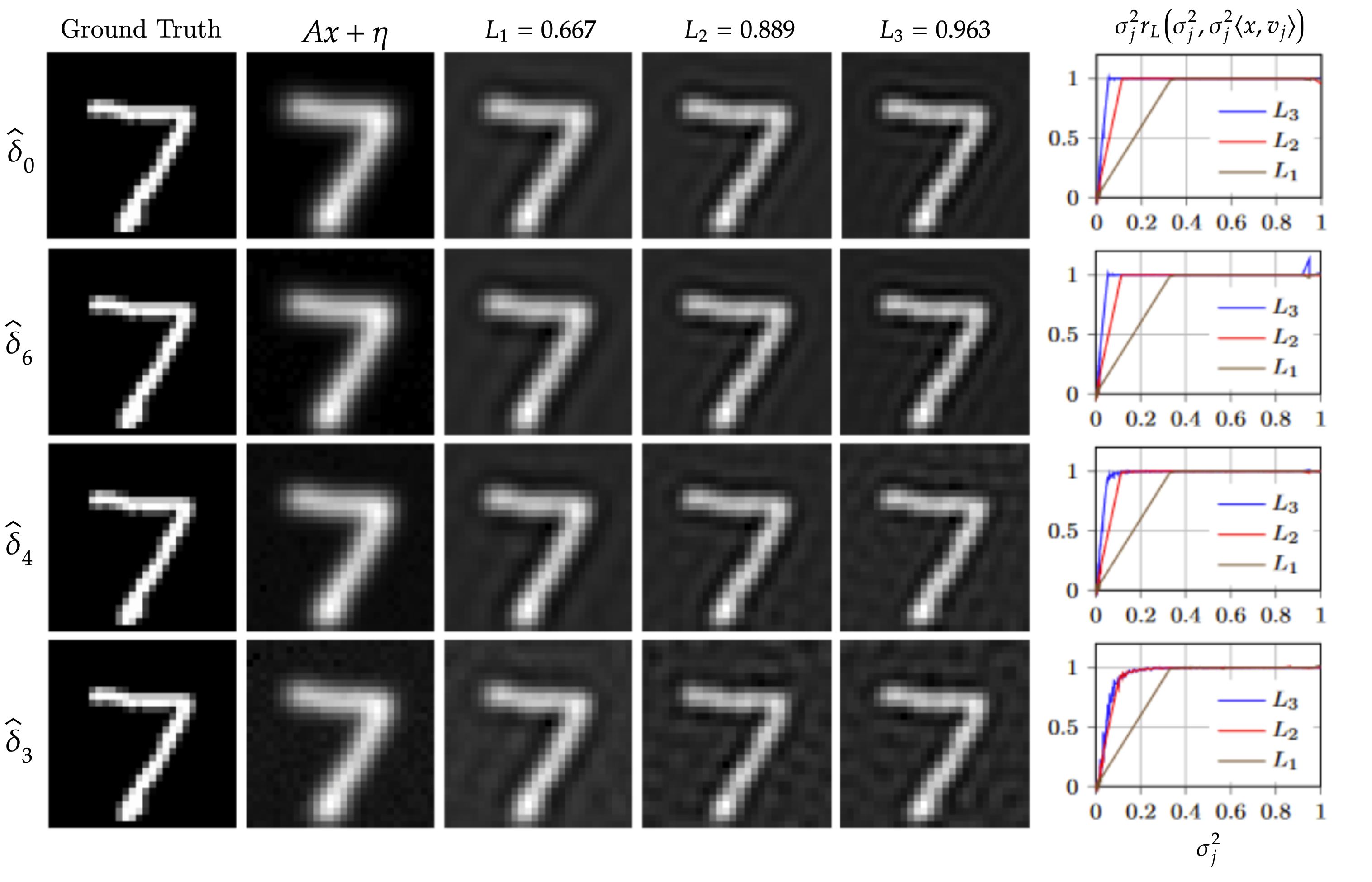}
    \end{subfigure}
    \caption{Reconstructions of an MNIST sample $x =x^{(1)}$ from the test dataset by computing $\varphi_{\theta(L_m,\delta_\ell)}^{-1}(Ax+\tilde{\eta})$ with $\tilde{\eta} \sim \mathcal{N}(0,\delta_\ell \Id)$ for Lipschitz bounds $L_m$ with $m=1,2,3$~\textit{(columns)} and noise levels $\delta_\ell= \hat{\delta}_\ell \cdot \mathrm{std}_\mathrm{MNIST}$ with $\ell =0,6,4,3$~\textit{(rows)} together with corresponding filter functions for $A = M_a$. The top subfigure depicts the reconstructions from networks trained via approximation training, and the bottom subfigure corresponds to the networks optimized via reconstruction training.}
    \label{fig:reconstructions_conv}
\end{figure}

\setlength{\tabcolsep}{0.7em} 
{\renewcommand{\arraystretch}{1.1}
\begin{table}[]
\centering
\begin{tabular}{|c|ccc|ccc|}
\hline
approximation & \multicolumn{3}{c|}{SSIM} & \multicolumn{3}{c|}{MSE} \\ \cline{2-7}
 training & $L_1$    & $L_2$    & $L_3$    & $L_1$    & $L_2$    & $L_3$    \\ \hline
 $\delta_0$                                                                       & 0.8242 & 0.8769 & 0.9035  & 0.0099  & 0.0057 & 0.0042 \\
$\delta_6$                                                                       & 0.8206 & 0.8782 & 0.9039  & 0.0102  & 0.0056 & 0.0042 \\
$\delta_4$                                                                       & 0.8178 & 0.8426 & 0.7863  & 0.0101  & 0.0063 & 0.0072 \\
$\delta_3$                                                                       & 0.7849 & 0.7996 & 0.5738 & 0.0109  & 0.0113 & 0.0299 \\ \hline
reconstruction & \multicolumn{3}{c|}{SSIM} & \multicolumn{3}{c|}{MSE} \\ \cline{2-7}  training & $L_1$    & $L_2$    & $L_3$    & $L_1$    & $L_2$    & $L_3$ \\ \hline
$\delta_0$                                                                       & 0.8238 & 0.8792 & 0.8958  & 0.0101  & 0.0057 & 0.0043 \\
$\delta_6$                                                                       & 0.8213 & 0.8802 & 0.8951  & 0.0103  & 0.0056 & 0.0043 \\
$\delta_4$                                                                       & 0.8181 & 0.8703 & 0.8809  & 0.0104  & 0.0059 & 0.0047 \\
$\delta_3$                                                                       & 0.8030 & 0.8257 & 0.8128 & 0.0106 & 0.0067 & 0.0069 \\
\hline
\end{tabular}
\caption{SSIM and MSE measures corresponding to reconstructions of $x^{(1)}$ in Figure~\ref{fig:reconstructions_conv}.}
\label{tab:reconstructions_conv}
\end{table}}

Another way to study the convergence in $L$ and $\delta$ of the trained models is to evaluate the approximation properties in terms of the overall errors on the dataset with respect to different error measures. To evaluate the localized approximation property that has been introduced in~\cite[Theorem 3.1]{iresnet_01_regtheory}, we define
\begin{align}
    \mathcal{E}_\mathrm{mean}(\varphi_{\theta(L)},A) &= \frac{1}{N} \sum_{i=1}^N \|\varphi_{\theta(L)}(x^{(i)}) - A x^{(i)} \|, \\
    \mathcal{E}_{x^{(m)}}(\varphi_{\theta(L)},A) &= \|\varphi_{\theta(L)}(x^{(m)}) - A x^{(m)} \|,
\end{align}
estimating the approximation error of the trained model for the whole dataset or a single sample.  In Figure~\ref{fig:loc_approx_conv_MNIST}, we plot this error for the models trained without noise and varying $L$. In the case of approximation training, we observe slightly superlinear convergence in the dataset on average, indicating that the property is satisfied for many samples. As proven in prior work \cite{iresnet_01_regtheory}, this implies that the training constructs a convergent regularization scheme for these samples. For reconstruction-based training, this is not fulfilled on average.  To preserve some insights on the convergence of the method, we extend on the weaker but sufficient condition in \cite[Remark~3.2]{iresnet_01_regtheory} and define
\begin{align}
    \widetilde{\mathcal{E}}_\mathrm{mean}(\varphi_{\theta(L)},A) &= \frac{1}{N} \sum_{i=1}^N\|x^{(i)} - \varphi_{\theta(L)}^{-1} (A x^{(i)}) \| \\
     \widetilde{\mathcal{E}}_{x^{(m)}}(\varphi_{\theta(L)},A) &= \|x^{(m)} - \varphi_{\theta(L)}^{-1} (A x^{(m)}) \|,
\end{align}
which is closer to the target in reconstruction training. In this case, $\widetilde{\mathcal{E}}_{x^{(m)}}(\varphi_{\theta(L)},A)\xrightarrow{L \to 1}0$ would be sufficient for local convergence. Figure~\ref{fig:loc_approx_conv_MNIST} indicates that this property can still be satisfied, however, with slow convergence rates.

\begin{figure}
\captionsetup[subfigure]{labelformat=empty}
\begin{subfigure}{\textwidth}
    \centering
    \begin{tikzpicture}
        \pgfplotsset{
        every axis plot/.append style={thick},
        tick style={black, thick},
        every axis plot/.append style={line width=0.8pt},
        every axis/.style={
            axis y line=left,
            axis x line=bottom,
            axis line style={ultra thick,->,>=latex, shorten >=-.4cm}
        },
        }
        \begin{groupplot}[
            group style={
                group name=my plots,
                group size=2 by 1,
                ylabels at=edge left,
                horizontal sep=3.0cm
            },
            scale only axis=true,
            width=0.23\linewidth,
            height=3.0cm,
            axis lines=middle, 
            xlabel = {$1-L$},
            ytick = {0.,5,10,15,20,25,30},
            xtick = {0, 250,500, 750},
            ymin = 0.0,
            ymax = 0.22,
            xmin = 0.0,
             x label style={at={(axis description cs:1.14,0.1)},anchor=north},
             y label style={at={(axis description cs:0.035,1.32)},anchor=north},
        ]
        \nextgroupplot[
            xmode=log,ymode=log, 
            ymin =0.0001, ymax=0.135,
            xmin=0.001, xmax=1,
            axis y line=left,
            ytick={0.001,0.01,0.1},
            xtick={0.001,0.01,0.1,1},
            axis x line=bottom,
            x label style={at={(axis description cs:1.29,0.1)},anchor=north},
            y label style={at={(axis description cs:0.025,1.27)},anchor=north},
            legend style={at={(1.95,1)},draw=none},
            legend cell align={left}
        ]
        \addplot[color=black,dashed, domain=0.001:1]{0.3*x};
        \addplot[color=gray,dashed, domain=0.001:1]{0.3*x^2};
        \addplot[teal,mark=*] table [x=L, y=mean_norm, col sep=comma]{csv_files/lap_conv_approx_pt2.csv};
        \addplot[blue,mark=*] table [x=L, y=x1_norm, col sep=comma]{csv_files/lap_conv_approx_pt2.csv};
        \addplot[red,mark=*] table [x=L, y=x2_norm, col sep=comma]{csv_files/lap_conv_approx_pt2.csv};

        \nextgroupplot[
            xmode=log,ymode=log, 
            ymin =0.0001, ymax=0.135,
            xmin=0.001, xmax=1,
            axis y line=left,
            ytick={0.001,0.01,0.1},
            xtick={0.001,0.01,0.1,1},
            axis x line=bottom,
            x label style={at={(axis description cs:1.29,0.1)},anchor=north},
            y label style={at={(axis description cs:0.025,1.27)},anchor=north},
            legend style={at={(2.02,1.1)},draw=none},
            legend cell align={left}
        ]
        \addplot[color=black,dashed, domain=0.001:1]{0.3*x};
        \addlegendentry{{\footnotesize{$\mathcal{O}\big((1-L)\big)$}}}
        \addplot[color=gray,dashed, domain=0.001:1]{0.3*x^2};
        \addlegendentry{{\footnotesize{$\mathcal{O}\big((1-L)^2\big)$}}}
        \addplot[teal,mark=*] table [x=L, y=mean_norm, col sep=comma]{csv_files/lap_conv_reco_pt2.csv};
        \addlegendentry{{$\mathcal{E}_\mathrm{mean}(\varphi_{\theta(L)},A)$}}
        \addplot[blue,mark=*] table [x=L, y=x1_norm, col sep=comma]{csv_files/lap_conv_reco_pt2.csv};
        \addlegendentry{{$\mathcal{E}_{x^{(1)}}(\varphi_{\theta(L)},A)$}}
        \addplot[red,mark=*] table [x=L, y=x2_norm, col sep=comma]{csv_files/lap_conv_reco_pt2.csv};
        \addlegendentry{{$\mathcal{E}_{x^{(2)}}(\varphi_{\theta(L)},A)$}}
        
        \end{groupplot}
    \end{tikzpicture}
    \end{subfigure}

    \vspace{4.0ex}
 
    \hfill
    \begin{subfigure}[b]{0.145\textwidth}
    \centering

    \vspace{-10.0ex}
    \includegraphics[width=0.75\textwidth]{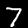}
    \vspace{-1.0ex}
    \caption{\footnotesize{$x^{(1)}$}}

    \phantom{\includegraphics[width=.3\textwidth]{images/test_min.png}}
    \vfill
    \end{subfigure}
    \begin{subfigure}{0.145\textwidth}
    \centering
    \vspace{-3.0ex}
        \includegraphics[width=0.75\textwidth]{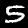}
    \vspace{-1.0ex}
    \caption{\footnotesize{$x^{(2)}$}}
 
    \phantom{\includegraphics[width=.3\textwidth]{images/test_min.png}}
    \end{subfigure}
    \hspace{2.0ex}
    \begin{subfigure}{.59\textwidth}
        \centering
    \begin{tikzpicture}
        \pgfplotsset{
        every axis plot/.append style={thick},
        tick style={black, thick},
        every axis plot/.append style={line width=0.8pt},
        every axis/.style={
            axis y line=left,
            axis x line=bottom,
            axis line style={ultra thick,->,>=latex, shorten >=-.4cm}
        },
        }
        \begin{groupplot}[
            group style={
                group name=my plots,
                group size=1 by 1,
                ylabels at=edge left,
                horizontal sep=1.5cm
            },
            scale only axis=true,
            width=0.39\linewidth,
            height=3.0cm,
            axis lines=middle, 
            xlabel = {$1-L$},
            ytick = {0.,5,10,15,20,25,30},
            xtick = {0, 250,500, 750},
            ymin = 0.0,
            ymax = 0.22,
            xmin = 0.0,
             x label style={at={(axis description cs:1.14,0.1)},anchor=north},
             y label style={at={(axis description cs:0.035,1.32)},anchor=north},
        ]
        
        \nextgroupplot[
                    xmode=log,ymode=log, 
                    ymin =0.04, ymax=0.15,
                    xmin=0.002, xmax=1,
                    axis y line=left,
                    ytick={0.001,0.01,0.1},
                    xtick={0.001,0.01,0.1,1},
                    axis x line=bottom,
                    x label style={at={(axis description cs:1.29,0.1)},anchor=north},
                    y label style={at={(axis description cs:0.025,1.27)},anchor=north},
                    legend style={at={(2.02,1)},draw=none},
                    legend cell align={left}
                ]
                \addplot[teal,mark=*] table [x=L, y=gen_mean_norm, col sep=comma]{csv_files/lap_conv_reco_pt2.csv};
                \addlegendentry{{$\tilde{\mathcal{E}}_\mathrm{mean}(\varphi_{\theta(L)},A)$}}
                \addplot[blue,mark=*] table [x=L, y=x1_gen_norm, col sep=comma]{csv_files/lap_conv_reco_pt2.csv};
                \addlegendentry{{$\tilde{\mathcal{E}}_{x^{(1)}}(\varphi_{\theta(L)},A)$}}
                \addplot[red,mark=*] table [x=L, y=x2_gen_norm, col sep=comma]{csv_files/lap_conv_reco_pt2.csv};
                \addlegendentry{{$\tilde{\mathcal{E}}_{x^{(2)}}(\varphi_{\theta(L)},A)$}}
        \end{groupplot}
    \end{tikzpicture}
    \end{subfigure}
    \caption{Test samples $x^{(1)}$ and $x^{(2)}$~\textit{(bottom left)}. Evaluations of the local approximation property via $\mathcal{E}_\mathrm{mean}(\varphi_{\theta(L_m)},A), \ \mathcal{E}_{x^{(1)}}(\varphi_{\theta(L_m)},A)$ and $\mathcal{E}_{x^{(2)}}(\varphi_{\theta(L_m)},A)$ for the approximation training~\textit{(top left)} and the reconstruction training~\textit{(top right)}, and evaluations of the generalized approximation property via $\tilde{\mathcal{E}}_\mathrm{mean}(\varphi_{\theta(L_m)},A), \tilde{\mathcal{E}}_{x^{(1)}}(\varphi_{\theta(L_m)},A)$ and $\tilde{\mathcal{E}}_{x^{(2)}}(\varphi_{\theta(L_m)},A)$ for the reconstruction training~\textit{(bottom right)} for $L_m = 1-\nicefrac{1}{3^m}$ with $m=1,\hdots,5$, $A = M_a$ on the MNIST test dataset.  
    }

    \label{fig:loc_approx_conv_MNIST}
\end{figure}



\begin{figure}
\centering
\begin{subfigure}{\textwidth}
\centering
\begin{tikzpicture}
\pgfplotsset{
every axis plot/.append style={line width=0.8pt},
xmode = log,
ymode = log,
log y ticks with fixed point/.style={
      yticklabel={
        \pgfkeys{/pgf/fpu=true}
        \pgfmathparse{exp(\tick)}%
        \pgfmathprintnumber[fixed relative, precision=3]{\pgfmathresult}
        \pgfkeys{/pgf/fpu=false}}    
}
}
\begin{groupplot}[
    group style={
        group name=my plots,
        group size=3 by 1,
        ylabels at=edge left,
        xlabels at=edge bottom,
        horizontal sep=2.5cm
    },
    ymode=log,
    scale only axis,
    footnotesize,
    width=3.7cm,
    height=3.7cm,
    tickpos=left,
    x tick label style={/pgf/number format/fixed,
    /pgf/number format/1000 sep = \thinspace},
    ytick align=outside,
    xtick align=outside,
    enlarge x limits=false,
    xmax=0.33,
    grid = both,
    xtick={0,0.0001,0.001,0.01,0.1,1},
    anchor = north,
    clip=true,
    cycle list name=color list
]
\nextgroupplot[ ylabel={$\mathrm{MSE}_\mathrm{reco}^{\delta\ell}(\varphi_{\theta(L,\delta_\ell)},A)$}, ymax = 0.7, ytick ={0.001,0.01,0.1,1}, yticklabels={$10^{-3}$,$ 10^{-2}$,$10^{-1}$}
 ]
\addplot table [x=noise, y=1, col sep=comma]{csv_files/conv_MSEoverNoise_MNIST_approx_reco_noise=True.csv};
\addplot table [x=noise, y=2, col sep=comma]{csv_files/conv_MSEoverNoise_MNIST_approx_reco_noise=True.csv};
\addplot table [x=noise, y=3, col sep=comma]{csv_files/conv_MSEoverNoise_MNIST_approx_reco_noise=True.csv};
\addplot table [x=noise, y=4, col sep=comma]{csv_files/conv_MSEoverNoise_MNIST_approx_reco_noise=True.csv};
\addplot table [x=noise, y=5, col sep=comma]{csv_files/conv_MSEoverNoise_MNIST_approx_reco_noise=True.csv};

  \nextgroupplot[
   legend pos=outer north east,
    legend style={draw=none},
    legend cell align={left},
  xshift=-1.3cm,
  ymax = 0.03, ytick ={0.001,0.01,0.1,1}, yticklabels={$10^{-3}$,$ 10^{-2}$,$10^{-1}$}
  ]
  \addplot table [x=noise, y=1, col sep=comma]{csv_files/conv_MSEoverNoise_MNIST_reco_reco_noise=True.csv};
  \addlegendentry{$L_1$}
 \addplot table [x=noise, y=2, col sep=comma]{csv_files/conv_MSEoverNoise_MNIST_reco_reco_noise=True.csv};
\addlegendentry{$L_2 $}
\addplot table [x=noise, y=3, col sep=comma]{csv_files/conv_MSEoverNoise_MNIST_reco_reco_noise=True.csv};
\addlegendentry{$L_3 $}
\addplot table [x=noise, y=4, col sep=comma]{csv_files/conv_MSEoverNoise_MNIST_reco_reco_noise=True.csv};
\addlegendentry{$L_4$}
 \addplot table [x=noise, y=5, col sep=comma]{csv_files/conv_MSEoverNoise_MNIST_reco_reco_noise=True.csv};
\addlegendentry{$L_5 $}

\nextgroupplot[
     legend pos= north west,
    legend cell align={left},
  ymax = 0.15, ytick ={0.001,0.01,0.1,1}, yticklabels={$10^{-3}$,$ 10^{-2}$,$10^{-1}$}
  ]
\addplot table [x=noise, y=min_mean, col sep=comma]{csv_files/conv_MSEoverNoise_MNIST_approx_reco_noise=True.csv};
\addlegendentry{approximation}
 \addplot table [x=noise, y=min_mean, col sep=comma]{csv_files/conv_MSEoverNoise_MNIST_reco_reco_noise=True.csv};
\addlegendentry{reconstruction}

\end{groupplot}
\end{tikzpicture}
\end{subfigure}

\begin{subfigure}{\textwidth}
\centering
\begin{tikzpicture}
\pgfplotsset{
every axis plot/.append style={line width=0.8pt},
xmode = log,
log y ticks with fixed point/.style={
      yticklabel={
        \pgfkeys{/pgf/fpu=true}
        \pgfmathparse{exp(\tick)}%
        \pgfmathprintnumber[fixed relative, precision=3]{\pgfmathresult}
        \pgfkeys{/pgf/fpu=false}}    
}
}
\begin{groupplot}[
    group style={
        group name=my plots,
        group size=3 by 1,
        ylabels at=edge left,
        xlabels at=edge bottom,
        horizontal sep=2.5cm
    },
    scale only axis,
    footnotesize,
    width=3.7cm,
    height=3.7cm,
    tickpos=left,
    x tick label style={/pgf/number format/fixed,
    /pgf/number format/1000 sep = \thinspace},
    ytick align=outside,
    xtick align=outside,
    enlarge x limits=false,
    xmax=0.33,
    xlabel={$\hat{\delta}_\ell$},
    grid = both,
    xtick={0,0.0001,0.001,0.01,0.1,1},
    anchor = north,
    clip=true,
    cycle list name=color list
]
\nextgroupplot[ ylabel={$\mathrm{SSIM}^{\delta_\ell}(\varphi_{\theta(L,\delta_\ell)},A)$}, ymin=0.15, ymax = 1, ytick ={0.2,0.4,0.6,0.8,1}, yticklabels={$0.2$,$ 0.4$,$0.6$,$0.8$,$1$}
 ]
\addplot table [x=noise, y=1, col sep=comma]{csv_files/conv_SSIM_MNIST_approx_noise=True.csv};
\addplot table [x=noise, y=2, col sep=comma]{csv_files/conv_SSIM_MNIST_approx_noise=True.csv};
\addplot table [x=noise, y=3, col sep=comma]{csv_files/conv_SSIM_MNIST_approx_noise=True.csv};
\addplot table [x=noise, y=4, col sep=comma]{csv_files/conv_SSIM_MNIST_approx_noise=True.csv};
\addplot table [x=noise, y=5, col sep=comma]{csv_files/conv_SSIM_MNIST_approx_noise=True.csv};

  \nextgroupplot[ 
    legend pos=outer north east,
    legend style={draw=none},
    legend cell align={left},
  xshift=-1.3cm,
  ymin=0.6, ymax = 1, ytick ={0.2,0.4,0.6,0.8,1}, yticklabels={$0.2$,$ 0.4$,$0.6$,$0.8$,$1$}
  ]
  \addplot table [x=noise, y=1, col sep=comma]{csv_files/conv_SSIM_MNIST_reco_noise=True.csv};
  \addlegendentry{$L_1$}
 \addplot table [x=noise, y=2, col sep=comma]{csv_files/conv_SSIM_MNIST_reco_noise=True.csv};
\addlegendentry{$L_2 $}
\addplot table [x=noise, y=3, col sep=comma]{csv_files/conv_SSIM_MNIST_reco_noise=True.csv};
\addlegendentry{$L_3 $}
\addplot table [x=noise, y=4, col sep=comma]{csv_files/conv_SSIM_MNIST_reco_noise=True.csv};
\addlegendentry{$L_4$}
 \addplot table [x=noise, y=5, col sep=comma]{csv_files/conv_SSIM_MNIST_reco_noise=True.csv};
\addlegendentry{$L_5 $}

\nextgroupplot[
    legend pos= south west,
    legend cell align={left},
 ymin=0.2, ymax = 1, ytick ={0.2,0.4,0.6,0.8,1}, yticklabels={$0.2$,$ 0.4$,$0.6$,$0.8$,$1$}
  ]
\addplot table [x=noise, y=max_mean, col sep=comma]{csv_files/conv_SSIM_MNIST_approx_noise=True.csv};
\addlegendentry{approximation}
 \addplot table [x=noise, y=max_mean, col sep=comma]{csv_files/conv_SSIM_MNIST_reco_noise=True.csv};
\addlegendentry{reconstruction}

\end{groupplot}
\end{tikzpicture}
\end{subfigure}
\caption{Reconstruction errors $\mathrm{MSE}_\mathrm{reco}^{\delta_\ell}(\varphi_{\theta(L,\delta_\ell)},A)$~\textit{(top row)} and $\mathrm{SSIM}^{\delta_\ell}(\varphi_{\theta(L,\delta_\ell)},A)$~\textit{(bottom row)} for networks trained on noisy samples with noise levels $\delta_\ell$ for $\ell = 0,\hdots,6$ and reconstructions from noisy samples of the same noise level for the approximation training~\textit{(left)} and for the reconstruction training~\textit{(middle)} with Lipschitz bounds $L_m$ on the MNIST dataset for $A = M_a$. Outcomes of optimal parameter choices for both training strategies over different noise levels can be seen on the right-hand side.
}
\label{fig:MSE_SSIM_reco_MNIST_conv}
\end{figure}
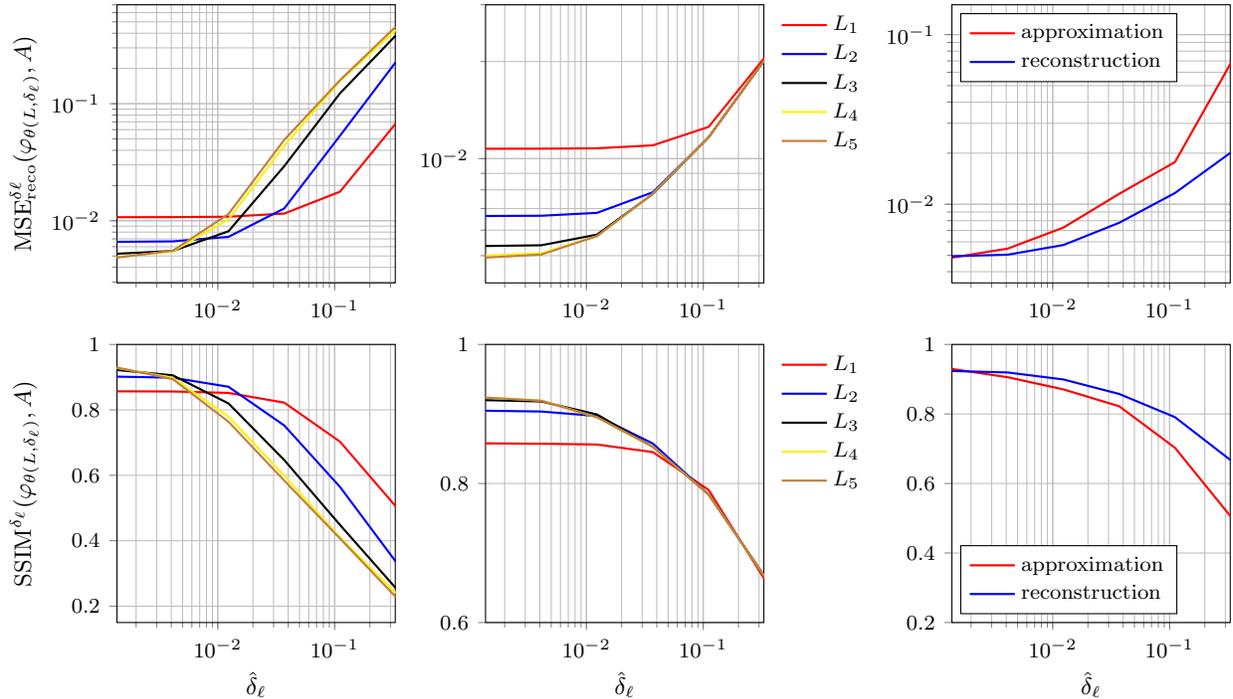

In addition, we evaluate the reconstruction error of the training approaches for varying noise levels. Figure~\ref{fig:MSE_SSIM_reco_MNIST_conv} depicts the results for the mean squared error
\begin{align}
    \mathrm{MSE}_\mathrm{reco}^{\delta_\ell}(\varphi_{\theta(L,\delta_\ell)},A) = \frac{1}{N} \sum_{i=1}^N\| x^{(i)}-\varphi_{\theta(L,\delta_\ell)}^{-1}(Ax^{(i)}+\eta^{(i)}) \|^2
\end{align}
and averaged structural similarity index measure~(SSIM) as defined in~\cite{wang2004}, computed for the dataset by 
\begin{align}
    \mathrm{SSIM}^{\delta_\ell}(\varphi_{\theta(L,\delta_\ell)},A) = \frac{1}{N} \sum_{i=1}^N \mathrm{SSIM}(x^{(i)},\varphi_{\theta(L,\delta_\ell)}^{-1}(Ax^{(i)}+\eta^{(i)})).
\end{align}
Here, one can again see that the approximation training comes with a typical parameter choice rule known from regularization theory, where one has to choose $L\to1$ while $\delta\to 0$. In contrast, the reconstruction training performs best with the largest $L$ among all noise levels since the regularization emanates from the data.
This learned regularization behavior also becomes apparent in Figure~\ref{fig:trueL_conv}, which illustrates the actual Lipschitz constant of the learned residual function.
While in the approximation training case, the actual Lipschitz constant $\tilde{L}$ always reaches the constraint one, we can observe a different behavior in the reconstruction training, particularly for larger Lipschitz constants used in the constraint and for larger noise. 
First, with increasing noise levels, we observe a decay in the actual Lipschitz constant. 
This is because the noise distribution $p_H$, which suffers from a larger standard deviation, decreases the slope of the conditional mean $\hat{\psi}$ and thus, for larger singular values, the slope of $\varphi_{\theta(L,\delta_\ell)}^{-1}$ does not need to reach the upper bound $\frac{1}{1-L}$ anymore. 
As a result, the actual Lipschitz constant further decreases until it reaches a minimum and then increases again.  
While the larger singular values cause the first decay, the subsequent increase is caused by small singular values due to an analogous influence of $p_H$ on the conditional mean. 
The larger the noise becomes, the more the slope of $\varphi_{\theta(L,\delta_\ell)}^{-1}$ for small singular values need to reach the lower bound $\frac{1}{1+L}$ to approximate the posterior mean. 
As a result, we observe the increase of the actual Lipschitz constant for large noise. 
A visualization of this behavior can be found in the appendix in Figure~\ref{fig:SW_trueL_conv}, which illustrates the differences in the behavior for small, intermediate, and large singular values.
Overall, the reconstruction training method outperforms the approximation training for all large $\delta$ but stays slightly behind for very small noise.

\begin{figure}
\centering
\begin{tikzpicture}
\pgfplotsset{
every axis plot/.append style={line width=0.8pt},
xmode = log,
ymode = log,
log y ticks with fixed point/.style={
      yticklabel={
        \pgfkeys{/pgf/fpu=true}
        \pgfmathparse{exp(\tick)}%
        \pgfmathprintnumber[fixed relative, precision=3]{\pgfmathresult}
        \pgfkeys{/pgf/fpu=false}}    
}
}
\begin{groupplot}[
    group style={
        group name=my plots,
        group size=2 by 1,
        ylabels at=edge left,
        xlabels at=edge bottom,
        horizontal sep=1.8cm
    },
    scale only axis,
    footnotesize,
    width=5.3cm,
    height=4.2cm,
    tickpos=left,
    x tick label style={/pgf/number format/fixed,
    /pgf/number format/1000 sep = \thinspace},
    ytick align=outside,
    xtick align=outside,
    enlarge x limits=false,
    xmax=0.33,
    xlabel={$\hat{\delta}_\ell$},
    grid = both,
    xtick={0,0.0001,0.001,0.01,0.1,1},
    anchor = north,
    clip=true,
    cycle list name=color list
]
\nextgroupplot[ 
ymin=0.65, ymax = 1.07, ytick ={0.7,0.8,0.9,1.0}, yticklabels={$0.7$,$0.8$,$0.9$,$1.0$}
 ]
\addplot table [x=noise, y=1, col sep=comma]{csv_files/conv_trueL_MNIST_approx_noise=True.csv};
\addplot table [x=noise, y=2, col sep=comma]{csv_files/conv_trueL_MNIST_approx_noise=True.csv};
\addplot table [x=noise, y=3, col sep=comma]{csv_files/conv_trueL_MNIST_approx_noise=True.csv};
\addplot table [x=noise, y=4, col sep=comma]{csv_files/conv_trueL_MNIST_approx_noise=True.csv};
\addplot table [x=noise, y=5, col sep=comma]{csv_files/conv_trueL_MNIST_approx_noise=True.csv};
\addplot[color=red, dashed, domain=0.009:1]{0.667};
\addplot[color=blue, dashed, domain=0.009:1]{0.889};
\addplot[color=black, dashed, domain=0.009:1]{0.963};
\addplot[color=yellow, dashed, domain=0.009:1]{0.988};
\addplot[color=brown, dashed, domain=0.009:1]{0.996};

\nextgroupplot[ 
ymin=0.65, ymax = 1.07, ytick ={0.7,0.8,0.9,1.0}, yticklabels={$0.7$,$0.8$,$0.9$,$1.0$},legend pos=outer north east,
    legend style={draw=none},
    legend columns=2,legend style={
            /tikz/column 2/.style={
                column sep=5pt,
            },},
    legend cell align={left}
 ]
\addplot[red] table [x=noise, y=1, col sep=comma]{csv_files/conv_trueL_MNIST_reco_noise=True.csv};
\addlegendentry{$\tilde{L}_1$}
\addplot[color=red, dashed, domain=0.009:1]{0.667};
\addlegendentry{$L_1$}
\addplot[blue] table [x=noise, y=2, col sep=comma]{csv_files/conv_trueL_MNIST_reco_noise=True.csv};
\addlegendentry{$\tilde{L}_2$}
\addplot[color=blue, dashed, domain=0.009:1]{0.889};
\addlegendentry{$L_2$}
\addplot[black] table [x=noise, y=3, col sep=comma]{csv_files/conv_trueL_MNIST_reco_noise=True.csv};
\addlegendentry{$\tilde{L}_3$}
\addplot[color=black, dashed, domain=0.009:1]{0.963};
\addlegendentry{$L_3$}
\addplot[yellow] table [x=noise, y=4, col sep=comma]{csv_files/conv_trueL_MNIST_reco_noise=True.csv};
\addlegendentry{$\tilde{L}_4$}
\addplot[color=yellow, dashed, domain=0.009:1]{0.988};
\addlegendentry{$L_4$}
\addplot[brown] table [color=brown, x=noise, y=5, col sep=comma]{csv_files/conv_trueL_MNIST_reco_noise=True.csv};
\addlegendentry{$\tilde{L}_5$}
\addplot[color=brown,dashed, domain=0.009:1]{0.996};
\addlegendentry{$L_5$}
\end{groupplot}
\end{tikzpicture}
\caption{Lipschitz constraints $L_m$ of networks $\varphi_{\theta(L_m,\delta_\ell)}$ together with respective computed Lipschitz constants $\tilde{L}_m$ of the trained residual functions $f_{\theta(L_m,\delta_\ell)}$ at $m=1,\hdots,5$ for noise levels $\delta_\ell$ with $\ell = 0,\hdots,6$ via approximation training~\textit{(left)} and reconstruction training~\textit{(right)} on the MNIST dataset for $A =M_a$.
}
\label{fig:trueL_conv}
\end{figure}
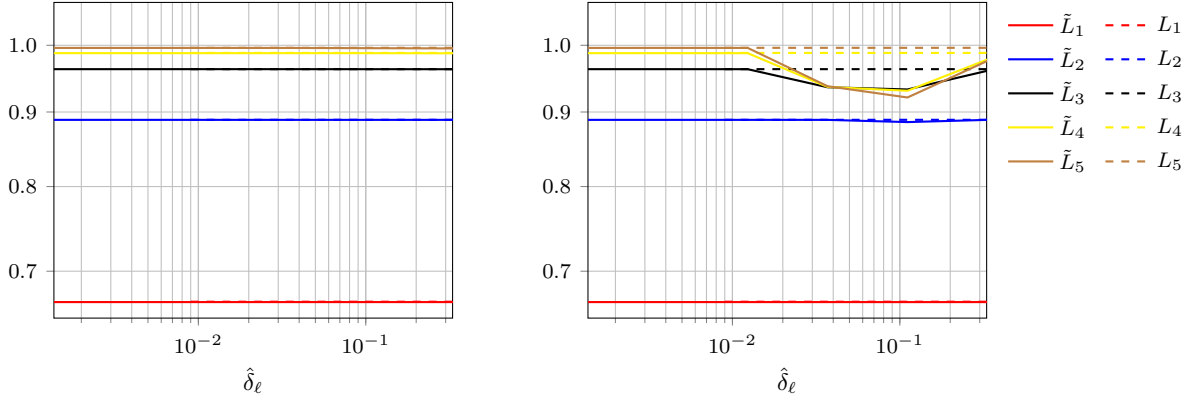

\section{Discussion and outlook}
\label{sec:discussion}

The present work can be seen as a continuation of \cite{iresnet_01_regtheory}. There, the authors investigated regularization properties of the proposed iResNet reconstruction approach for specific network architectures trained according to the approximation training on samples to impart data dependency. Here, we have extended the theory by focusing on the question of to what extent the training data distribution influences the optimal parameters of the iResNet and, in turn, the resulting reconstruction scheme. To this end, we considered the training of the iResNets from a Bayesian perspective and focused on two different loss functions, namely the approximation training and the reconstruction training.

In the approximation training, our results for the diagonal architecture show that for all possible prior and noise distributions, the best suited residual function $f$ is an affine linear one whose optimal parameters depend on the mean of the prior distribution, the eigenvalue $\sigma_j^2$ and the Lipschitz constant $L < 1$. Thus, in this setting, the data dependency on the training outcome is minimal, with no influence of the noise distribution and, especially, the regularization properties of the resulting reconstruction scheme. Instead, the amount of regularization of $\varphi_\theta^{-1}$ is solely controlled by the Lipschitz constant $L$, which also becomes apparent by the observed equivalence of $\varphi_\theta^{-1}$ to a convergent filter-based regularization scheme with bias.

In contrast, the prior and noise distributions significantly impact the optimal architecture and parameters when considering the reconstruction training. Here, we realize the network by an iResNet's inverse and train it to approximate $A^{-1}$, resulting in a stable reconstruction scheme. We showed that we can interpret the optimal network as an approximation of the conditional mean estimator $z^\delta \mapsto \mathbb{E}(x | z^\delta)$ w.r.t.\ the $p_Z$-weighted $L^2$-norm, where $p_Z$ is the density function of $Ax + \eta$. Hence the optimal architecture choice and the corresponding optimal parameters depend on the prior and the noise distribution. Consequently, this indicates that the amount of regularization of the trained network is controlled by the Lipschitz constant $L$ and possibly by the amount of noise in the training data. 

The theoretical findings are validated and further corroborated by a series of numerical investigations on the MNIST dataset and an artificially generated dataset following a bimodal Gaussian distribution for two different forward operators. In particular, the results highlight that in reconstruction training, the noise distribution influences the regularization properties of the network. In the approximation training, the noise does not influence the regularization properties; they are solely controlled by the Lipschitz constant $L$. As a result, the reconstruction training leads to superior regularizations in high noise regimes, whereas the approximation training is more suitable in low noise regimes due to better convergence properties.

These investigations of the approximation and reconstruction training illustrate how the loss function determines the influence of the prior and noise distribution on the reconstruction scheme and shed light on which architectures are suitable. 
Investigating a link to MAP estimation and how it could be represented in terms of an iResNet might allow for revealing further links to regularization theory in future works (see also a more detailed discussion in Appendix~\ref{app:MAP_resnet_Bayesian}). 


The presented results allow further investigations and can serve as a foundation for future research directions. 
In the approximation training case, it might be desirable to relax the naive Bayes assumption and to consider the general non-diagonal network case \eqref{eq:approx_general_network}. In line with the found limited data dependency in the diagonal case, we conjecture a dependency on second moments of the prior distribution $p_X$ at most.
In the general network case, we showed that reconstruction training leads to a data-dependent and stable reconstruction scheme that approximates the mean of the posterior distribution, where the degree of stability can be controlled by the Lipschitz constant $L$. What is left to prove is a convergence property as discussed in Remark~\ref{rem:convergence reco}, which could, in principle, further manifest the superiority of the reconstruction training approach and provide additional guarantees. Here, potential generalizations also could incorporate alternative loss functions in the integrand of the reconstruction training loss, which might result in approximations of alternative estimators induced by Bayes costs \cite{kaipio2006statistical}. In this context, as a starting point one may limit oneself to linear estimators to further investigate the relation to learned MAP estimators, respectively Tikhonov regularization, as in \cite{kabri2022convergent, alberti2021learning}.
In order to obtain convergence guarantees as well as data dependence, when desirable, one can explore a noise-controlled convex combination of both training losses. Theoretical as well as numerical investigations in this direction remain future research.

In addition, Remark~\ref{rem:inverse_of_ires_is_ires} serves as a basis to improve the numerical implementation of the reconstruction training by constructing the network as a scaled iResNet resulting in a more efficient training approach as mentioned in Section~\ref{sec:numeric_experiments}. This is a reasonable approach when one is not interested in directly comparing the approximation and reconstruction training. Numerical investigations, including the general non-diagonal network architecture, remain immediate future research.

Besides these improvements, extending our results to deeper network architectures could be beneficial, e.g., a concatenation of iResNets. This would allow for more expressive network architectures and further improve the reconstruction quality of the networks. Finally, it might be worthwhile to generalize the results to nonlinear inverse problems allowing for an application to a larger number of operators.

\section*{Acknowledgments}
N. Heilenk\"otter, M. Iske, and J. Nickel acknowledge the support of the Deutsche Forschungsgemeinschaft (DFG, German
Research Foundation) - Project number 281474342/GRK2224/2. T. Kluth acknowledges support from the DELETO project funded by the Federal Ministry of Education and Research (BMBF, project number 05M20LBB).

\appendix
\section{Proof of Lemma~\ref{lem:solution_reco_train_example}}
\label{appendix:proof_solution_type}

\begin{proof} 
Due to (iii) it immediately follows that $p_X$ is uniformly continuous. From (iii), we also deduce that the marginal of $p_X$ on $\mathcal{N}(A)^\bot$, $p_{X,\mathcal{N}(A)^{\bot}}: \mathcal{N}(A)^\bot \to \R_{\geq 0}$ with 
     \begin{equation}
       p_{X,\mathcal{N}(A)^{\bot}}(x)= \int_{\mathcal{N}(A)}  p_X(x_0 + x) \ \mathrm{d} x_0  
     \end{equation} 
     is compactly supported and bounded and thus also uniformly continuous.
     Analogously, we can deduce uniform continuity of the mappings $g_0: \mathcal{N}(A)^\bot  \to \mathcal{N}(A)$ and $g_\dagger: \mathcal{N}(A)^\bot  \to \mathcal{N}(A)^\bot$ with
          \begin{equation}
       g_0(x)= \int_{\mathcal{N}(A)}  p_X(x_0 + x) x_0 \ \mathrm{d} x_0  
     \end{equation} 
     and
               \begin{equation}
       g_\dagger(x)= \int_{\mathcal{N}(A)}  p_X(x_0 + x) \ \mathrm{d} x_0 \  x . 
     \end{equation}

Next, by using $X=\mathcal{N}(A) \oplus \mathcal{N}(A)^\bot$ we now observe that for arbitrary $z\in X$ due to (i) it holds
\begin{align}
    p_{Z,\delta}(z)&= \int_X p_{H,\delta}(z-Ax) p_X(x) \ \mathrm{d} x \ \notag \\
    &= \int_{\mathcal{N}(A)^\bot} p_{H,\delta}(z-Ax_1)   \int_{\mathcal{N}(A)} p_X(x_0+x_1)  \ \mathrm{d} x_0  \ \mathrm{d} x_1 \notag \\
    &= p^0_{H,\delta}(P_{\mathcal{N}(A)}z)  \int_{\mathcal{R}(A) }  p^\dagger_{H,\delta}(P_{\mathcal{N}(A)^\bot}z-y_1)   \int_{\mathcal{N}(A)} p_X(x_0+A^\dagger y_1)  \ \mathrm{d} x_0  \ \mathrm{d} y_1 |\det A^\dagger |,
\end{align}
where the transformation $x_1=A^\dagger y_1$, with generalized inverse $A^\dagger: \mathcal{R}(A) \to \mathcal{N}(A)^\bot $, is used in the last equality.
Analogously, we further obtain for arbitrary $z\in \mathrm{supp}(p_{Z,\delta})\subset X$
\begin{align}
    \hat{\psi}_{\delta}(z)&= \frac{1}{p_{Z,\delta}(z)}  \int_X p_{H,\delta}(z-Ax) p_X(x) x \ \mathrm{d} x \ \notag \\
    &= \frac{1}{p_{Z,\delta}(z)} \int_{\mathcal{N}(A)^\bot} p_{H,\delta}(z-Ax_1)   \int_{\mathcal{N}(A)} p_X(x_0+x_1) (x_0 + x_1)  \ \mathrm{d} x_0  \ \mathrm{d} x_1 \notag \\
    &= \frac{p^0_{H,\delta}(P_{\mathcal{N}(A)}z) }{p_{Z,\delta}(z)}   \int_{\mathcal{R}(A) }  p^\dagger_{H,\delta}(P_{\mathcal{N}(A)^\bot}z-y_1)    \int_{\mathcal{N}(A)} p_X(x_0+A^\dagger y_1) (x_0 + A^\dagger y_1)  \ \mathrm{d} x_0  \ \mathrm{d} y_1 |\det A^\dagger |.
\end{align}
Exploiting the previously derived  representation of $p_{Z,\delta}$ and (i)  yields
\begin{align}
    \hat{\psi}_\delta (z) & = \frac{\int_{\mathcal{R}(A)} p^\dagger_{H,\delta} ( P_{\mathcal{N}(A)^\bot} z - y_1 ) \left( \int_{\mathcal{N}(A)}  p_X(x_0+A^\dagger y_1 ) (x_0 + A^\dagger y_1)  \ \mathrm{d} x_0  \right) \ \mathrm{d} y_1}{\int_{\mathcal{R}(A)} p^\dagger_{H,\delta} ( P_{\mathcal{N}(A)^\bot} z - y_1 )\int_{\mathcal{N}(A)}  p_X(x_0+A^\dagger y_1)   \ \mathrm{d} x_0 \ \mathrm{d} y_1 } \notag \\
    &= \frac{\int_{\mathcal{R}(A)} p^\dagger_{H,\delta} ( P_{\mathcal{N}(A)^\bot} z - y_1 ) \left( g_0(A^\dagger y_1) + g_\dagger(A^\dagger y_1)  \right) \ \mathrm{d} y_1}{\int_{\mathcal{R}(A)} p^\dagger_{H,\delta} ( P_{\mathcal{N}(A)^\bot} z - y_1 )  p_{X,\mathcal{N}(A)^{\bot}}(A^\dagger y_1) \ \mathrm{d} y_1}.    \label{eq:psi_delta_aux1}
\end{align}
We now consider the denominator and nominator separately. 
In the denominator we first observe that the mapping $p_{X,\mathcal{N}(A)^\bot}(A^\dagger y_1)$ 
is also uniformly continuous due to the continuity of $A^\dagger$.
Exploiting $\mathcal{R}(A)=\mathcal{R}(A^\ast) = \mathcal{N}(A)^{\bot}$, the approximation property \cite[Sec II]{konigsberger2013analysis} of the Dirac sequence $p_{H,\delta}^\dagger$ delivers uniform convergence of the denominator, i.e., this implies pointwise convergence such that for any $z \in X$ it holds 
\begin{equation}
    \int_{\mathcal{R}(A)} p^\dagger_{H,\delta} ( P_{\mathcal{N}(A)^\bot} z - y_1 ) p_{X,\mathcal{N}(A)^\bot} ( A^\dagger y_1 ) \ \mathrm{d} y_1  \underset{\delta \to 0}{ \longrightarrow} p_{X,\mathcal{N}(A)^\bot} ( A^\dagger P_{\mathcal{N}(A)^\bot} z ). 
\end{equation}
Analogous arguments apply to the nominator of \eqref{eq:psi_delta_aux1}
by exploiting the approximation property of the Dirac sequence implying uniform convergence and thus pointwise convergence such that for any $z \in X$ it holds
\begin{equation}
    \int_{\mathcal{R}(A)} p^\dagger_{H,\delta} ( P_{\mathcal{N}(A)^\bot} z - y_1 ) \left( g_0(A^\dagger y_1) + g_\dagger(A^\dagger y_1)   \right) \ \mathrm{d} y_1 \underset{\delta \to 0}{ \longrightarrow} g_0(A^\dagger P_{\mathcal{N}(A)^\bot} z)+  g_\dagger(A^\dagger P_{\mathcal{N}(A)^\bot} z). 
\end{equation}

Due to (iv) for any fixed $z\in \mathcal{R}_{p_X}(A)$ the representation of $\hat{\psi}_\delta$ in \eqref{eq:psi_delta_aux1} is well-defined for sufficiently small $\delta$. 
Consequently, we have convergent sequences in the nominator and in the denominator such that the quotient converges by standard sequence arguments. As $\mathcal{R}_{p_X}(A)\subset \mathcal{R}(A)=\mathcal{N}(A)^\bot$ we thus obtain the desired pointwise convergence
\begin{equation}
\hat{\psi}_\delta(z) \underset{\delta \to 0}{ \longrightarrow} \frac{g_\dagger(A^\dagger z) + g_0(A^\dagger z)}{p_{X,\mathcal{N}(A)^{\bot}}(A^\dagger z)} = A^\dagger z + \int_{\mathcal{N}(A)}  \frac{p_X(x_0 + A^\dagger z)}{\int_{\mathcal{N}(A)} p_X(x'_0 + A^\dagger z)\ \mathrm{d} x'_0 } x_0 \ \mathrm{d} x_0. 
\end{equation}

\end{proof}

\section{MAP estimation and its interpretation as an iResNet}
\label{app:MAP_resnet_Bayesian}

Consider a linear inverse problem $\tilde{A} x = y$ with $\tilde{A} \in L(X,Y)$ as in Remark \ref{rem:general_inverse_problem}, where noisy data $y^\delta \in Y$ is given.
One established approach in Bayesian inverse problems is the so-called MAP (maximum a posteriori) estimator
\begin{equation}
    x_{\text{MAP}} = \mathrm{arg} \max_{x \in X} p(x | y^\delta).
\end{equation}
The posterior density $p(x|y^\delta)$ can be derived via Bayes rule from the pdf's of the prior ($x \sim p_X$), the noise ($y^\delta - Ax \sim \tilde{p}_H$) and the data ($y^\delta \sim p_Y$).
Using this and the monotonicity of the logarithm, one obtains
\begin{align}
    & & x_{\text{MAP}} &= \mathrm{arg} \max_{x \in X} \frac{\tilde{p}_H(y^\delta - \tilde{A} x) p_X(x)}{p_Y(y^\delta)} \notag \\
    &\Leftrightarrow & x_{\text{MAP}} &= \mathrm{arg} \max_{x \in X} \tilde{p}_H(y^\delta - \tilde{A}x) p_X(x) \notag \\
    &\Leftrightarrow & x_{\text{MAP}} &= \mathrm{arg} \min_{x \in X} - \log(\tilde{p}_H(y^\delta - \tilde{A}x))  - \log(p_X(x)).
\end{align}
Here, one can observe a well-known similarity to variational regularization schemes. In the case of Gaussian noise with noise level $\delta > 0$, i.e.,
\begin{equation}
    \tilde{p}_H(\tilde{\eta}) \propto \exp \left( -\frac{1}{2 \delta^2} \|\tilde{\eta}\|^2 \right)
\end{equation}
it holds
\begin{equation}
    x_{\text{MAP}} = \mathrm{arg} \min_{x \in X} - \frac{1}{2} \|\tilde{A}x - y^\delta\|^2  - \delta^2 \log(p_X(x)).
\end{equation}

The negative log-likelihood (NLL) $- \log p_X$ can be interpreted as a penalty term, weighted with the squared noise level $\delta^2$.
If $- \log p_X$ is differentiable, we can use the first-order optimality condition and derive
\begin{align}
    & & 0 &= \tilde{A}^* (\tilde{A} x_{\text{MAP}} - y^\delta) - \delta^2 \partial ( \log p_X )(x_{\text{MAP}}) \notag \\
    &\Rightarrow & \tilde{A}^* y^\delta &= \tilde{A}^* \tilde{A} x_{\text{MAP}} - \delta^2 \partial ( \log p_X )(x_{\text{MAP}}) \notag \\
   &\Rightarrow & x_{\text{MAP}} &= \left(\Id - \left( \Id - \tilde{A}^* \tilde{A} x_{\text{MAP}} + \delta^2 \partial ( \log p_X ) \right) \right)^{-1}(\tilde{A}^* y^\delta), \label{eq:MAP_resnet}
\end{align}
where the last implication only holds if $\tilde{A}^* \tilde{A} - \delta^2 \partial (\log p_X)$ is invertible (which is guaranteed, e.g., in case of a convex NLL).

Now, we can interpret \eqref{eq:MAP_resnet} as an iResNet approach for solving the inverse problem
\begin{equation}
    Ax = z
\end{equation}
where $A = \tilde{A}^* \tilde{A}$ and $z = \tilde{A}^* y$. It holds $x_{\text{MAP}} = \varphi_\theta^{-1}(z^\delta) = (\Id - f_\theta)^{-1} (\tilde{A}^* y^\delta)$ if the residual layer is given by
\begin{equation}
    f_\theta = \Id - A + \delta^2 \partial (\log p_X).
\end{equation}
Thus, MAP estimation with an iResNet is possible as long as the above-defined $f_\theta$ has a Lipschitz constant of at most $L < 1$.

We can derive conditions for the prior $p_X$ and the noise level $\delta$ from this Lipschitz constraint by making use of Assumption \ref{ass:independence}, i.e., stochastic independence of the components $x_j \sim p_{X,j}$ and the eigendecomposition of $A$. In this setting, the components can be handled separately and, thus, 
\begin{equation}
    f_{\theta, j} = (1-\sigma_j^2)\, \Id + \delta^2 \partial (\log p_{X,j}).
\end{equation}
To obtain further insights, we distinguish between large eigenvalues (i.e., $\sigma_j^2 > 1-L$) and small ones (i.e., $\sigma_j^2 \leq 1-L$).

\begin{remark}\label{rem:MAP Lipschitz}
    The prior $p_{X, j}$ corresponding to a large eigenvalue can have a rather arbitrary character. The only important property is that the derivative of the NLL (i.e., $\partial (\log p_{X,j})$) must be Lipschitz continuous. Because in this case, for small enough $\delta$, it holds
    \begin{equation}
        \mathrm{Lip}(f_{\theta,j}) \leq (1-\sigma^2) + \delta^2 \mathrm{Lip}( \partial (\log p_{X,j}) ) \leq L.
    \end{equation}
    This is visualized in the left plot of Figure \ref{fig:MAP_resnet}.

    However, for smaller eigenvalues, it holds  $1-\sigma^2 \geq L$. Hence, the prior must decrease the slope of $f_{\theta,j}$. This holds true if the NLL of the prior is convex (or even strongly convex). Because then, $\partial (\log p_{X,j})$ is monotonously decreasing and $\delta$ can again be chosen s.t.\
    \begin{equation}
        \mathrm{Lip}(f_{\theta,j}) = \mathrm{Lip}((1 - \sigma^2) \Id + \partial (\log p_{X,j})) \leq L
    \end{equation}
    holds. An example for such a $p_{X,j}$ is a Gaussian prior, where $\partial (\log p_{X,j})$ is a linear function with a negative slope (see the right plot of Figure \ref{fig:MAP_resnet} and the subsequent derivations).
\end{remark}

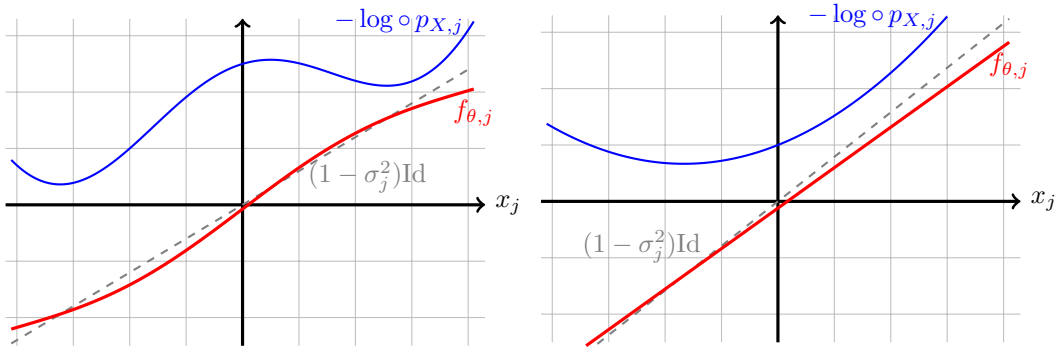
\begin{figure}[ht]
\centering
\begin{tikzpicture}[scale=0.75]
  \draw[step=1.0, lightgray, thin] (-4.2,-2.5) grid (4.3,3.3);
  \draw[->, thick, very thick] (-4.2, 0) -- (4.3, 0) node[right] {$x_j$};
  \draw[->, thick, very thick] (0, -2.5) -- (0, 3.3);
  \draw[scale=1, domain=-4.1:4.1, variable=\t, gray, thick, dashed] plot[samples=161] ({\t}, {(1-0.4)*\t});
    \draw[scale=1, domain=-4.1:4.1, variable=\t, blue, thick] plot[samples=161] ({\t}, {0.15*\t*\t +0.3*\t + 1 + 1.5*cos(45*\t)}) node[left] {$-\log \circ \, p_{X,j}$};
\draw[gray] (1,0.5) node[right] {$(1-\sigma_j^2) \Id$};
\draw[scale=1, domain=-4.1:4.1, variable=\t, red, very thick] plot[samples=161] ({\t}, {(1-0.4)*\t - 0.25*(0.3*\t + 0.3 -1.5*3.14159/4*sin(45*\t))}) node[below] {$f_{\theta,j}$};
\end{tikzpicture}
\begin{tikzpicture}[scale=0.75]
  \draw[step=1.0, lightgray, thin] (-4.2,-2.5) grid (4.3,3.3);
  \draw[->, thick, very thick] (-4.2, 0) -- (4.3, 0) node[right] {$x_j$};
  \draw[->, thick, very thick] (0, -2.5) -- (0, 3.3);
  \draw[scale=1, domain=-3.2:4.1, variable=\t, gray, thick, dashed] plot[samples=161] ({\t}, {(1-0.21)*\t});
    \draw[scale=1, domain=-4.1:3, variable=\t, blue, thick] plot[samples=161] ({\t}, {0.12*\t*\t +0.4*\t + 1}) node[left] {$-\log \circ \, p_{X,j}$};
\draw[scale=1, domain=-3.4:4.1, variable=\t, red, very thick] plot[samples=161] ({\t}, {(1-0.21)*\t - 0.3*(0.24*\t + 0.4)}) node[below] {$f_{\theta, j}$};
\draw[gray] (-1.2,-0.8) node[left] {$(1-\sigma_j^2) \Id$};
\end{tikzpicture}
\caption{MAP estimation with iResNet. For large singular values $\sigma_j^2 > 1-L$, the NLL is not very restricted (can be nonconvex) to guarantee $\mathrm{Lip}(f_{\theta, j}) \leq L < 1$. For small singular values $\sigma_j^2 \leq 1-L$, the NLL has to be (strongly) convex to guarantee that the slope of $f_{\theta, j}$ is smaller than $1-\sigma_j^2$ (and smaller than $L$).}
\label{fig:MAP_resnet}
\end{figure}

\paragraph{Example prior distributions:}
To exemplify the previous observations, let us look at two commonly used prior distributions, namely the Gaussian distribution and the Laplace distribution. In the case of the Gaussian distribution we have 
\begin{equation}
    p_{X,j}(x) = \frac{1}{\sqrt{2\pi b_j^2}} \exp{\left(-\frac{(x-\mu_j)^2}{2b_j^2}\right)}
\end{equation}
for all $j\in\N$ with mean $\mu_j\in\R$ and variance $b_j^2>0$. Consequently, for the residual layer we obtain
\begin{equation}
    f_j(x) = \left(1 - \sigma_j^2 - \frac{\delta^2}{b_j^2}\right) x + \frac{\delta^2 \mu_j}{b_j^2} \quad \text{for } x\in\R
\end{equation}
with Lipschitz constant $\text{Lip}(f_j) = \vert 1 - \sigma_j^2 - \nicefrac{\delta^2}{b_j^2} \vert$. Hence, similar to the observations in the previous remark, for singular values with $1-\sigma_j^2 < L$ the Lipschitz constraint $\text{Lip}(f_j) \leq L$ is fulfilled for all $\delta$ and $b_j$. In the case $1-\sigma_j^2 \geq L$, $\delta$ and $b_j$ need to satisfy $\nicefrac{\delta^2}{b_j^2} \geq 1-\sigma_j^2-L$ to guarantee $\text{Lip}(f_j) \leq L$. \\
In the case of the prior distribution being a Laplacian, i.e.\
\begin{equation}
    p_{X,j}(x) = \frac{1}{\sqrt{2 b_j}} \exp{\left(-\frac{\vert x-\mu_j \vert}{b_j}\right)} \quad \text{for } x\in\R
\end{equation}
with mean $\mu_j\in\R$ and variance $2b_j^2>0$, the subgradient of $\log p_{X,j} $ is given by
\begin{equation}
    \partial (\log p_{X,j})(x) = \begin{cases}
        \frac{1}{b_j} & x < \mu_j \\
        [-\nicefrac{1}{b_j}, \nicefrac{1}{b_j}] & x = \mu_j \\
        -\frac{1}{b_j} & x>\mu_j
    \end{cases}  \quad \text{for } x\in\R.
\end{equation}
Consequently, the residual layer for the Laplacian prior distribution is given by 
\begin{equation}
    f_j(x) = \begin{cases}
        (1-\sigma_j^2)\,x+\nicefrac{\delta^2}{b_j} & x < \mu_j \\
        [(1-\sigma_j^2)\,x-\nicefrac{\delta^2}{b_j}, (1-\sigma_j^2)x+\nicefrac{\delta^2}{b_j}] & x = \mu_j \\
        (1-\sigma_j^2)\,x-\frac{\delta^2}{b_j} & x>\mu_j
    \end{cases}  \quad \text{for } x\in\R,
\end{equation}
which is not Lipschitz-continuous. As a result, the Laplace distribution does not satisfy the conditions of Remark~\ref{rem:MAP Lipschitz}.

In summary, the previous considerations illustrate that MAP estimation with a Gaussian noise model can also be represented by the proposed iResNet approach for certain prior distributions guaranteeing invertibility. 

\section{Approximation training in diagonal architecture with dependent data and noise distribution}\label{appendix:approx_training_general_noise}

In Section~\ref{sec:approx_train_Bayesian}, we assumed that the random variables $x\sim p_X$ and $\eta \sim p_H$ are independent. However, one can obtain a more general version of Lemma~\ref{lem:approx_training_linear} with less restrictive assumptions on the joint data and noise distribution. To this end, we denote by $p \colon \R^2 \to [0, \infty)$ the joint probability density function with marginal distributions  $p_X(x)=\int_\R p(x,\eta) \, \mathrm{d}\eta$,  $p_H(\eta )=\int_\R p(x,\eta) \, \mathrm{d}x$ and assume that the respective first and second moments exist. 
In this setting, the minimization problem of the approximation training reads
\begin{equation} \label{eq:bayes_train_probl_comp_wise_general_noise}
    \min_{f \in \mathcal{F}} \int_{\R^2} p(x,\eta)\, | (1 - \sigma^2) x - \eta- f(x) |^2 \, \mathrm{d}(x,\eta)
\end{equation}
and the subsequent lemma provides a closed-form solution of the minimizer of problem~\ref{eq:bayes_train_probl_comp_wise_general_noise} in the case $L < 1-\sigma^2$.
\begin{lemma} \label{lem:approx_training_linear_general_noise_model}
Let $\mathcal{F} = \{f \in C(\R) \, | \, \exists\, m \in [-L, L], b \in \R \colon \,f(x) = mx + b\}$ and $L < 1-\sigma^2$. Then,
\begin{equation}
    f^*(\hat{x}) = \begin{cases} L \hat{x} + (1-\sigma^2-L) \mu_X - \mu_H& \text{ if } \frac{\mathrm{Cov}_{p}(x,\eta) }{\mathrm{Var}_{p_X}(x)} < 1-\sigma^2 - L \\ 
    \left( 1-\sigma^2 - \frac{\mathrm{Cov}_{p}(x,\eta) }{\mathrm{Var}_{p_X}(x)} \right) \hat{x} + \frac{\mathrm{Cov}_{p}(x,\eta) }{\mathrm{Var}_{p_X}(x)} \mu_X - \mu_H & \text{ if } \frac{\mathrm{Cov}_{p}(x,\eta) }{\mathrm{Var}_{p_X}(x)} \in [1-\sigma^2 - L,1-\sigma^2 + L] \\ 
    - L \hat{x} + (1-\sigma^2+L) \mu_X - \mu_H & \text{ if } \frac{\mathrm{Cov}_{p}(x,\eta) }{\mathrm{Var}_{p_X}(x)} > 1-\sigma^2 + L  \end{cases} 
\end{equation}
is the unique solution of the minimization problem~\ref{eq:bayes_train_probl_comp_wise_general_noise}, where $\mu_X$, $\mu_H$ denote the expectation values of the marginal distributions, $\mathrm{Cov}_{p}$ the covariance w.r.t.\ $(x,\eta)\sim p$ and $\mathrm{Var}_{p_X}$ the variance w.r.t.\ $x\sim p_X$.
\end{lemma}

\begin{proof}
For a function $f$ of the form $f(x) = mx + b$ with the constraint $m^2 \leq  L^2$, we can solve \eqref{eq:bayes_train_probl_comp_wise_general_noise} by using the Lagrangian
\begin{equation}
    K(m,b, \lambda) = \int_{\R^2} p(x,\eta)\, | (1 - \sigma^2-m) x - \eta- b |^2 \, \mathrm{d}(x,\eta) + \lambda (m^2 - L^2).
\end{equation}
Observe that the integral is well-defined due to the existence of the first and second moments of $p$. In addition, the convexity, coercivity, and continuity of the integral term w.r.t. $(m,b)$ implies that there exists a minimizer. The minimizer must satisfy the necessary conditions (KKT conditions) 
\begin{align} \label{eq:dLdm_general_noise}
    \frac{\partial K}{\partial m} (m,b,\lambda) = -2\int_{\R^2} p(x,\eta) \left((1 - \sigma^2 - m) x - \eta - b \right) x \, \mathrm{d}(x,\eta) + 2\lambda m  \stackrel{!}{=} 0, \\
    \label{eq:dLdb_general_noise}
    \frac{\partial K}{\partial b} (m,b,\lambda) = - 2 \int_{\R^2}  p(x,\eta) \left((1 - \sigma^2 - m) x - \eta - b \right) \, \mathrm{d}(x,\eta) \stackrel{!}{=} 0,\\
    \label{eq:lagrange_constraint_general_noise}
    \lambda (m^2 - L^2) \stackrel{!}{=} 0, \\
    \lambda \geq  0.
\end{align}
Exploiting the marginal distribution $p_X$, $p_H$ and rearranging \eqref{eq:dLdm_general_noise} for $m$ and \eqref{eq:dLdb_general_noise} for $b$ leads to
\begin{align}
m &= \frac{(1-\sigma^2)\mathbb{E}_{ p_X}(x^2) - b\, \mu_X - \mathbb{E}_{ p}(x\cdot \eta)}{\mathbb{E}_{ p_X}(x^2) + \lambda}, \\
    b  &= (1 - \sigma^2 - m) \mu_X - \mu_H, \label{eq:b_general_noise}
\end{align}
where we use the abbreviated notation $\mathbb{E}_{ p}$ for $\mathbb{E}_{(x,\eta)\sim p}$, $\mathbb{E}_{ p_X}$ for $\mathbb{E}_{x\sim p_X}$ and $\mathbb{E}_{ p_H}$ for $\mathbb{E}_{\eta\sim p_H}$. 
Combining both equations yields
\begin{align}
    & & \left(1 - \frac{\mu_X^2}{\mathbb{E}_{p_X}(x^2) + \lambda} \right) m &= \frac{(1-\sigma^2) \left(\mathbb{E}_{p_X}(x^2) -  \mu_X^2 \right) - \left( \mathbb{E}_p(x\cdot \eta) - \mu_X\, \mu_H \right)}{\mathbb{E}_p(x^2) + \lambda}\notag \\
    &\Leftrightarrow & \frac{\mathrm{Var}_{p_X}(x) + \lambda}{\mathbb{E}_{p_X}(x^2) + \lambda}  m &= \frac{(1-\sigma^2)\mathrm{Var}_{p_X}(x) - \mathrm{Cov}_{p}(x,\eta)}{\mathbb{E}_{p_X}(x^2) + \lambda} \notag \\
    &\Leftrightarrow & m &= (1-\sigma^2)\frac{\mathrm{Var}_{p_X}(x) }{\mathrm{Var}_{p_X}(x)+ \lambda} - \frac{\mathrm{Cov}_{p}(x,\eta) }{\mathrm{Var}_{p_X}(x)+ \lambda}. \label{eq:m_general_noise}
\end{align}
In order to determine the value of $\lambda$, we need to distinguish two cases.
\begin{itemize}
    \item[(I):] $\mathrm{Cov}_{p}(x,\eta) \leq 0$: \\
    Since $1-\sigma^2 > L$ holds by assumption, the case $\lambda=0$ is not possible and we need $\lambda >0$ to ensure $m\leq  L$. Then, \eqref{eq:lagrange_constraint_general_noise} directly implies $m = L$ as \eqref{eq:m_general_noise} cancels out the possibility $m=-L$. Thus we also know $b = (1-\sigma^2 - L) \mu_X - \mu_H$.
\item[(II):] $\mathrm{Cov}_{p}(x,\eta) > 0$: \\
In this case, we need to distinguish the cases $\lambda=0$ and $\lambda>0$.
\begin{itemize}
    \item[(IIa):] $\lambda=0$:\\ 
    In this case, we have 
    \begin{equation}
        m= (1-\sigma^2) - \frac{\mathrm{Cov}_{p}(x,\eta) }{\mathrm{Var}_{p_X}(x)}.
    \end{equation}
 The constraint that $m\leq L$ and $m\geq -L$ only holds true if 
\begin{align}
    & \mathrm{Cov}_{p}(x,\eta) \geq (1-\sigma^2-L) \mathrm{Var}_{p_X}(x) \\
   \land \ & \mathrm{Cov}_{p}(x,\eta) \leq (1-\sigma^2+L) \mathrm{Var}_{p_X}(x) 
\end{align}
taking into account that $1-\sigma^2>L$.
 
    \item[(IIa):] $\lambda>0$:\\ 
    In this case, we can have either $m=L$ or $m=-L$. 
    Rearranging \eqref{eq:m_general_noise} yields
    \begin{equation}
        \lambda= \frac{1}{m} \left( (1-\sigma^2-m) \mathrm{Var}_{p_X}(x) - \mathrm{Cov}_{p}(x,\eta) \right).
    \end{equation}
    From this we deduce that $\lambda>0$ holds if either
    \begin{itemize}
        \item $m=L$ and $\mathrm{Cov}_{p}(x,\eta) < (1-\sigma^2-L) \mathrm{Var}_{p_X}(x)$, or
        \item $m=-L$ and $\mathrm{Cov}_{p}(x,\eta) >(1-\sigma^2+L) \mathrm{Var}_{p_X}(x)$.
    \end{itemize}
\end{itemize}
    
\end{itemize}

Exploiting \eqref{eq:b_general_noise} yields $b$ in either case, which provides the desired $f^*$. In combination with the observation that $m$ and $b$ are uniquely determined the proof is complete.
\end{proof}

\section{Additional numerical experiments}
\label{app:numerics_radon}

The numerical results in Section~\ref{sec:numeric_experiments} are illustrated for the convolution operator $A=M_a$. In the following additional illustrations for the convolution operator and all corresponding results for the Radon operator, described in Section~\ref{sec:numeric_experiments}, are provided.

\begin{figure}[h]
    \centering
    \includegraphics[width=.97\textwidth]{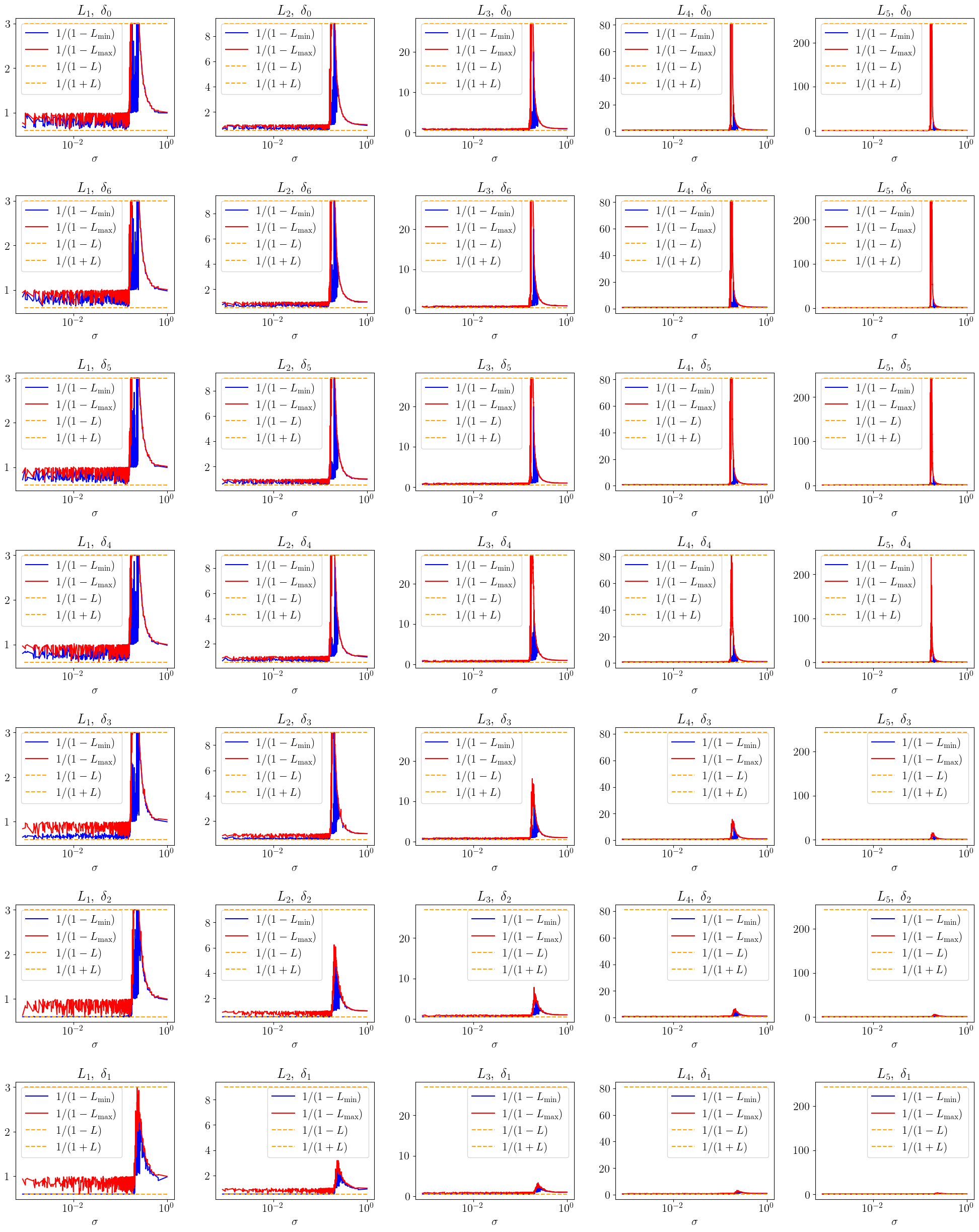}
    \caption{$\frac{1}{1-L_\mathrm{min}}$~\textit{(blue)} for minimal slopes $L_\mathrm{min}$ estimated for $f_{\theta(L_m,\delta_\ell)}$ and $\frac{1}{1-L_\mathrm{max}}$~\textit{(red)} for maximum slope $L_\mathrm{max}$ of $f_{\theta(L_m,\delta_\ell)}$ over singular values of $A=M_a$ of trained networks with Lipschitz constraint $L=L_m$ at $m=1,\hdots,5$ and noise levels $\delta_\ell$ with $\ell=0,\hdots,6$ via reconstruction training on the MNIST dataset.}
    \label{fig:SW_trueL_conv}
\end{figure}

\begin{figure}[h]
    \centering
    \begin{subfigure}[b]{0.48\textwidth}
    \centering
        \includegraphics[width=7.8cm]{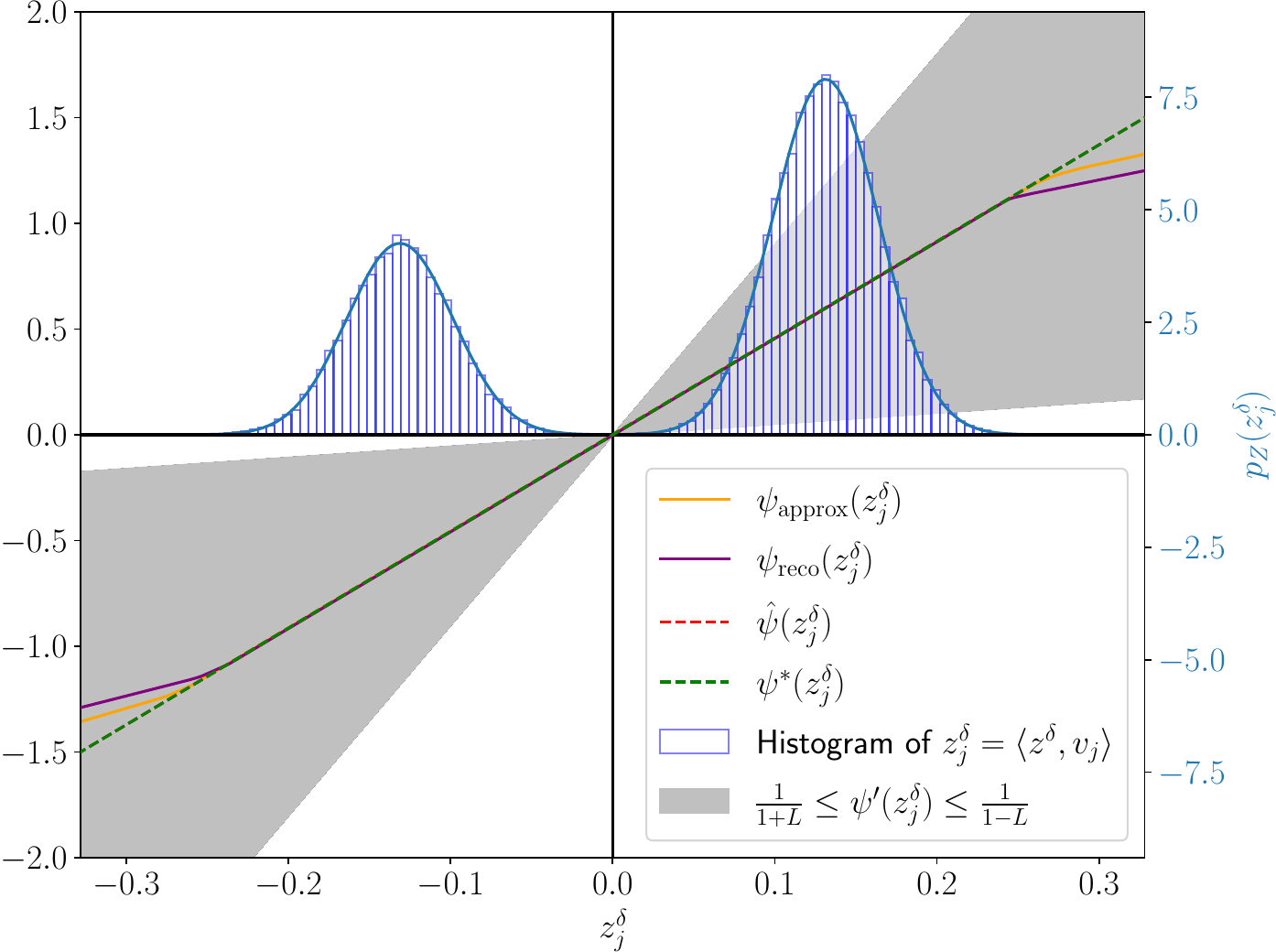}
    \end{subfigure}
    \hfill
    \begin{subfigure}[b]{0.48\textwidth}
    \centering
        \includegraphics[width=7.8cm]{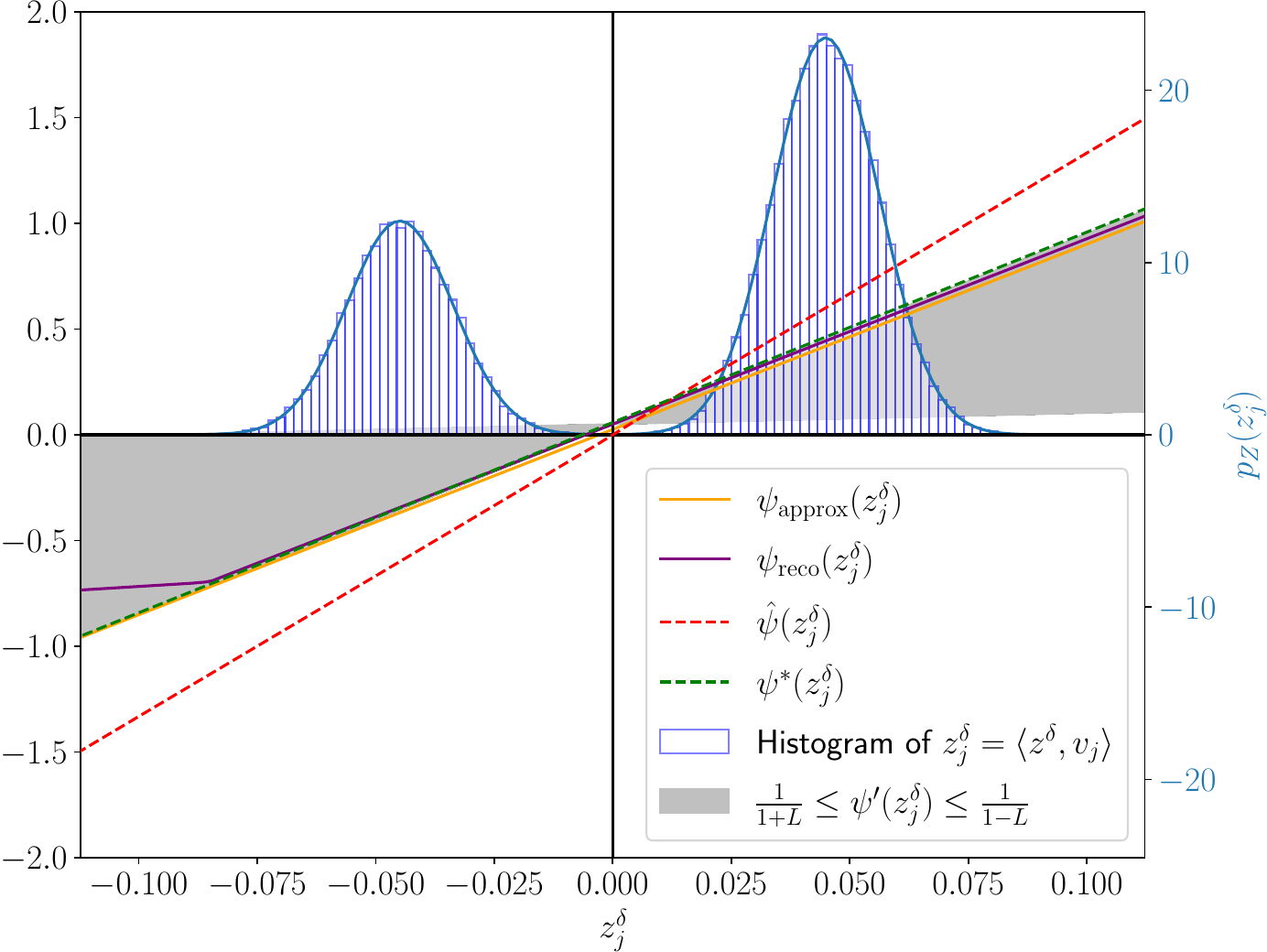}
    \end{subfigure}

    \vspace{2.0ex}
    \begin{subfigure}[b]{0.48\textwidth}
    \centering
        \includegraphics[width=7.8cm]{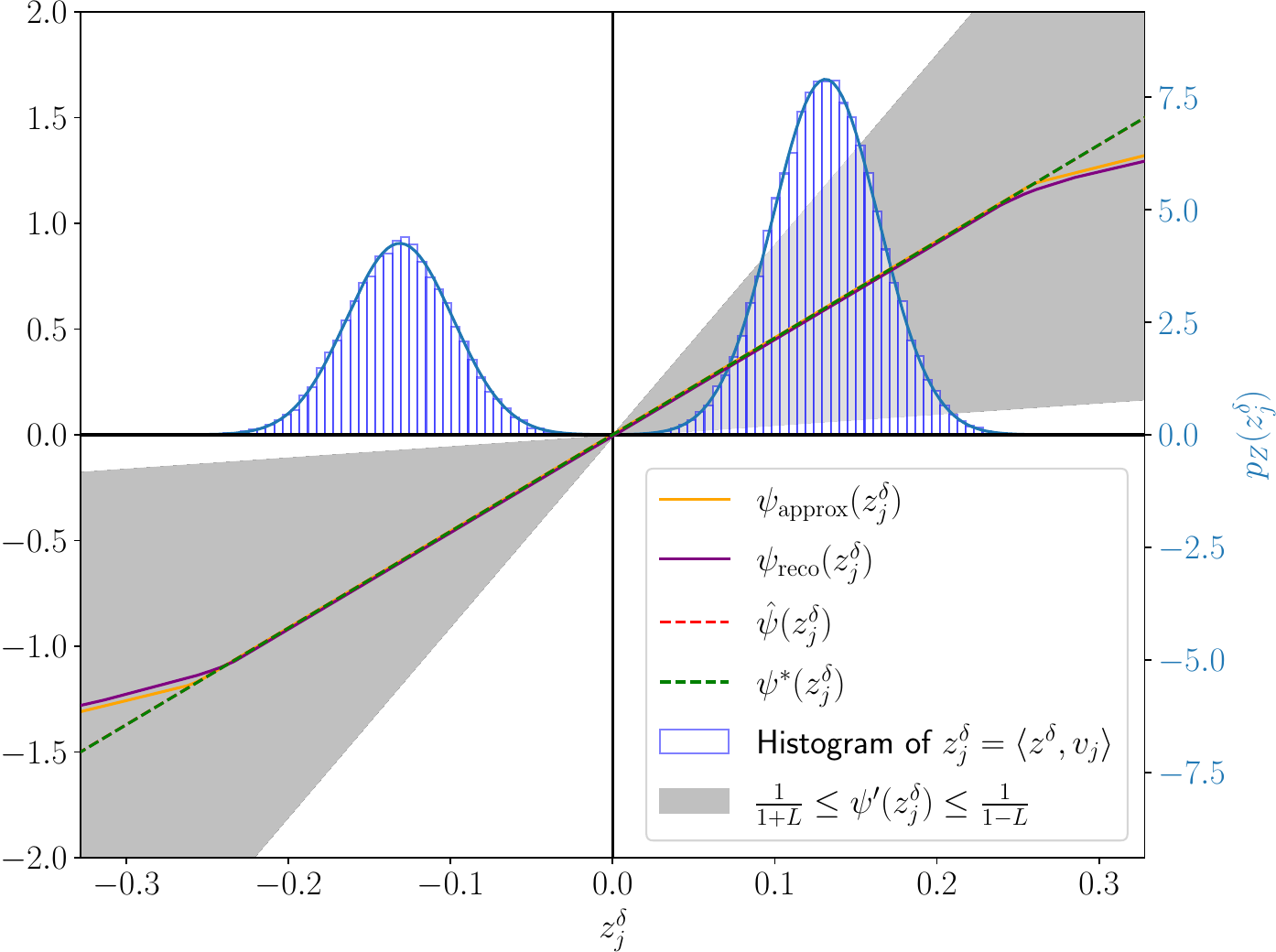}
    \end{subfigure}
    \hfill
    \begin{subfigure}[b]{0.48\textwidth}
    \centering
        \includegraphics[width=7.8cm]{images/solutionplots/v4/radon/radon_bimodal_n=2_m=6_50_0.074832.pdf}
    \end{subfigure}

    \vspace{2.0ex}
    \begin{subfigure}[b]{0.48\textwidth}
    \centering
        \includegraphics[width=7.8cm]{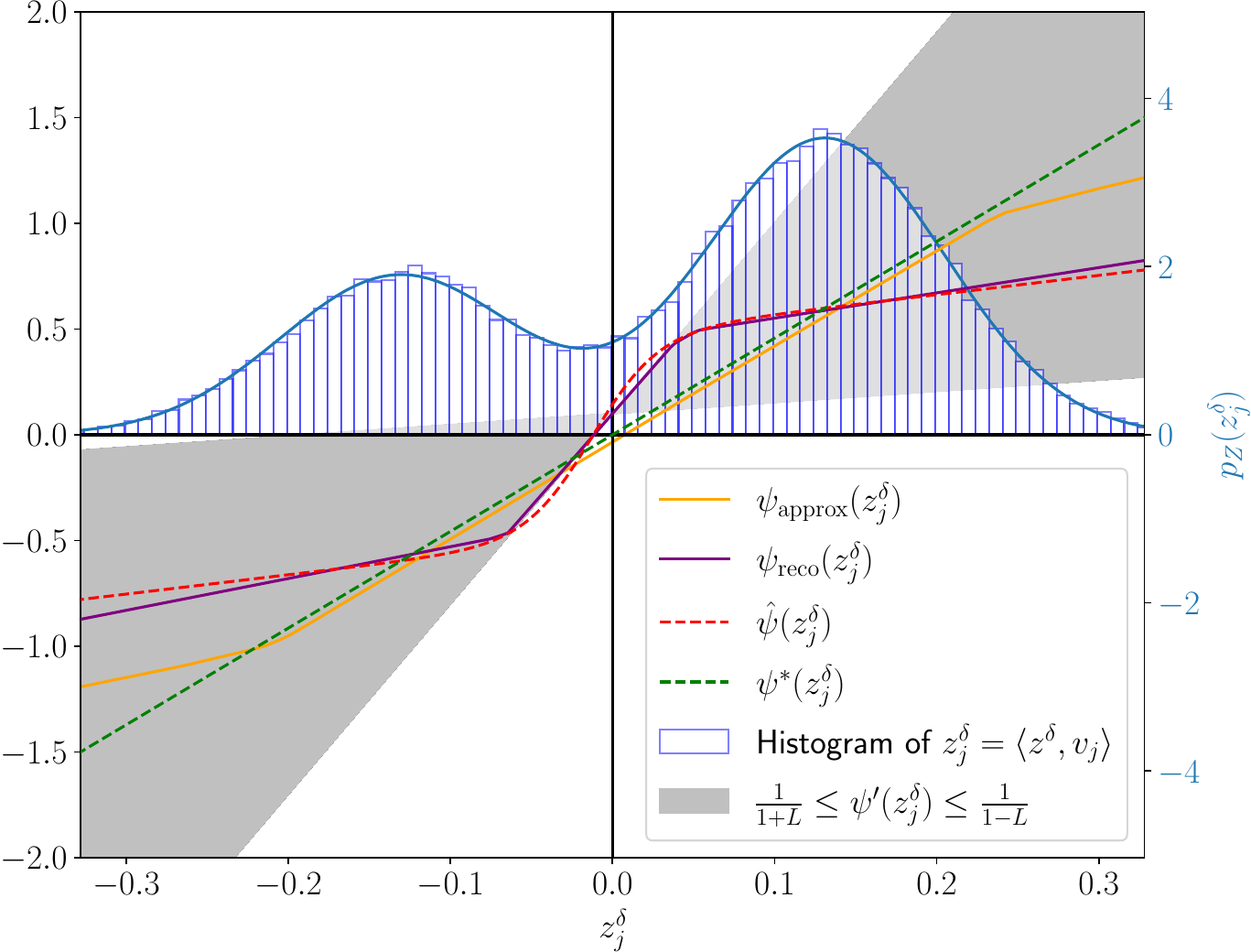}
        \subcaption{$\sigma_j^2=0.219, \ j = 8$}
    \end{subfigure}
    \hfill
    \begin{subfigure}[b]{.48\textwidth}
    \centering
        \includegraphics[width=7.8cm]{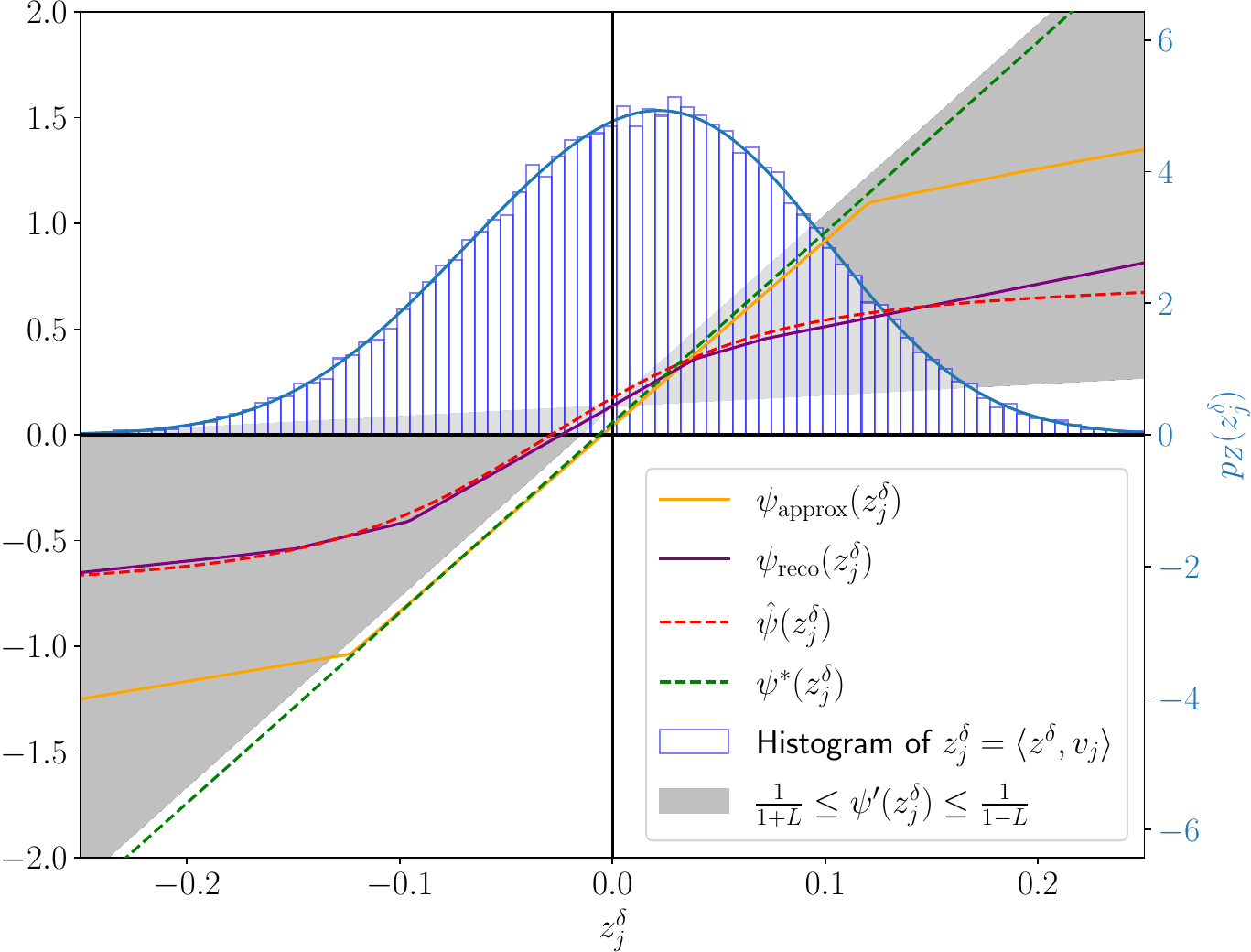}
        \subcaption{$\sigma_j^2=0.075, \ j = 51$}
    \end{subfigure}
    
    \caption{Reconstructions $\psi_\mathrm{approx}^\ast(z_j^\delta)$ trained via approximation training and  $\psi_\mathrm{reco}^\ast(z_j^\delta)$ trained via reconstruction training at Lipschitz bound $L_2$ for different singular values and 
    for noise levels "zero"~\textit{($\delta_0$, top row)}, "small"~\textit{($\delta_6$, middle row)} and "large"~\textit{($\delta_2$, bottom row)} for $\tilde{A}$ the Radon operator.}
\end{figure}

\begin{figure}[h]
    \centering
    \begin{subfigure}{\textwidth}
        \centering
        \includegraphics[width=.9\textwidth]{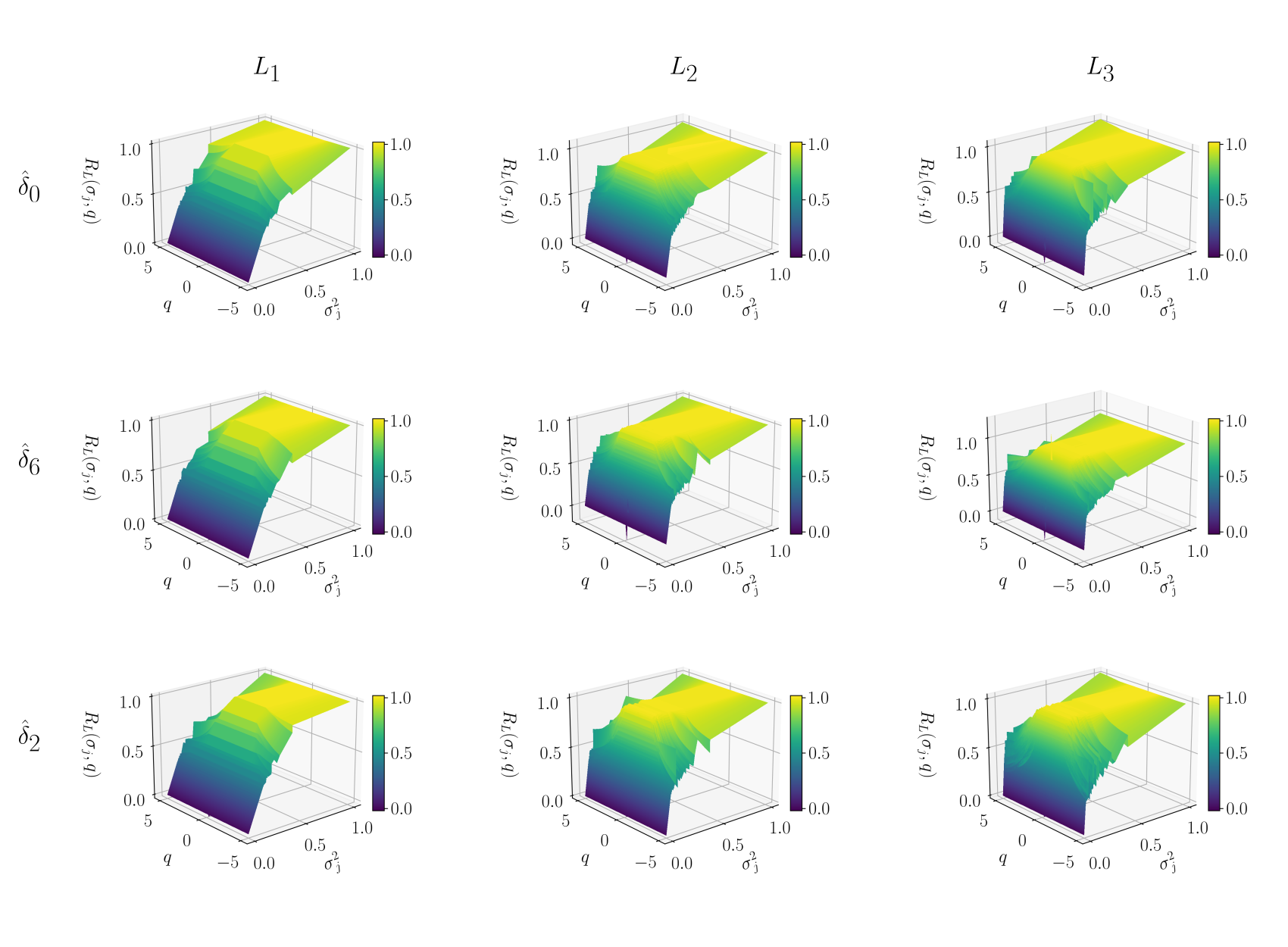}
    \end{subfigure}

    \vspace{-5.0ex}
    \begin{subfigure}{\textwidth}
        \centering
        \includegraphics[width=.9\textwidth]{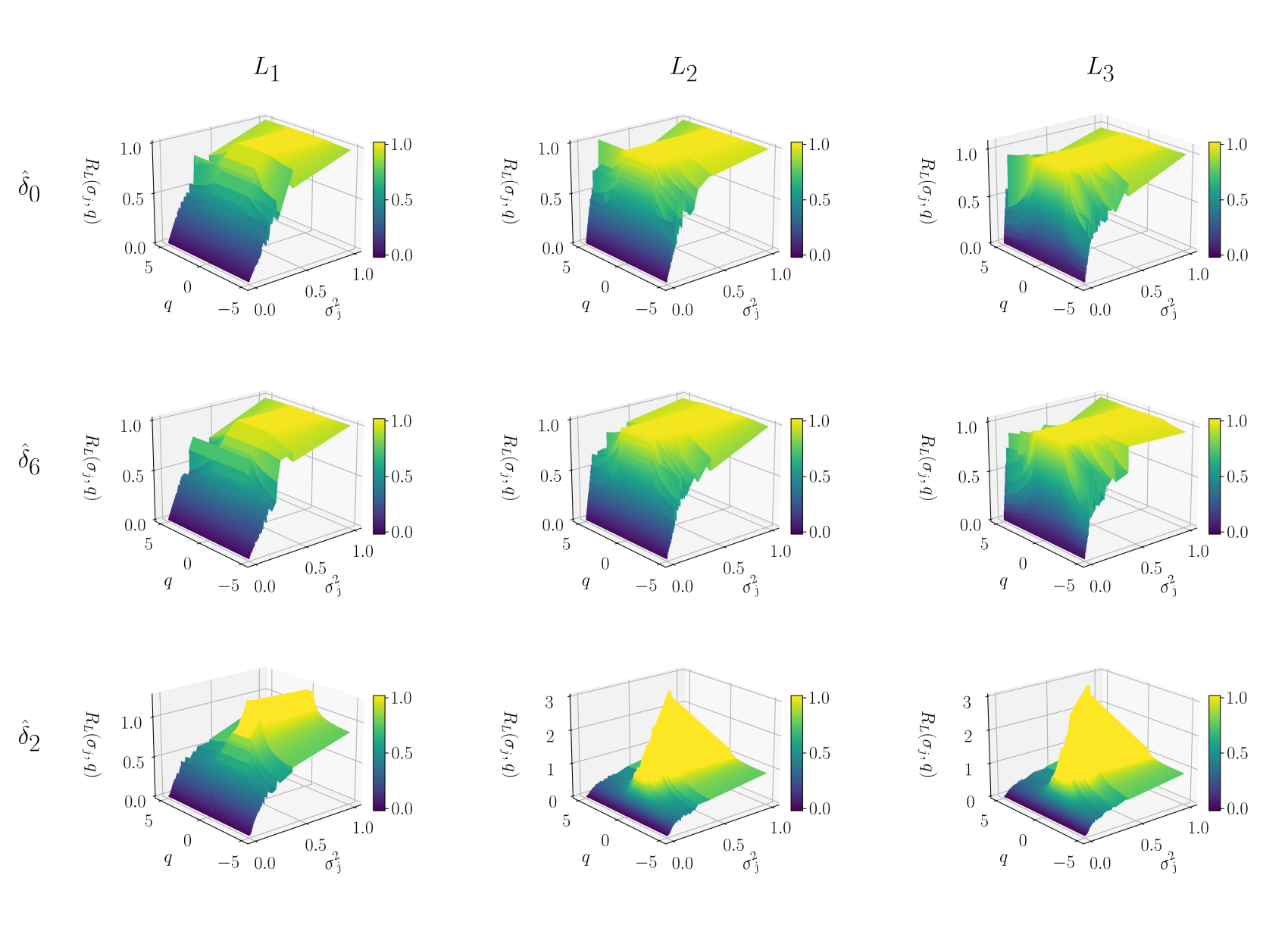}
    \end{subfigure}
    \caption{Filter functions $R_L(\sigma_j,q)$ as defined in~\eqref{eq:filter_function_numeric} corresponding to trained networks $\varphi_{\theta(L_m,\delta_\ell)}$ for $m=1,2,3$~\textit{(columns)} and $\ell = 0,6,2$~\textit{(rows)}, trained via approximation training~\textit{(top)} and via reconstruction training~\textit{(bottom)} on the bimodal dataset for $\tilde{A}$ the Radon operator}
\end{figure}

\begin{figure}[h]
    \centering
    \begin{subfigure}{\textwidth}
        \centering
        \includegraphics[width=0.75\textwidth]{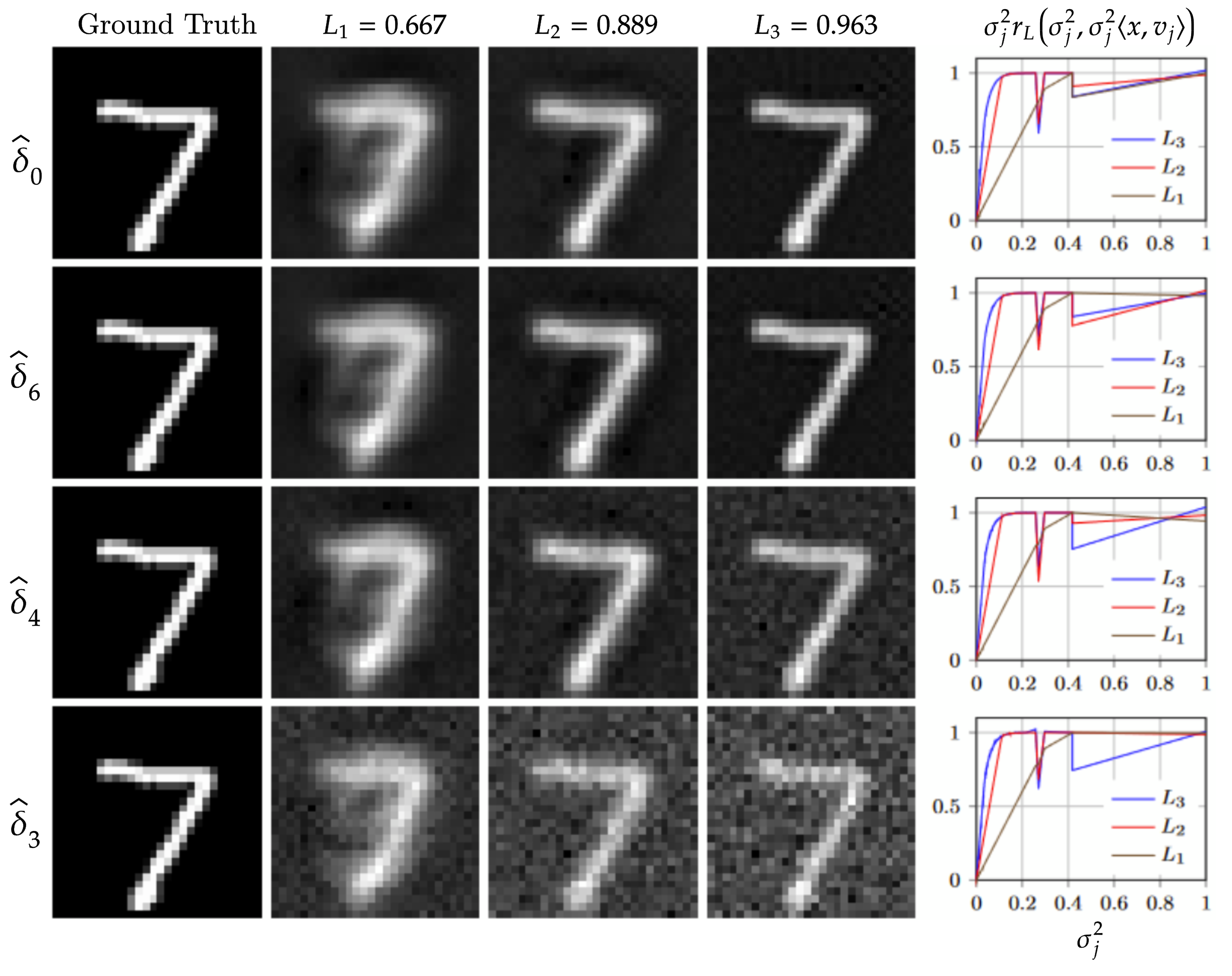}
    \end{subfigure}

    \begin{subfigure}{\textwidth}
        \centering
        \includegraphics[width=0.75\textwidth]{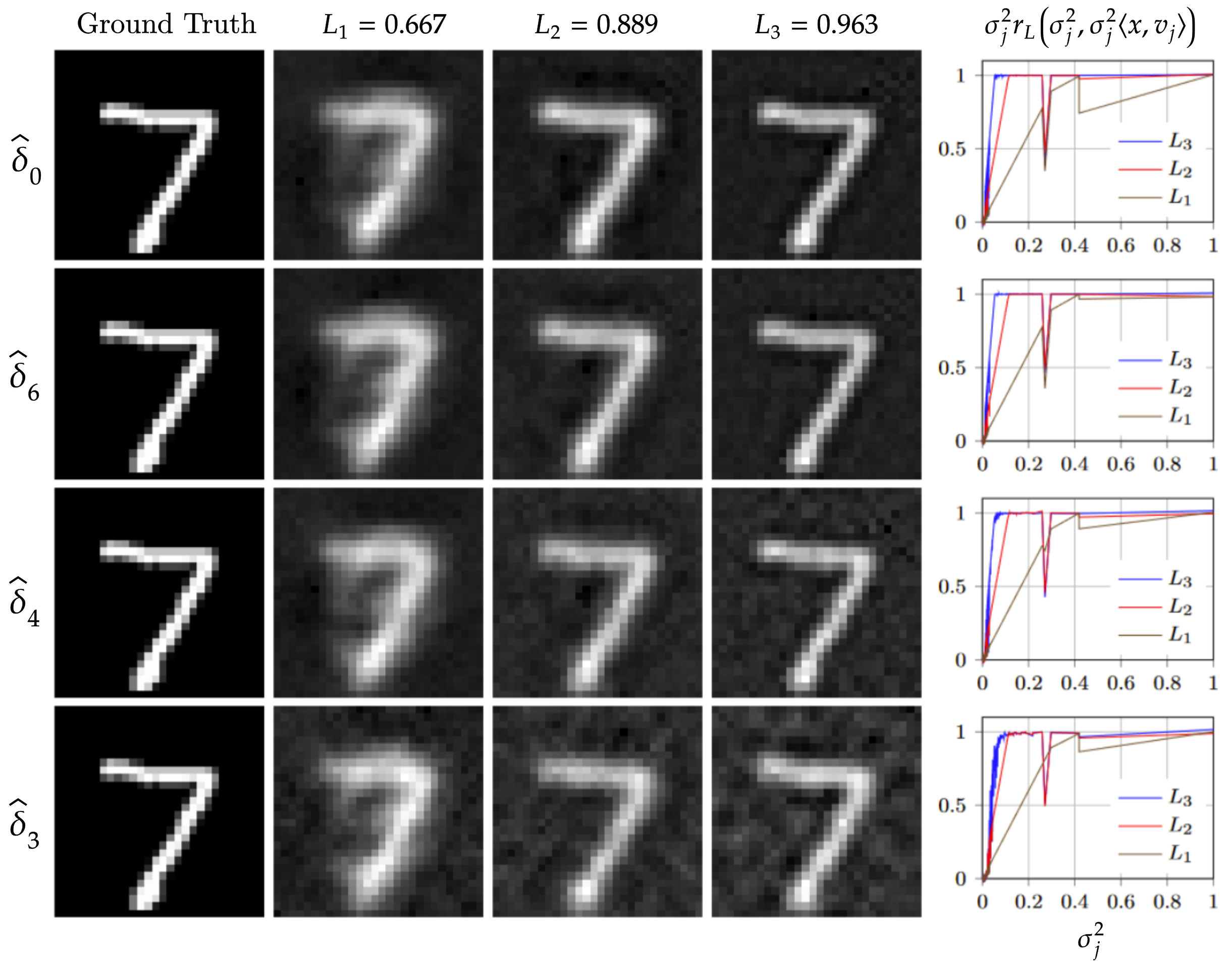}
    \end{subfigure}
    \caption{Reconstructions of an MNIST sample $x=x^{(1)}$ from the test dataset by computing $\varphi_{\theta(L_m,\delta_\ell)}^{-1}(Ax+\tilde{\eta})$ with $\tilde{\eta} \sim \mathcal{N}(0,\delta_\ell \Id)$ for Lipschitz bounds $L_m$ with $m=1,2,3$~\textit{(columns)} and noise levels $\delta_\ell= \hat{\delta}_\ell \cdot \mathrm{std}_\mathrm{MNIST}$ with $\ell =0,6,4,3$~\textit{(rows)} together with corresponding filter functions for $\tilde{A}$ the Radon operator. The top subfigure depicts the reconstructions from networks trained via approximation training and the bottom subfigure corresponds to the networks optimized via reconstruction training.}
    \label{fig:reconstructions_radon}
\end{figure}
\FloatBarrier
\setlength{\tabcolsep}{0.8em} 
\begin{table}[h]
\centering
\begin{tabular}{|c|ccc|ccc|}
\hline
approximation & \multicolumn{3}{c|}{SSIM} & \multicolumn{3}{c|}{MSE} \\ \cline{2-7}
training & $L_1$      & $L_2$    & $L_3$  & $L_1$      &$L_2$    & $L_3$    \\ \hline
$\delta_0$                                                                                        & 0.4667   & 0.7998 & 0.9224  & 0.0308   & 0.0117 & 0.0043 \\
$\delta_6$                                                                                        & 0.4628   & 0.7940 & 0.9179 & 0.0306   & 0.0119 & 0.0043 \\
$\delta_4$                                                                                        & 0.4597   & 0.7547 & 0.7832  & 0.0309   & 0.0129 & 0.0075 \\
$\delta_3$                                                                                        & 0.4559   & 0.6150 & 0.5572  & 0.0310   & 0.0194 & 0.0313 \\ \hline
reconstruction & \multicolumn{3}{c|}{SSIM} & \multicolumn{3}{c|}{MSE} \\ \cline{2-7} training & $L_1$      & $L_2$    & $L_3$  & $L_1$      &$L_2$    & $L_3$    \\ \hline
$\delta_0$                                                                                        & 0.4892 & 0.7804 & 0.8952  & 0.0291 & 0.0124 & 0.0045 \\
$\delta_6$                                                                                        & 0.4718   & 0.7786 & 0.8905  & 0.0304   & 0.0126 & 0.0047 \\
$\delta_4$                                                                                        & 0.4671   & 0.7482 & 0.8243  & 0.0300   & 0.0129 & 0.0061 \\
$\delta_3$                                                                                        & 0.4697   & 0.6929 & 0.6964  & 0.0299   & 0.0149 & 0.0121 \\
\hline
\end{tabular}
\caption{SSIM and MSE measures corresponding to reconstructions of $x^{(1)}$ in Figure~\ref{fig:reconstructions_radon}.}
\end{table}

\begin{figure}[h]
\captionsetup[subfigure]{labelformat=empty}
\begin{subfigure}{\textwidth}
    \centering
    \begin{tikzpicture}
        \pgfplotsset{
        every axis plot/.append style={thick},
        tick style={black, thick},
        every axis plot/.append style={line width=0.8pt},
        every axis/.style={
            axis y line=left,
            axis x line=bottom,
            axis line style={ultra thick,->,>=latex, shorten >=-.4cm}
        },
        }
        \begin{groupplot}[
            group style={
                group name=my plots,
                group size=2 by 1,
                ylabels at=edge left,
                horizontal sep=3.0cm
            },
            scale only axis=true,
            width=0.23\linewidth,
            height=3.0cm,
            axis lines=middle, 
            xlabel = {$1-L$},
            ytick = {0.,5,10,15,20,25,30},
            xtick = {0, 250,500, 750},
            ymin = 0.0,
            ymax = 0.22,
             x label style={at={(axis description cs:1.14,0.1)},anchor=north},
             y label style={at={(axis description cs:0.035,1.32)},anchor=north},
        ]
        \nextgroupplot[
            xmode=log,ymode=log, 
            ymin =0.0001, ymax=0.15,
            xmin=0.001, xmax=1,
            axis y line=left,
            ytick={0.001,0.01,0.1},
            xtick={0.001,0.01,0.1,1},
            axis x line=bottom,
            x label style={at={(axis description cs:1.29,0.1)},anchor=north},
            y label style={at={(axis description cs:0.025,1.27)},anchor=north},
            legend style={at={(1.95,1)},draw=none},
            legend cell align={left}
        ]
        \addplot[color=black,dashed, domain=0.001:1]{0.3*x};
        \addplot[color=gray,dashed, domain=0.001:1]{0.3*x^2};
        \addplot[teal,mark=*] table [x=L, y=mean_norm, col sep=comma]{csv_files/lap_radon_approx_pt2.csv};
        \addplot[blue,mark=*] table [x=L, y=x1_norm, col sep=comma]{csv_files/lap_radon_approx_pt2.csv};
        \addplot[red,mark=*] table [x=L, y=x2_norm, col sep=comma]{csv_files/lap_radon_approx_pt2.csv};

        \nextgroupplot[
            xmode=log,ymode=log, 
            ymin =0.0001, ymax=0.15,
            xmin=0.001, xmax=1,
            axis y line=left,
            ytick={0.001,0.01,0.1},
            xtick={0.001,0.01,0.1,1},
            axis x line=bottom,
            x label style={at={(axis description cs:1.29,0.1)},anchor=north},
            y label style={at={(axis description cs:0.025,1.27)},anchor=north},
            legend style={at={(2.02,1.1)},draw=none},
            legend cell align={left}
        ]
        \addplot[color=black,dashed, domain=0.001:1]{0.3*x};
        \addlegendentry{{\footnotesize{$\mathcal{O}\big((1-L)\big)$}}}
        \addplot[color=gray,dashed, domain=0.001:1]{0.3*x^2};
        \addlegendentry{{\footnotesize{$\mathcal{O}\big((1-L)^2\big)$}}}
        \addplot[teal,mark=*] table [x=L, y=mean_norm, col sep=comma]{csv_files/lap_radon_reco_pt2.csv};
        \addlegendentry{{$\mathcal{E}_\mathrm{mean}(\varphi_{\theta(L)},A)$}}
        \addplot[blue,mark=*] table [x=L, y=x1_norm, col sep=comma]{csv_files/lap_radon_reco_pt2.csv};
        \addlegendentry{{$\mathcal{E}_{x^{(1)}}(\varphi_{\theta(L)},A)$}}
        \addplot[red,mark=*] table [x=L, y=x2_norm, col sep=comma]{csv_files/lap_radon_reco_pt2.csv};
        \addlegendentry{{$\mathcal{E}_{x^{(2)}}(\varphi_{\theta(L)},A)$}}
        
        \end{groupplot}
    \end{tikzpicture}
    \end{subfigure}

    \vspace{4.0ex}
 
    \hfill
    \begin{subfigure}[b]{0.145\textwidth}
    \centering

    \vspace{-10.0ex}
    \includegraphics[width=0.75\textwidth]{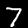}
    \vspace{-1.0ex}
    \caption{\footnotesize{$x^{(1)}$}}

    \phantom{\includegraphics[width=.3\textwidth]{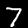}}
    \vfill
    \end{subfigure}
    \begin{subfigure}{0.145\textwidth}
    \centering
    \vspace{-3.0ex}
        \includegraphics[width=0.75\textwidth]{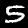}
    \vspace{-1.0ex}
    \caption{\footnotesize{$x^{(2)}$}}
 
    \phantom{\includegraphics[width=.3\textwidth]{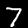}}
    \end{subfigure}
    \hspace{2.0ex}
    \begin{subfigure}{.59\textwidth}
        \centering
    \begin{tikzpicture}
        \pgfplotsset{
        every axis plot/.append style={thick},
        tick style={black, thick},
        every axis plot/.append style={line width=0.8pt},
        every axis/.style={
            axis y line=left,
            axis x line=bottom,
            axis line style={ultra thick,->,>=latex, shorten >=-.4cm}
        },
        }
        \begin{groupplot}[
            group style={
                group name=my plots,
                group size=1 by 1,
                ylabels at=edge left,
                horizontal sep=1.5cm
            },
            scale only axis=true,
            width=0.39\linewidth,
            height=3.0cm,
            axis lines=middle, 
            xlabel = {$1-L$},
            ytick = {0.,5,10,15,20,25,30},
            xtick = {0, 250,500, 750},
            ymin = 0.0,
            ymax = 0.22,
             x label style={at={(axis description cs:1.14,0.1)},anchor=north},
             y label style={at={(axis description cs:0.035,1.32)},anchor=north},
        ]
        
        \nextgroupplot[
                    xmode=log,ymode=log, 
                    ymin =0.04, ymax=0.25,
                    xmin=0.002, xmax=1,
                    axis y line=left,
                    ytick={0.001,0.01,0.1},
                    xtick={0.001,0.01,0.1,1},
                    axis x line=bottom,
                    x label style={at={(axis description cs:1.29,0.1)},anchor=north},
                    y label style={at={(axis description cs:0.025,1.27)},anchor=north},
                    legend style={at={(2.02,1)},draw=none},
                    legend cell align={left}
                ]
                \addplot[teal,mark=*] table [x=L, y=gen_mean_norm, col sep=comma]{csv_files/lap_radon_reco_pt2.csv};
                \addlegendentry{{$\tilde{\mathcal{E}}_\mathrm{mean}(\varphi_{\theta(L)},A)$}}
                \addplot[blue,mark=*] table [x=L, y=x1_gen_norm, col sep=comma]{csv_files/lap_radon_reco_pt2.csv};
                \addlegendentry{{$\tilde{\mathcal{E}}_{x^{(1)}}(\varphi_{\theta(L)},A)$}}
                \addplot[red,mark=*] table [x=L, y=x2_gen_norm, col sep=comma]{csv_files/lap_radon_reco_pt2.csv};
                \addlegendentry{{$\tilde{\mathcal{E}}_{x^{(2)}}(\varphi_{\theta(L)},A)$}}
        \end{groupplot}
    \end{tikzpicture}
    \end{subfigure}
    \caption{Test samples $x^{(1)}$ and $x^{(2)}$~\textit{(bottom left)}. Evaluations of the local approximation property via $\mathcal{E}_\mathrm{mean}(\varphi_{\theta(L_m)},A), \ \mathcal{E}_{x^{(1)}}(\varphi_{\theta(L_m)},A)$ and $\mathcal{E}_{x^{(2)}}(\varphi_{\theta(L_m)},A)$ for the approximation training~\textit{(top left)} and the reconstruction training~\textit{(top right)}, and evaluations of the generalized approximation property via $\tilde{\mathcal{E}}_\mathrm{mean}(\varphi_{\theta(L_m)},A), \tilde{\mathcal{E}}_{x^{(1)}}(\varphi_{\theta(L_m)},A)$ and $\tilde{\mathcal{E}}_{x^{(2)}}(\varphi_{\theta(L_m)},A)$ for the reconstruction training~\textit{(bottom right)} for $L_m = 1-\nicefrac{1}{3^m}$ with $m=1,\hdots,5$, $\tilde{A}$ the Radon operator on the MNIST test dataset.  
    }
\end{figure}

\begin{figure}[h]
\centering
\begin{subfigure}{\textwidth}
\centering
\begin{tikzpicture}
\pgfplotsset{
every axis plot/.append style={line width=0.8pt},
xmode = log,
ymode = log,
log y ticks with fixed point/.style={
      yticklabel={
        \pgfkeys{/pgf/fpu=true}
        \pgfmathparse{exp(\tick)}%
        \pgfmathprintnumber[fixed relative, precision=3]{\pgfmathresult}
        \pgfkeys{/pgf/fpu=false}}    
}
}
\begin{groupplot}[
    group style={
        group name=my plots,
        group size=3 by 1,
        ylabels at=edge left,
        xlabels at=edge bottom,
        horizontal sep=2.5cm
    },
    ymode=log,
    scale only axis,
    footnotesize,
    width=3.7cm,
    height=3.7cm,
    tickpos=left,
    x tick label style={/pgf/number format/fixed,
    /pgf/number format/1000 sep = \thinspace},
    ytick align=outside,
    xtick align=outside,
    enlarge x limits=false,
    xmax=0.33,
    grid = both,
    xtick={0,0.0001,0.001,0.01,0.1,1},
    anchor = north,
    clip=true,
    cycle list name=color list
]
\nextgroupplot[ ylabel={$\mathrm{MSE}_\mathrm{reco}^{\delta\ell}(\varphi_{\theta(L,\delta_\ell)},A)$}, ymax = 0.7, ytick ={0.001,0.01,0.1,1}, yticklabels={$10^{-3}$,$ 10^{-2}$,$10^{-1}$}
 ]
\addplot table [x=noise, y=1, col sep=comma]{csv_files/part2_MSEoverNoise_approx_reco_noise=True.csv};
\addplot table [x=noise, y=2, col sep=comma]{csv_files/part2_MSEoverNoise_approx_reco_noise=True.csv};
\addplot table [x=noise, y=3, col sep=comma]{csv_files/part2_MSEoverNoise_approx_reco_noise=True.csv};
\addplot table [x=noise, y=4, col sep=comma]{csv_files/part2_MSEoverNoise_approx_reco_noise=True.csv};
\addplot table [x=noise, y=5, col sep=comma]{csv_files/part2_MSEoverNoise_approx_reco_noise=True.csv};

  \nextgroupplot[
   legend pos=outer north east,
    legend style={draw=none},
    legend cell align={left},
  xshift=-1.3cm,
  ymax = 0.7, ytick ={0.001,0.01,0.1,1}, yticklabels={$10^{-3}$,$ 10^{-2}$,$10^{-1}$}
  ]
  \addplot table [x=noise, y=1, col sep=comma]{csv_files/part2_MSEoverNoise_reco_reco_noise=True.csv};
  \addlegendentry{$L_1$}
 \addplot table [x=noise, y=2, col sep=comma]{csv_files/part2_MSEoverNoise_reco_reco_noise=True.csv};
\addlegendentry{$L_2 $}
\addplot table [x=noise, y=3, col sep=comma]{csv_files/part2_MSEoverNoise_reco_reco_noise=True.csv};
\addlegendentry{$L_3 $}
\addplot table [x=noise, y=4, col sep=comma]{csv_files/part2_MSEoverNoise_reco_reco_noise=True.csv};
\addlegendentry{$L_4$}
 \addplot table [x=noise, y=5, col sep=comma]{csv_files/part2_MSEoverNoise_reco_reco_noise=True.csv};
\addlegendentry{$L_5 $}

\nextgroupplot[
     legend pos= north west,
    legend cell align={left},
  ymax = 0.7, ytick ={0.001,0.01,0.1,1}, yticklabels={$10^{-3}$,$ 10^{-2}$,$10^{-1}$}
  ]
\addplot table [x=noise, y=min_mean, col sep=comma]{csv_files/part2_MSEoverNoise_approx_reco_noise=True.csv};
\addlegendentry{approximation}
 \addplot table [x=noise, y=min_mean, col sep=comma]{csv_files/part2_MSEoverNoise_reco_reco_noise=True.csv};
\addlegendentry{reconstruction}

\end{groupplot}
\end{tikzpicture}

\end{subfigure}

\begin{subfigure}{\textwidth}
\centering
\begin{tikzpicture}
\pgfplotsset{
every axis plot/.append style={line width=0.8pt},
xmode = log,
log y ticks with fixed point/.style={
      yticklabel={
        \pgfkeys{/pgf/fpu=true}
        \pgfmathparse{exp(\tick)}%
        \pgfmathprintnumber[fixed relative, precision=3]{\pgfmathresult}
        \pgfkeys{/pgf/fpu=false}}    
}
}
\begin{groupplot}[
    group style={
        group name=my plots,
        group size=3 by 1,
        ylabels at=edge left,
        xlabels at=edge bottom,
        horizontal sep=2.5cm
    },
    scale only axis,
    footnotesize,
    width=3.7cm,
    height=3.7cm,
    tickpos=left,
    x tick label style={/pgf/number format/fixed,
    /pgf/number format/1000 sep = \thinspace},
    ytick align=outside,
    xtick align=outside,
    enlarge x limits=false,
    xmax=0.33,
    xlabel={$\hat{\delta}_\ell$},
    grid = both,
    xtick={0,0.0001,0.001,0.01,0.1,1},
    anchor = north,
    clip=true,
    cycle list name=color list
]
\nextgroupplot[ ylabel={$\mathrm{SSIM}^{\delta_\ell}(\varphi_{\theta(L,\delta_\ell)},A)$}, ymin=0.15, ymax = 1, ytick ={0.2,0.4,0.6,0.8,1}, yticklabels={$0.2$,$ 0.4$,$0.6$,$0.8$,$1$}
 ]
\addplot table [x=noise, y=1, col sep=comma]{csv_files/radon_SSIM_MNIST_approx_noise=True.csv};
\addplot table [x=noise, y=2, col sep=comma]{csv_files/radon_SSIM_MNIST_approx_noise=True.csv};
\addplot table [x=noise, y=3, col sep=comma]{csv_files/radon_SSIM_MNIST_approx_noise=True.csv};
\addplot table [x=noise, y=4, col sep=comma]{csv_files/radon_SSIM_MNIST_approx_noise=True.csv};
\addplot table [x=noise, y=5, col sep=comma]{csv_files/radon_SSIM_MNIST_approx_noise=True.csv};

  \nextgroupplot[   
    legend pos=outer north east,
    legend style={draw=none},
    legend cell align={left},
  xshift=-1.3cm,
  ymin=0.2, ymax = 1, ytick ={0.2,0.4,0.6,0.8,1}, yticklabels={$0.2$,$ 0.4$,$0.6$,$0.8$,$1$}
  ]
  \addplot table [x=noise, y=1, col sep=comma]{csv_files/radon_SSIM_MNIST_reco_noise=True.csv};
  \addlegendentry{$L_1$}
 \addplot table [x=noise, y=2, col sep=comma]{csv_files/radon_SSIM_MNIST_reco_noise=True.csv};
\addlegendentry{$L_2 $}
\addplot table [x=noise, y=3, col sep=comma]{csv_files/radon_SSIM_MNIST_reco_noise=True.csv};
\addlegendentry{$L_3 $}
\addplot table [x=noise, y=4, col sep=comma]{csv_files/radon_SSIM_MNIST_reco_noise=True.csv};
\addlegendentry{$L_4$}
 \addplot table [x=noise, y=5, col sep=comma]{csv_files/radon_SSIM_MNIST_reco_noise=True.csv};
\addlegendentry{$L_5 $}

\nextgroupplot[ 
    legend pos= south west,
    legend cell align={left},
 ymin=0.2, ymax = 1, ytick ={0.2,0.4,0.6,0.8,1}, yticklabels={$0.2$,$ 0.4$,$0.6$,$0.8$,$1$}
  ]
\addplot table [x=noise, y=max_mean, col sep=comma]{csv_files/radon_SSIM_MNIST_approx_noise=True.csv};
\addlegendentry{approximation}
 \addplot table [x=noise, y=max_mean, col sep=comma]{csv_files/radon_SSIM_MNIST_reco_noise=True.csv};
\addlegendentry{reconstruction}

\end{groupplot}
\end{tikzpicture}
\end{subfigure}
\caption{Reconstruction errors $\mathrm{MSE}_\mathrm{reco}^{\delta_\ell}(\varphi_{\theta(L,\delta_\ell)},A)$~\textit{(top row)} and $\mathrm{SSIM}^{\delta_\ell}(\varphi_{\theta(L,\delta_\ell)},A)$~\textit{(bottom row)} for networks trained on noisy samples with noise levels $\delta_\ell$ for $\ell = 0,\hdots,6$ and reconstructions from noisy samples of the same noise level for the approximation training~\textit{(left)} and for the reconstruction training~\textit{(middle)} with Lipschitz bounds $L_m$ on the MNIST dataset for the Radon operator $\tilde{A}$. Outcomes of optimal parameter choices for both training strategies over different noise levels can be seen on the right-hand side.
}
\end{figure}
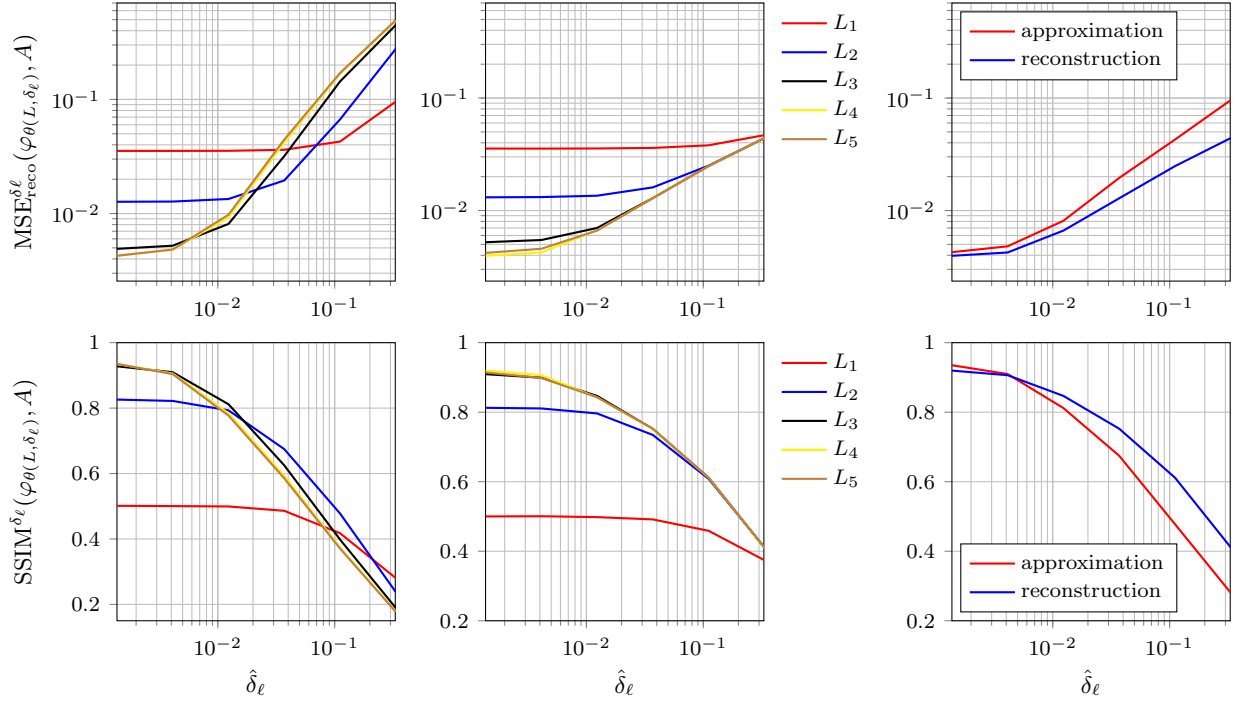

\begin{figure}[h]
\centering
\begin{tikzpicture}
\pgfplotsset{
every axis plot/.append style={line width=0.8pt},
xmode = log,
ymode = log,
log y ticks with fixed point/.style={
      yticklabel={
        \pgfkeys{/pgf/fpu=true}
        \pgfmathparse{exp(\tick)}%
        \pgfmathprintnumber[fixed relative, precision=3]{\pgfmathresult}
        \pgfkeys{/pgf/fpu=false}}    
}
}
\begin{groupplot}[
    group style={
        group name=my plots,
        group size=2 by 1,
        ylabels at=edge left,
        xlabels at=edge bottom,
        horizontal sep=1.8cm
    },
    scale only axis,
    footnotesize,
    width=5.3cm,
    height=4.2cm,
    tickpos=left,
    x tick label style={/pgf/number format/fixed,
    /pgf/number format/1000 sep = \thinspace},
    ytick align=outside,
    xtick align=outside,
    enlarge x limits=false,
    xmax=0.33,
    xlabel={$\hat{\delta}_\ell$},
    grid = both,
    xtick={0,0.0001,0.001,0.01,0.1,1},
    anchor = north,
    clip=true,
    cycle list name=color list
]
\nextgroupplot[
ymin=0.65, ymax = 1.07, ytick ={0.7,0.8,0.9,1.0}, yticklabels={$0.7$,$0.8$,$0.9$,$1.0$}
 ]
\addplot table [x=noise, y=1, col sep=comma]{csv_files/radon_trueL_MNIST_approx_noise=True.csv};
\addplot table [x=noise, y=2, col sep=comma]{csv_files/radon_trueL_MNIST_approx_noise=True.csv};
\addplot table [x=noise, y=3, col sep=comma]{csv_files/radon_trueL_MNIST_approx_noise=True.csv};
\addplot table [x=noise, y=4, col sep=comma]{csv_files/radon_trueL_MNIST_approx_noise=True.csv};
\addplot table [x=noise, y=5, col sep=comma]{csv_files/radon_trueL_MNIST_approx_noise=True.csv};
\addplot[color=red, dashed, domain=0.009:1]{0.667};
\addplot[color=blue, dashed, domain=0.009:1]{0.889};
\addplot[color=black, dashed, domain=0.009:1]{0.963};
\addplot[color=yellow, dashed, domain=0.009:1]{0.988};
\addplot[color=brown, dashed, domain=0.009:1]{0.996};

\nextgroupplot[ 
ymin=0.65, ymax = 1.07, ytick ={0.7,0.8,0.9,1.0}, yticklabels={$0.7$,$0.8$,$0.9$,$1.0$},legend pos=outer north east,
    legend style={draw=none},
    legend columns=2,legend style={
            /tikz/column 2/.style={
                column sep=5pt,
            },},
    legend cell align={left}
 ]
\addplot[red] table [x=noise, y=1, col sep=comma]{csv_files/radon_trueL_MNIST_reco_noise=True.csv};
\addlegendentry{$\tilde{L}_1$}
\addplot[color=red, dashed, domain=0.009:1]{0.667};
\addlegendentry{$L_1$}
\addplot[blue] table [x=noise, y=2, col sep=comma]{csv_files/radon_trueL_MNIST_reco_noise=True.csv};
\addlegendentry{$\tilde{L}_2$}
\addplot[color=blue, dashed, domain=0.009:1]{0.889};
\addlegendentry{$L_2$}
\addplot[black] table [x=noise, y=3, col sep=comma]{csv_files/radon_trueL_MNIST_reco_noise=True.csv};
\addlegendentry{$\tilde{L}_3$}
\addplot[color=black, dashed, domain=0.009:1]{0.963};
\addlegendentry{$L_3$}
\addplot[yellow] table [x=noise, y=4, col sep=comma]{csv_files/radon_trueL_MNIST_reco_noise=True.csv};
\addlegendentry{$\tilde{L}_4$}
\addplot[color=yellow, dashed, domain=0.009:1]{0.988};
\addlegendentry{$L_4$}
\addplot[brown] table [color=brown, x=noise, y=5, col sep=comma]{csv_files/radon_trueL_MNIST_reco_noise=True.csv};
\addlegendentry{$\tilde{L}_5$}
\addplot[color=brown,dashed, domain=0.009:1]{0.996};
\addlegendentry{$L_5$}
\end{groupplot}
\end{tikzpicture}
\caption{Lipschitz constraints $L_m$ of networks $\varphi_{\theta(L_m,\delta_\ell)}$ together with respective computed Lipschitz constants $\tilde{L}_m$ of the trained residual functions $f_{\theta(L_m,\delta_\ell)}$  at $m=1,\hdots,5$ for noise levels $\delta_\ell$ with $\ell = 0,\hdots,6$ via approximation training~\textit{(left)} and reconstruction training~\textit{(right)} on the MNIST dataset for $\tilde{A}$ the Radon operator.
}
\end{figure}
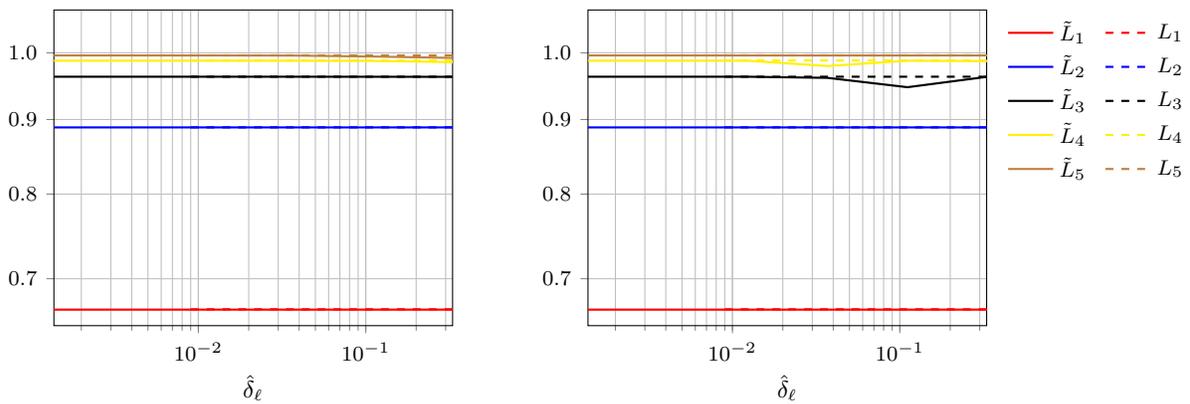

\begin{figure}
    \centering
    \includegraphics[width=.97\textwidth]{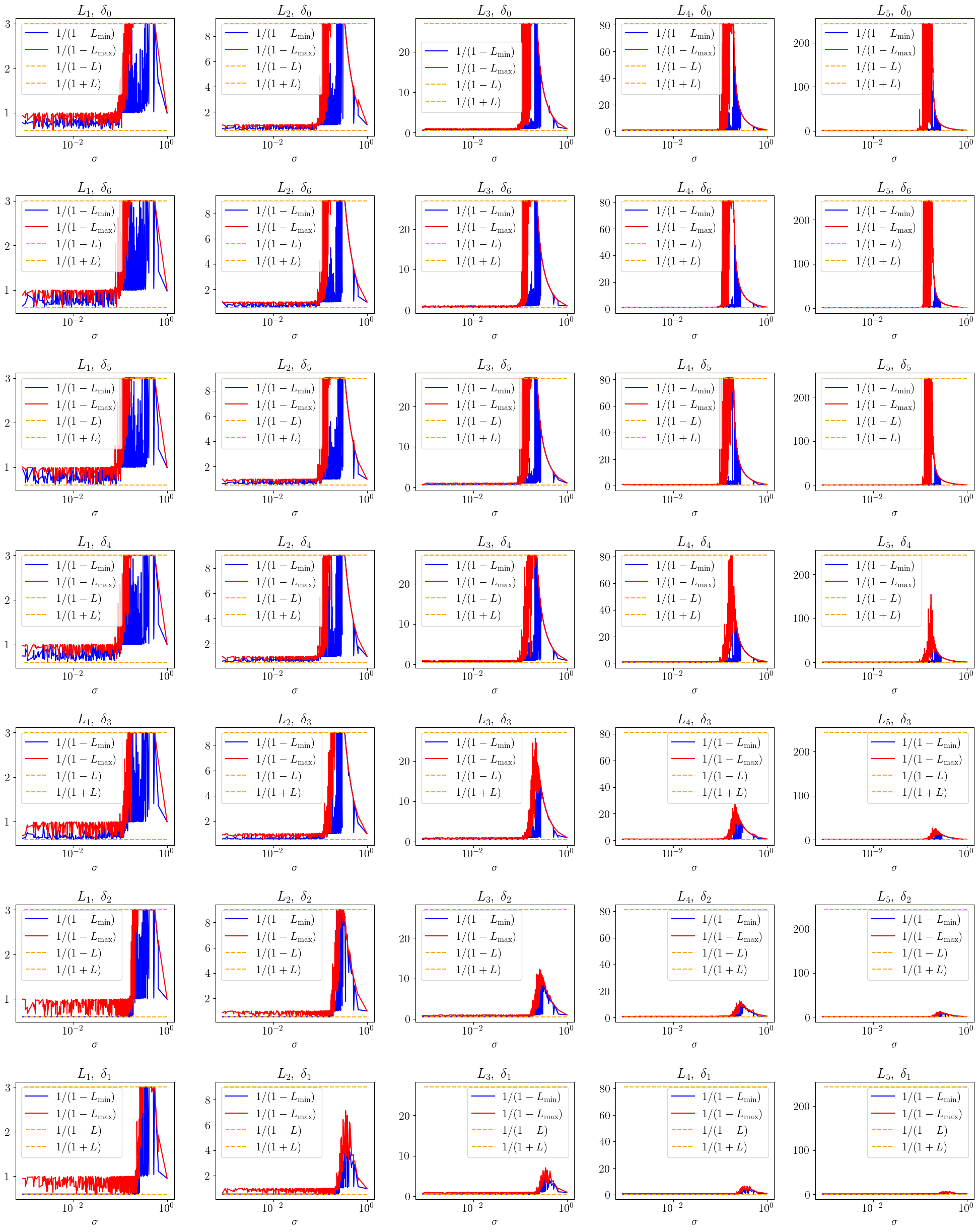}
    \caption{$\frac{1}{1-L_\mathrm{min}}$~\textit{(blue)} for minimal slopes $L_\mathrm{min}$ estimated for $f_{\theta(L_m,\delta_\ell)}$ and $\frac{1}{1-L_\mathrm{max}}$~\textit{(red)} for maximum slope $L_\mathrm{max}$ of $f_{\theta(L_m,\delta_\ell)}$ over singular values of Radon operator $\tilde{A}$ of trained networks with Lipschitz constraint $L=L_m$ at $m=1,\hdots,5$ and noise levels $\delta_\ell$ with $\ell=0,\hdots,6$ via reconstruction training on the MNIST dataset.}
\end{figure}
\FloatBarrier

\bibliographystyle{abbrv}
\bibliography{literature}

\end{document}